\documentclass[a4paper,10pt,reqno]{article}

\usepackage[utf8]{inputenc}
\usepackage[T1]{fontenc}
\usepackage{lmodern}
\usepackage[english]{babel}
\usepackage{microtype}

\usepackage{amsmath,amssymb,amsfonts,amsthm,esint}
\usepackage{mathtools,accents}
\usepackage{mathrsfs}
\usepackage{aliascnt}
\usepackage{braket}
\usepackage{bm}

\usepackage[a4paper,margin=3cm]{geometry}
\usepackage[citecolor=blue,colorlinks]{hyperref}

\usepackage{enumerate}
\usepackage{xcolor}

\makeatletter
\g@addto@macro\@floatboxreset\centering
\makeatother

\makeatletter
\def\newaliasedtheorem#1[#2]#3{
	\newaliascnt{#1@alt}{#2}
	\newtheorem{#1}[#1@alt]{#3}
	\expandafter\newcommand\csname #1@altname\endcsname{#3}
}
\makeatother

\numberwithin{equation}{section}

\newtheoremstyle{slanted}{\topsep}{\topsep}{\slshape}{}{\bfseries}{.}{.5em}{}

\theoremstyle{plain}
\newtheorem{theorem}{Theorem}[section]
\newaliasedtheorem{proposition}[theorem]{Proposition}
\newaliasedtheorem{lemma}[theorem]{Lemma}
\newaliasedtheorem{corollary}[theorem]{Corollary}
\newaliasedtheorem{counterexample}[theorem]{Counterexample}

\theoremstyle{definition}
\newaliasedtheorem{definition}[theorem]{Definition}
\newaliasedtheorem{question}[theorem]{Question}
\newaliasedtheorem{openquestion}[theorem]{Open Question}
\newaliasedtheorem{conjecture}[theorem]{Conjecture}

\theoremstyle{remark}
\newaliasedtheorem{remark}[theorem]{Remark}
\newaliasedtheorem{example}[theorem]{Example}

\newcommand{\setN}{\mathbb{N}}

\newcommand{\setR}{\mathbb{R}}

\newcommand{\eps}{\varepsilon}

\let\altphi\phi
\let\phi\varphi
\let\varphi\altphi
\let\altphi\undefined

\newcommand{\abs}[1]{\left\lvert#1\right\rvert}
\newcommand{\norm}[1]{\left\lVert#1\right\rVert}

\newcommand{\weakto}{\rightharpoonup}

\let\div\undefined
\DeclareMathOperator{\div}{div}
\DeclareMathOperator{\Hess}{Hess}
\DeclareMathOperator{\sign}{sign}

\newcommand{\di}{\mathop{}\!\mathrm{d}}

\newcommand{\bs}{{\rm bs}}
\newcommand{\loc}{{\rm loc}}
\newcommand{\sym}{\rm sym}

\newcommand{\diam}{{\rm diam}}
\newcommand{\res}{\mathop{\hbox{\vrule height 7pt width .5pt depth 0pt
			\vrule height .5pt width 6pt depth 0pt}}\nolimits}

\DeclareMathOperator{\supp}{supp}

\newcommand{\Ch}{{\sf Ch}}

\DeclareMathOperator{\Lip}{Lip}
\DeclareMathOperator{\Lipb}{Lip_b}
\DeclareMathOperator{\Lipbs}{Lip_\bs}
\DeclareMathOperator{\Cb}{C_b}

\DeclareMathOperator{\Cbs}{C_\bs}

\DeclareMathOperator{\lip}{lip} 
\DeclareMathOperator{\BV}{BV}
\DeclareMathOperator{\Per}{Per}
\DeclareMathOperator{\Ric}{Ric}

\newcommand{\haus}{\mathscr{H}}
\newcommand{\Leb}{\mathscr{L}}
\newcommand{\Meas}{\mathscr{M}}
\newcommand{\Prob}{\mathscr{P}}
\newcommand{\Borel}{\mathscr{B}}

\newcommand{\XX}{{\boldsymbol{X}}}

\newcommand{\dist}{\mathsf{d}}

\newcommand{\meas}{\mathfrak{m}}

\newcommand{\Test}{{\rm Test}}

\DeclareMathOperator{\CD}{CD}
\DeclareMathOperator{\RCD}{RCD}
\DeclareMathOperator{\ncRCD}{ncRCD}
\newfont{\tmpf}{cmsy10 scaled 2500}

\DeclareMathOperator{\Tan}{Tan}

\def\Xint#1{\mathchoice
	{\XXint\displaystyle\textstyle{#1}}%
	{\XXint\textstyle\scriptstyle{#1}}%
	{\XXint\scriptstyle\scriptscriptstyle{#1}}%
	{\XXint\scriptscriptstyle\scriptscriptstyle{#1}}%
	\!\int}
\def\XXint#1#2#3{{\setbox0=\hbox{$#1{#2#3}{\int}$ }
		\vcenter{\hbox{$#2#3$ }}\kern-.6\wd0}}

\def\dashint{\Xint-}

\begin{document}
	
\title{Rigidity of the $1$-Bakry-\'Emery inequality and \\sets of finite perimeter in $\RCD$ spaces}
\author{Luigi Ambrosio
\thanks{Scuola Normale Superiore, \url{luigi.ambrosio@sns.it}.}\and
Elia Bru\'e 
		\thanks{Scuola Normale Superiore, \url{elia.brue@sns.it}.}\and Daniele Semola
		\thanks{Scuola Normale Superiore,
			\url{daniele.semola@sns.it}.}}

\maketitle

\begin{abstract}
This note is dedicated to the study of the asymptotic behaviour of sets of finite perimeter over $\RCD(K,N)$ metric measure spaces. Our main result asserts existence of a Euclidean tangent half-space almost everywhere with respect to the perimeter measure and it can be improved to an existence and uniqueness statement when the ambient is \textit{non collapsed}.\\ 
As an intermediate tool, we provide a complete characterization of the class of $\RCD(0,N)$ spaces for which there exists a nontrivial
function satisfying the equality in the $1$-Bakry-\'Emery inequality. This result is of independent interest and it is new, up to our knowledge, 
even in the smooth framework. 
\end{abstract}

\tableofcontents

\section*{Introduction}

In recent years the theory of metric measure spaces $(X,\dist,\meas)$ (in short, m.m.s.) satisfying, in a weak sense, upper bounds on dimension and lower bounds
on the Ricci tensor has undergone fast and remarkable developments, see \cite{Ambrosio18} for a recent survey on this topic.
In this paper we focus on the $\CD(K,N)$ theory pioneered by Lott-Villani and Sturm. More specifically, we are concerned  with
the ``Riemannian'' side of the theory, the class of $\RCD(K,N)$ (with $K\in\setR$,
$1\leq N<+\infty$) m.m.s. proposed in \cite{Gigli15} after the study of their infinite-dimensional counterparts $\RCD(K,\infty)$ in
\cite{AmbrosioGigliSavare14}. The first results providing the equivalence between $\RCD(K,\infty)$ and the
validity of Bochner's inequality were obtained in \cite{AmbrosioGigliSavare15}, then \cite{ErbarKuwadaSturm15} estabilished the equivalence with the
dimensional Bochner inequality for the so-called class $\RCD^*(K,N)$. Eventually, full equivalence between $\RCD^*(K,N)$ and $\RCD(K,N)$ has been achieved, among other things, in \cite{Cavalletti-Milman}), closing the circle. By now the class of $\RCD(K,N)$ spaces is known to
include the so-called Ricci limit spaces of the Cheeger-Colding theory, Alexandrov spaces \cite{Petrunin} and other singular spaces.

Many structural properties of $\RCD(K,N)$ spaces have been established in the recent years, in most cases after their
proof in the setting of Ricci limit spaces, and sometimes with essentially new strategies of proof, given the absence of a smooth
approximation. Among these structural results it is worth mentioning the splitting theorem \cite{Gigli13}, the second-order calculus \cite{Gigli18},
the existence of an ``essential dimension'' \cite{BrueSemola18}, the rectifiability as metric measure spaces \cite{MondinoNaber14,KellMondino18,DePhlippisMarcheseRindler17,GigliPasqualetto16a}, 
the validity of sharp heat kernel bounds \cite{JangLiZhang}, the well-posedness of ODE's associated to Sobolev vector fields 
\cite{AmbrosioTrevisan14}.

At this stage of the development of the $\RCD$ theory, it is natural to investigate the typical themes of Geometric Measure Theory, 
since GMT provides techniques for dealing with nonsmooth objects already when the ambient space is smooth.
One of the most fundamental results of GMT, that eventually led to the Federer-Fleming theory of currents \cite{FedererFleming60}, is De Giorgi's structure
theorem for sets $E$ of finite perimeter. De Giorgi's theorem, established in \cite{DeGiorgi54,DeGiorgi55}, provides the representation 
$$
|D\chi_E|=\haus^{n-1}\res {\mathcal F}E
$$
of the perimeter measure $|D\chi_E|$ as the restriction of $\haus^{n-1}$ to a suitable measure-theoretic boundary ${\mathcal F}E$ 
of $E$. In addition, it provides a description of $E$ on small scales, showing that for all $x\in{\mathcal F}E$ the rescaled
set $r^{-1}(E-x)$ is close, for $r>0$ sufficiently small, to an halfspace orthogonal to $\nu_E(x)$. 

Our goal in this paper is to provide an extension of this result to the setting of $\RCD(K,N)$ m.m.s. In the study of
this problem, we realized the importance of the study of the rigidity case in the Bakry-\'Emery inequality, namely the analysis
of the implications of the condition
\begin{equation}\label{eq:intro1}
|\nabla P_t f|=P_t|\nabla f|\qquad\text{$\meas$-a.e. in $X$, for every $t\geq 0$}
\end{equation}
for some nontrivial function $f$, if the ambient space is $\RCD(0,N)$.

Since the study of \eqref{eq:intro1} has an independent interest, we first illustrate our result in this direction, then we move
back to sets of finite perimeter. Our rigidity result, stated in \autoref{maintheorem}, shows that \eqref{eq:intro1} is sufficiently
strong to imply the splitting of the m.m.s. as $Z\times\setR$, in addition with a monotonic dependence on 
$f$ on the split real variable. This result could be considered as ``dual'' to the classical splitting theorem, since the basic
assumption is not the existence of a curve with a special property (namely an entire geodesic), but rather the existence
of a function satisfying \eqref{eq:intro1}. However, our proof builds on Gigli's splitting theorem and is achieved in these
steps:
\begin{itemize}
\item[1)] by first variations in \eqref{eq:intro1} we prove that the unit vector fields $b_t=\nabla P_t f/|\nabla P_t f|$
are independent of $t$, divergence-free and with symmetric part of derivative in $L^2$, in a suitable weak sense;

\item[2)] because of this, the theory of flows developed in \cite{AmbrosioTrevisan14} applies, and provides a measure-preserving semigroup
of isometries $\XX_t$;

\item[3)] we use $\XX_t$ to show in \autoref{prop: HJ equation} that $(P_sf)\circ\XX_{-t}$ is a value function, more
precisely
$$
P_s f(\XX_{-t}(x))=\min_{\overline{B}_t(x)} P_s f\qquad\forall x\in X,\,s>0,\,t\geq 0.
$$
In the proof of this fact we have been inspired by the analysis of isotropic Hamilton-Jacobi equations made in
\cite{Nakayasu14} (see also \cite{AmbrosioFeng}), even though our proof is self-contained. Using this representation
of $(P_sf)\circ\XX_{-t}$ it is not hard to prove that all flow curves $t\mapsto\XX_t(x)$ are lines and, in particular,
Gigli's theorem \cite{Gigli13,Gigli14} applies. Even though this refinement does not play a role in the second part of the
paper, we also prove that the validity of $|\nabla P_t f|=P_t|\nabla f|$ for some $t>0$ implies 
the validity for all $t\geq 0$, namely \eqref{eq:intro1}.
\end{itemize}

Now, what is the relation between \eqref{eq:intro1} and the fine structure of sets of finite perimeter? In De Giorgi's
proof and its many extensions to currents and other complex objects, the normal direction $\nu_E$ coming out of the blow-up analysis is identified by
looking at the polar decomposition $D\chi_E=\nu_E|D\chi_E|$ of the distributional derivative (choosing approximate
continuity points of $\nu_E$, relative to $|D\chi_E|$). In turn, the polar decomposition essentially depends on the particular structure of the tangent bundle of the Euclidean space. In the $\RCD$ theory, as in Cheeger's theory of PI spaces 
(see \cite{GigliPasqualetto16b} for a deeper discussion of the relation between the two notions of tangent bundle), the tangent bundle is defined only
up to $\meas$-negligible sets, not in a pointwise sense. So, it could in principle be used to write a polar decomposition
analogous to the Euclidean theory only for vector-valued (in a suitable sense) measures absolutely continuous w.r.t. $\meas$. 
We bypass this difficulty by establishing this new principle: at $|D\chi_E|$-a.e. point $x$,
any tangent set $F$ to $E$ at $x$ in any tangent, pointed, metric measure structure $(Y,\varrho,\mu,y)$ has to satisfy the condition
\begin{equation}\label{eq:intro2}
|\nabla P_t\chi_F|\mu=P_t^*|D\chi_F|\qquad\forall t\geq 0.
\end{equation}
Notice that $|D\chi_F|$, the semigroup $P_t$ and its dual $P_t^*$ in \eqref{eq:intro2} have, of course, to be understood in the 
tangent metric measure structure. The proof of this principle, given in \autoref{thm:rigidtangent}, ultimately relies on the lower semicontinuity 
of the perimeter measure $|D\chi_E|$ (as it happens for the powerful principle that lower semicontinuity and locality imply asymptotic local minimality, 
see \cite{Fleming,White}, and \cite{Cheeger99}) and gradient contractivity.
From \eqref{eq:intro2}, gradient contractivity easily yields that all functions $f=P_s\chi_F$ satisfy \eqref{eq:intro1}; this leads
to a splitting \textit{both} of $(Y,\varrho,\mu)$ and $F$, and to the identification of a ``tangent halfspace'' $F$ to $E$ at $x$.

Using these ideas we can prove the following structure results for sets of finite perimeter $E$ in 
$\RCD(K,N)$ m.m.s.: first, in \autoref{thm:tangenthalfspace}, we prove that 
$E$ admits a Euclidean half-space as tangent at $x$ for $\abs{D\chi_E}$-a.e. $x\in X$. This result
becomes more precise in the setting of noncollapsed $\RCD(K,N)$ m.m.s. of \cite{DePhilippisGigli18}: in this
case we prove in \autoref{cor:imprnoncollapsed} and \autoref{cor:identificationdensity} that $|D\chi_E|$ is concentrated on the $N$-dimensional
regular set of the ambient metric measure structure and we provide the representation formula $|D\chi_E|=\mathcal{S}^{N-1}\res\mathcal{F}E$ (where $\mathcal S^k$ denotes the $k$-dimensional
spherical Hausdorff measure, see also \eqref{eq:bellafrontiera} 
for the precise definition of ${\mathcal F}E$), for an appropriate notion of measure-theoretic boundary ${\mathcal F}E$.  

The paper is organized as follows. After the preliminary \autoref{sec:1}, where we collect basic notation, calculus rules and basic facts about
$\RCD(K,N)$ spaces, Sobolev and $\BV$ functions and flows associated to vector fields, in \autoref{sec:rigidityBE} we prove our rigidity results,
allong the lines described above. We dedicate \autoref{sec:compstab} to the study of the behaviour of sequences of sets $E_i$ in m.m.s. 
$(X_i,\dist_i,\meas_i)$ convergent in the measured Gromov-Hausdorff sense to $(X,\dist,\meas)$. In particular, using appropriate
notions of compactness for sequences of functions and measures in varying metric measure structures, we focus on 
compactness and lower semicontinuity of the perimeter measure. We apply these results in \autoref{sec:tangentsRCD}, where we
specialize our analysis to the case when $(X_i,\dist_i,\meas_i)$ arise from the rescaling of a pointed m.m.s. This theme is also
investigated in \cite{Shanmugalingam18}, but in our paper we take advantage of the curvature/dimension bounds to establish the
stronger rigidity property \eqref{eq:intro2} satisfied by tangent sets $F$ in the tangent metric measure structure. Then, using the splitting
property and the principle that ``tangents to a tangent are tangent'', we are able to recover the above-mentioned structure results
of sets of finite perimeter. 
Finally, \autoref{sec:Appendix} is devoted to a self-contained proof of the iterated tangents principle, adapting 
the argument of Preiss' seminal paper \cite{Preiss87} (see also \cite{AmbrosioKleinerLeDonne08,GigliMondinoRajala15}).

A natural question arising from the results and methods in this paper is the rectifiability of the measure-theoretic boundary $\mathcal{F}E$. 
Our blow up-analysis shows that the measure $|D\chi_E|=\mathcal{S}^{N-1}\res\mathcal{F}E$ is 
asymptotically flat $|D\chi_E|$-a.e., but flatness should be understood only in the measured-Gromov Hausdorff sense (for instance, a measure concentrated
on a spiral would be asymptotically flat at the origin of the spiral). So, trying to get rectifiability using only this information is related
to the search of rectifiability criteria where Jones' $\beta$ numbers are replaced by smaller numbers built with the scaled and localized 
Wasserstein distance. As far as we know, despite important progress on rectifiability criteria for measures (see for 
instance \cite{Tolsa1,Tolsa2}) this question is open even for measures in the Euclidean space. 

Finally, for sets $E$ of finite perimeter in general $\RCD(K,N)$ m.m.s. $(X,\dist,\meas)$, it would be important to get additional 
information on the structure of $|D\chi_E|$, besides the existence of flat tangents $|D\chi_E|$-a.e. in $X$. 
 
{\bf Acknowledgements.}
The authors wish to thank Andrea Mondino for fruitful discussions around the topic of this note and the anonymous referee for several valuable comments that greatly improved the present note.

\section{Preliminaries}\label{sec:1}
In a metric space $(Z,\dist)$, we will denote by $B_r(x)=\{\dist(\cdot,x)<r\}$ and $\overline{B}_r(x)=\{\dist(\cdot,x)\leq r\}$ the open
and closed balls respectively, by $\Cbs(Z)$ the space of bounded continuous functions with
bounded support, by $\Lipbs(Z)\subset\Cbs(Z)$ the subspace of Lipschitz functions. We shall adopt the notation $\Cb(Z)$ and $\Lipb(Z)$ for bounded continuous and bounded Lipschitz functions respectively. For any $f\in\Lip(Z)$ we shall denote by $\Lip f$ its global Lipschitz constant. The characteristic function with values in $\set{0,1}$ of a set $E\subset Z$ will be denoted by $\chi_E$.

The Borel $\sigma$-algebra of a metric space $(Z,\dist)$ is denoted by $\Borel(Z)$. We shall denote by $\Meas(Z)$ the space of signed Borel measures with finite total variation on $Z$ and by $\Meas^+(Z),\Meas^+_{\loc}(Z),\Prob(Z)$ the spaces of nonnegative finite Borel measures, nonnegative measures finite on bounded subsets of $Z$ and probability measures, respectively. We will denote by $\supp\meas$ the support of any $\meas\in\Meas^+_{\loc}(Z)$.\\
We will use the standard notation $L^p(Z,\meas)$, $L^p_{\loc}(Z,\meas)$ for the $L^p$ spaces, whenever $\meas$ is nonnegative, and $\Leb^n,\mathcal{H}^n$ for the $n$-dimensional Lebesgue measure on $\setR^n$ and the $n$-dimensional Hausdorff measure on a metric space, respectively. We shall denote by $\omega_n$ the Lebesgue measure of the unit ball in $\setR^n$. 

Below we list two useful lemmas. The proof of the first one, based on Cavalieri's formula, can be found for instance in \cite[Lemma 3.3]{AmbrosioGigliSavare15} (notice that since we are assuming that $\mu$ and all $\mu_n$ are probability measures, weak convergence in duality
w.r.t $\Cbs(Z)$ and w.r.t. $\Cb(Z)$ are equivalent).

\begin{lemma}\label{lemma:uppersc}
	Let $(Z,\dist_Z)$ be a complete and separable metric space. Let $(\mu_n)\subset\Prob(Z)$ be weakly converging in duality with $\Cbs(Z)$
	to $\mu\in\Prob(Z)$ and let $f_n$ be uniformly bounded Borel functions such that
	\begin{equation}\label{eq:Gammalimsup}
	\limsup_{n\to\infty}f_n(z_n)\le f(z)\quad\text{whenever $\supp\mu_n\ni z_n\to z\in\supp\mu$},
	\end{equation}
	for some Borel function $f$. Then
	\begin{equation*}
	\limsup_{n\to\infty}\int_Zf_n\di\mu_n\le\int_Zf\di\mu.
	\end{equation*}
\end{lemma}

\begin{remark}\label{rm:droppingbounds}
	If  $(Z,\dist_Z)$ is proper, $f_n$ and $f$ are continuous, and $\mu_n$ have uniformly bounded supports, then the uniform bound 
	from above for $f_n$ is a direct consequence of \eqref{eq:Gammalimsup}.
\end{remark}

The proof of \autoref{lemma:uppersc} can be easily adapted to the case when we need to estimate the liminf of $\int_Z f_n\di\mu_n$.

\begin{lemma}\label{lemma:lowersc}
	Let $(Z,\dist_Z)$ be a complete and separable metric space. Let $(\mu_n)$ be a sequence of nonnegative Borel measures on $Z$ finite on bounded sets and assume that $\mu_n$ weakly converge to $\mu$ in duality w.r.t. $\Cbs(Z)$. Let $(f_n)$ and $f$ be nonnegative Borel functions on $Z$ such that
	\begin{equation}\label{eq:Gammaliminf}
	f(z)\le\liminf_{n\to\infty}f_n(z_n)\quad\text{whenever $\supp\mu_n\ni z_n\to z\in\supp\mu$}.
	\end{equation}
	Then 
	\begin{equation*}
	\int f\di\mu\le\liminf_{n\to\infty}\int f_n\di\mu_n.
	\end{equation*}
\end{lemma}

\subsection{Calculus tools}

Throughout this paper a \textit{metric measure space} is a triple $(X,\dist,\meas)$, where $(X,\dist)$ is a locally compact and separable metric space (even though some intermediate results do not need the local compactness assumption) and $\meas$ is a nonnegative Borel measure on $X$ finite on bounded sets. We shall adopt the notation $(X,\dist,\meas,x)$ for \textit{pointed metric measure spaces}, that is metric measure spaces $(X,\dist,\meas)$ with a fixed reference point $x\in X$.  

The Cheeger energy $\Ch:L^2(X,\meas)\to[0,+\infty]$ associated to a m.m.s. $(X,\dist,\meas)$ is the convex and lower semicontinuous functional defined through
\begin{equation}\label{eq:cheeger}
\Ch(f):=\inf\left\lbrace\liminf_{n\to\infty}\int\lip^2f_n\di\meas:\quad f_n\in\Lipb(X)\cap L^2(X,\meas),\ \norm{f_n-f}_2\to 0 \right\rbrace, 
\end{equation}
where $\lip f$ is the so called slope
\begin{equation*}
\lip f(x):=\limsup_{y\to x}\frac{\abs{f(x)-f(y)}}{\dist(x,y)}.
\end{equation*}
The finiteness domain of the Cheeger energy will be denoted by $H^{1,2}(X,\dist,\meas)$. We shall denote by $H^{1,2}_{\loc}(X,\dist,\meas)$ the space of those functions $f$ such that $\eta f\in H^{1,2}(X,\dist,\meas)$ for any $\eta\in\Lipbs(X,\dist)$.\\
Looking at the optimal approximating sequence in \eqref{eq:cheeger}, it is possible to identify a canonical object $\abs{\nabla f}$, called minimal relaxed slope, providing the integral representation
\begin{equation*}
\Ch(f):=\int_X\abs{\nabla f}^2\di\meas\qquad\forall f\in H^{1,2}(X,\dist,\meas).
\end{equation*}
Any metric measure space such that $\Ch$ is a quadratic form is said to be infinitesimally Hilbertian and from now on we shall always make this assumption, unless otherwise stated. Let us recall from \cite{AmbrosioGigliSavare14,Gigli15} that, under this assumption, the function 
\begin{equation*}
\nabla f_1\cdot\nabla f_2:=\lim_{\eps\to 0}\frac{\abs{\nabla(f_1+\eps f_2)}^2-\abs{\nabla f_1}^2}{2\eps}
\end{equation*}
defines a symmetric bilinear form on $H^{1,2}(X,\dist,\meas)\times H^{1,2}(X,\dist,\meas)$ with values into $L^1(X,\meas)$.\\

It is possible to define a Laplacian operator $\Delta:\mathcal{D}(\Delta)\subset L^{2}(X,\meas)\to L^2(X,\meas)$ in the following way. We let $\mathcal{D}(\Delta)$ be the set of those $f\in H^{1,2}(X,\dist,\meas)$ such that, for some $h\in L^2(X,\meas)$, one has
\begin{equation}\label{eq:amb1}
\int_X\nabla f\cdot\nabla g\di\meas=-\int_X hg\di\meas\qquad\forall g\in H^{1,2}(X,\dist,\meas) 
\end{equation} 
and in that case we put $\Delta f=h$ since $h$ is uniquely determined by \eqref{eq:amb1}. 
It is easy to check that the definition is well-posed and that the Laplacian is linear (because $\Ch$ is a quadratic form).\\
The heat flow $P_t$ is defined as the $L^2(X,\meas)$-gradient flow of $\frac{1}{2}\Ch$. Its existence and uniqueness
follow from the Komura-Brezis theory. It can equivalently be characterized by saying that for any
$u\in L^2(X,\meas)$ the curve $t\mapsto P_tu\in L^2(X,\meas)$ is locally absolutely continuous in $(0,+\infty)$ and satisfies
\begin{equation}\label{heat equation}
\frac{\di}{\di t}P_tu=\Delta P_tu \quad\text{for }\Leb^1\text{-a.e. }t\in(0,+\infty),\qquad
\lim_{t\downarrow 0}P_tu=u\quad\text{in $L^2(X,\meas)$.}
\end{equation}
Under our assumptions the heat flow provides a linear, continuous and self-adjoint contraction
semigroup in $L^2(X,\meas)$. Moreover $P_t$ extends to a linear, continuous and mass preserving operator,
still denoted by $P_t$ , in all the $L^p$ spaces for $1\le p < +\infty$. 

We recall the following regularization properties of $P_t$, ensured by the theory of gradient flows
and maximal monotone operators (even without the infinitesimal Hilbertian assumption):
\begin{equation}\label{eq:regularizingheat}
\norm{P_tf}_{L^{2}(X,\meas)}\le \norm{f}_{L^2(X,\meas)},\qquad\Ch(P_tf)\le\frac{\norm{f}^2_{L^2(X,\meas)}}{2t}\quad\text{and}\quad\norm{\Delta P_t f}_{L^2(X,\meas)}\le\frac{\norm{f}_{L^2(X,\meas)}}{t},
\end{equation}
for any $t>0$ and for any $f\in L^2(X,\meas)$.

Let us introduce now vector fields over $(X,\dist,\meas)$ as derivations over an algebra of test functions, following the approach introduced in \cite{Weaver} and adopted in \cite{AmbrosioTrevisan14} (see also \cite{Gigli13}).
	
	\begin{definition}\label{def:derivation}
		We say that a linear functional $b:\Lip(X)\to L^0(X,\meas)$ is a derivation if it satisfies the Leibniz rule, that is
		\begin{equation}\label{eq:Leibnizrule}
		b(fg)=b(f)g+fb(g),
		\end{equation}
		for any $f,\,g\in\Lip(X)$.
		
		Given a derivation $b$ and $p\in [1,+\infty]$, we write $b \in L^p(TX)$ (resp. $L^p_{\loc}(TX)$) if there exists $g\in L^p(X,\meas)$ (resp. $L^p_{\loc}(X,\meas)$) such that
		\begin{equation}\label{eq:continuityofderivation}
		|b(f)|\le g\lip(f)\quad\text{$\meas$-a.e. on $X$,}
		\end{equation}
		for any $f\in \Lip(X)$ and we denote by $\abs{b}$ the minimal (in the $\meas$-a.e. sense) $g$ with such property. 
		We also say that $b$ has compact support if $|b|$ has compact support.
	\end{definition}

	Let us remark that any $f\in H^{1,2}(X,\dist,\meas)$ defines in a canonical way a derivation $b_f\in L^2(TX)$ through the formula $b_f(g)=\nabla f\cdot\nabla g$, usually called the \textit{gradient derivation} associated to $f$. We will use the notation $b\cdot\nabla f$ in place of $b(f)$ in the rest of the paper. 
	
	A notion of divergence can be introduced via integration by parts.
	
	\begin{definition}\label{def:divergence}
		Let $b$ be a derivation in $L^1_{\loc}(TX)$ and $p\in [1,+\infty]$. 
		We say that $\div b\in L^p(X,\meas)$ if there exists $g\in L^p(X,\meas)$ such that
		\begin{equation}\label{eq:divergence}
		\int_X b\cdot\nabla f\di\meas=-\int_Xgf\di\meas\qquad\text{for any $f\in \Lipbs(X)$.}
		\end{equation}
		 By a density argument it is easy to check that such a $g$ is unique (when it exists) and we will denote it by $\div b$.
	\end{definition}
	We refer to \cite{Gigli18} for the introduction of the so-called tangent and cotangent moduli over an arbitrary metric measure space and for the identification results between derivations in $L^2$ and elements of the tangent modulus $L^2(TX)$ which justify the use of this notation.

We conclude this brief subsection introducing the basic notions and results about functions of bounded variation and sets of finite perimeter over metric measure spaces. Let us remark that, for the sake of this discussion, the assumption that $\Ch$ is quadratic is unnecessary.

\begin{definition}[Function of bounded variation]\label{def:bvfunction}
A function $f\in L^1(X,\meas)$ is said to belong to the space $\BV(X,\dist,\meas)$ if there exist
locally Lipschitz functions $f_i$ converging to $f$ in $L^1(X,\meas)$ such that
\begin{equation*}
\limsup_{i\to\infty}\int_X\abs{\nabla f_i}\di \meas<+\infty.
\end{equation*}  
By localizing this construction one can define 
\begin{equation*}
\abs{Df}(A):=\inf\left\lbrace \liminf_{i\to\infty}\int_A\abs{\nabla f_i}\di \meas: f_i\in\Lip_{\loc}(A),\quad f_i\to f \text{ in } L^1(A,\meas)\right\rbrace  
\end{equation*}
for any open $A\subset X$. In \cite{AmbDiM14} (see also \cite{MirandaJr} for the case of locally compact spaces) it is proven that this set function 
is the restriction to open sets of a finite Borel measure that we call \emph{total variation of $f$} and still denote $\abs{Df}$.
\end{definition}

Dropping the global integrability condition on $f=\chi_E$, let us recall now the analogus definition of set of finite perimeter 
in a metric measure space (see again \cite{Am02,MirandaJr,AmbDiM14}).

\begin{definition}[Perimeter and sets of finite perimeter]\label{def:setoffiniteperimeter}
Given a Borel set $E\subset X$ and an open set $A$ the perimeter $\Per(E,A)$ is defined in the following way:
\begin{equation*}
\Per(E,A):=\inf\left\lbrace \liminf_{n\to\infty}\int_A\abs{\nabla u_n}\di\meas: u_n\in\Lip_{\loc}(A),\quad u_n\to\chi_E\quad \text{in } L^1_{\loc}(A,\meas)\right\rbrace .
\end{equation*}
We say that $E$ has finite perimeter if $\Per(E,X)<+\infty$. In that case it can be proved that the set function $A\mapsto\Per(E,A)$ is the restriction to open sets of a finite Borel measure $\Per(E,\cdot)$ defined by
\begin{equation*}
\Per(E,B):=\inf\left\lbrace \Per(E,A): B\subset A,\text{ } A \text{ open}\right\rbrace.
\end{equation*}
\end{definition}

Let us remark for the sake of clarity that $E\subset X$ with finite $\meas$-measure is a set of finite perimeter if and only if $\chi_E\in\BV(X,\dist,\meas)$ and that $\Per(E,\cdot)=\abs{D\chi_E}(\cdot)$. In the following we will say that $E\subset X$ is a set of locally finite perimeter if $\chi_E$ is a function of locally bounded variation, that is to say $\eta\chi_E\in\BV(X,\dist,\meas)$ for any $\eta\in \Lipbs(X,\dist)$. 

The following coarea formula for functions of bounded variation on metric measure spaces is taken from 
\cite[Proposition 4.2]{MirandaJr}, dealing with locally compact spaces and its proof works in the more general setting of metric measure spaces.
It will play a key role in the rest of the paper.

\begin{theorem}[Coarea formula]\label{thm:coarea}
Let $v\in\BV(X,\dist,\meas)$.
Then, $\{v>r\}$ has finite perimeter for $\Leb^1$-a.e. $r\in\setR$ and, for any Borel function $f:X\to[0,+\infty]$, it holds 
\begin{equation}\label{eq:coarea}
\int_X f\di\abs{Dv}=\int_{-\infty}^{+\infty}\left(\int_X f\di\Per(\{v>r\},\cdot)\right)\di r.
\end{equation}
\end{theorem}

By applying the coarea formula to the distance function we obtain immediately that, given $x\in X$, the ball $B_r(x)$ has finite perimeter for $\Leb^1$-a.e. $r>0$, and in the sequel this fact will also be used in the quantitative form provided by \eqref{eq:coarea}. We
also recall (see for instance \cite{Am01,Am02}) that sets of locally finite perimeter are an algebra, more precisely $\Per(E,B)=\Per(X\setminus E,B)$
and
$$
\Per(E\cap F,B)+\Per(E\cup F,B)=\Per(E,B)+\Per(F,B).
$$

We will need also the following localized version of the coarea formula, which is an easy consequence of \cite[Remark 4.3]{MirandaJr}.

\begin{corollary}\label{cor:loccoarea}
Let $v\in \BV(X,\dist,\meas)$ be continuous and nonnegative.
Then, for any Borel function $f:X\to[0,+\infty]$, it holds 
\begin{equation}\label{eq:loccoarea}
\int_{\{s\le v<t\}} f\di\abs{Dv}=\int_s^{t}\left(\int_X f\di\Per(\{v>r\},\cdot)\right)\di r, \qquad  0\leq s<t<+\infty.
\end{equation}
\end{corollary}

\subsection{$\RCD(K,N)$ metric measure spaces}

The main object of our investigation in this note are $\RCD(K,N)$ metric measure spaces, that is infinitesimally Hilbertian spaces satisfying a lower Ricci curvature bound and an upper dimension bound in synthetic sense according to \cite{Sturm06a,Sturm06b,LottVillani}. Before than passing to the description of the main properties of $\RCD(K,N)$ spaces that will be relevant for the sake of this note, let us briefly focus on the adimensional case.\\
The class of $\RCD(K,\infty)$  spaces was introduced in \cite{AmbrosioGigliSavare14} (see also \cite{AmbrosioGigliMondinoRajala15} for the extension to the case of $\sigma$-finite reference measures) adding to the $\CD(K,\infty)$ condition, formulated in terms of $K$-convexity properties of the logarithmic entropy over the Wasserstein space $(\Prob_{2},W_2)$, the infinitesimally Hilbertianity assumption. 
Under the $\RCD(K,\infty)$ condition it was proved that the dual heat semigroup $P_t^{*}:\Prob_2(X)\to\Prob_2(X)$, defined by
\begin{equation*}
\int_Xf\di P_t^*\mu=\int_XP_tf\di\mu\qquad\forall\mu\in\Prob_2(X),\quad\forall f \in \Lipbs(X),
\end{equation*}
is $K$-contractive w.r.t. the $W_2$-distance and, for $t>0$, maps probability measures into probability measures absolutely continuous w.r.t. $\meas$. Then, for any $t>0$, it is possible to define the \textit{heat kernel} $p_t:X\times X\to[0,+\infty)$ by 
\begin{equation}\label{eq:heat kernel}
p_t(x,\cdot)\meas=P_t^*\delta_x.
\end{equation}
We go on stating a few regularization properties of $\RCD(K,\infty)$ spaces, referring again to \cite{AmbrosioGigliSavare14,AmbrosioGigliMondinoRajala15} for a more detailed discussion and for the proofs of these results.\\
First we have the \textit{Bakry-\'Emery} contraction estimate:
\begin{equation}\label{eq:BE2}
\abs{\nabla P_tf}^2\le e^{-2Kt}P_t\abs{\nabla f}^2\quad \meas\text{-a.e.,}
\end{equation}
for any $t>0$ and for any $f\in H^{1,2}(X,\dist,\meas)$. This contraction estimate can be generalized to the whole range of exponents $1<p<+\infty$, but in this note we will mainly be concerned with the case $p=1$. In \cite{GigliHan16} it has been proved that on any proper $\RCD(K,\infty)$ m.m.s. it holds
\begin{equation}\label{eq:BE1}
\abs{DP_t f}\le e^{-Kt}P_t^*\abs{Df},
\end{equation}
for any $t>0$ and for any $f\in\BV(X,\dist,\meas)$.\\
Next we have the so called \textit{Sobolev to Lipschitz} property, stating that any $f\in H^{1,2}(X,\dist,\meas)$ such that $\abs{\nabla f}\in L^{\infty}(X,\meas)$ admits a representative $\tilde{f}$ such that $\Lip(\tilde{f})\le \norm{\abs{\nabla f}}_{L^{\infty}}$, and the $L^{\infty}-\Lip$ regularization: for any 
$f\in L^{\infty}(X,\meas)$ and $t>0$ one has $P_tf\in \Lip(X)$ with the quantitative estimate
\begin{equation}\label{eq:infinitolip}
\sqrt{2I_{2K}(t)}\Lip(P_tf)\le\norm{f}_{L^{\infty}},
\end{equation}
where $I_{2K}(t):=\int_0^{t}e^{2Kr}\di r$.\\
Eventually let us introduce the space of test functions $\Test(X,\dist,\meas)$ following \cite{Gigli18}:
\begin{equation}\label{eq:test}
\Test(X,\dist,\meas):=\{f\in D(\Delta)\cap L^{\infty}(X,\meas): \abs{\nabla f}\in L^{\infty}(X,\meas)\quad\text{and}\quad\Delta f\in H^{1,2}(X,\dist,\meas) \}.
\end{equation}
We shall denote in the sequel by $\Test_c(X,\dist,\meas)$ the space of test functions with compact support.

In the context of $\RCD(K,\infty)$ spaces it is possible to introduce a notion of flow associated to a vector field which reads as follows (see \cite{AmbrosioTrevisan14}). 			
	\begin{definition}\label{def:Regularlagrangianflow}
		Let us fix a vector field $b$. We say that a Borel map $\XX:[0,\infty)\times X\rightarrow X$ is a Regular Lagrangian flow (RLF for short) associated to $b$ if the following conditions hold true:
		\begin{itemize}
			\item [1)] $\XX(0,x)=x$ and $X(\cdot,x)\in C([0,\infty);X)$ for every $x\in X$;
			\item [2)] there exists $L\ge0$, called compressibility constant, such that
			\begin{equation*}
			\XX(t,\cdot)_{\sharp} \meas\leq L\meas,\qquad\text{for every $t\geq 0$};
			\end{equation*}
			\item [3)] for every $f\in \Test(X,\dist,\meas)$ the map $t\mapsto f(\XX(t,x))$ is locally absolutely continuous
			in $[0,\infty)$ for $\meas$-a.e. $x\in X$ and
			\begin{equation}\label{eq: RLF condition 3}
			\frac{\di}{\di t} f(\XX(t,x))= b\cdot \nabla f(\XX(t,x)) \quad \quad \text{for a.e.}\ t\in (0,\infty).	
			\end{equation}
		\end{itemize}
	\end{definition}
	The selection of ``good'' trajectories is encoded in condition 2), which is added to ensure that the RLF does not concentrate too much the reference measure $\meas$. We remark that the notion of RLF is stable under modification in a negligible set of initial conditions, but we prefer to work with a pointwise defined map in order to avoid technical issues.

It is well known that to obtain an existence and uniqueness theory for regular Lagrangian flows it is necessary to restrict to a class of sufficiently regular vector fields, even in the case of a smooth ambient space.\\
Below we introduce our working definition of Sobolev vector field with symmetric covariant derivative in $L^2$, following \cite{AmbrosioTrevisan14}.
This definition is sufficient for our purposes, and weaker than the notion introduced in \cite{Gigli18}, which corresponds to a sort
of localized version of \eqref{eq:symmetric covariant derivative} (actually, we have been unable to prove differentiability in Gigli's stronger
sense of our vector field $b_s=\nabla P_sf/|\nabla P_s f|$).

\begin{definition}\label{def:symmetric covariant derivative}
	Let $b\in L^{\infty}(TX)$ with $\div b\in L^{\infty}(X,\meas)$. We write $|\nabla_{\sym} b|\in L^2(X,\meas)$ if there exists a constant $c>0$ such that
	\begin{equation}\label{eq:symmetric covariant derivative}
		\abs{\int_X \nabla_{\sym}b(\nabla \phi,\nabla\psi)\di \meas}
		\le c \norm{\nabla \phi}_{L^4}\norm{\nabla \psi}_{L^4}
		\qquad
		\forall \phi,\,\psi \in \Test(X,\dist,\meas),
	\end{equation}
	where
	\begin{equation*}
		\int_X \nabla_{\sym}b(\nabla \phi,\nabla\psi)\di \meas:=-\frac{1}{2}\int_X \left\lbrace b\cdot \nabla \phi\ \Delta \psi+b\cdot \nabla \psi\ \Delta \phi-\div b\ \nabla \phi\cdot \nabla \psi
		\right\rbrace \di \meas.
	\end{equation*}
	We let $\norm{\nabla_{\sym}b}_{L^2}$ be the smallest $c$ in \eqref{eq:symmetric covariant derivative}. In particular we write $\nabla_{\sym} b=0$ if $\norm{\nabla_{\sym}b}_{L^2}=0$.
\end{definition}

In the next theorem we resume some general result concerning Regular Lagrangian flows that will be used in the sequel.
\begin{theorem}\label{Th: Lagrangianflows}
	Let $(X,\dist,\meas)$ be an $\RCD(K,\infty)$ space for some $K\in\setR$. Fix 
	$b\in L^{\infty}(TX)$ with $\div b\in L^{\infty}(X,\meas)$ and $\nabla_{\sym} b\in L^2(X,\meas)$. Then
	\begin{itemize}
		\item [(i)] there exists a unique regular Lagrangian flow $\XX:\setR\times X\to X$ associated to $b$\footnote{To be more precise, there exist unique Regular Lagrangian flows $\XX^+,\XX^-:[0,+\infty)\times X\to X$ associated to $b$ and $-b$ respectively and we let $\XX_t=\XX^+_t$ for $t\ge0$ and $\XX_t=\XX_{-t}^-$ for $t\le0$.} (uniqueness is understood in the following sense: if $\XX$ and $\bar\XX$ are Regular Lagrangian flows associated to $b$, then for $\meas$-a.e. $x\in X$ one has $\XX_t(x)=\bar \XX_t(x)$ for any $t\in \setR$);
		\item [(ii)] $\XX$ satisfies the semigroup property: for any $s\in\setR$ it holds that, for $\meas$-a.e. $x\in X$,
		\begin{equation}\label{eq:semigroup property}
			\XX(t,\XX(s,x))=\XX(t+s,x)
			\qquad
			\forall t\in \setR,
		\end{equation}
		and the bound
		\begin{equation}\label{w}
		 e^{-t\norm{\div b}_{L^{\infty}}}\meas \le \left(\XX_t\right)_{\sharp}\meas \le e^{t\norm{\div b}_{L^{\infty}}}\meas;
		\end{equation}
		\item[(iii)] For any $\bar u\in L^1(X,\meas)\cap L^{\infty}(X,\meas)$ there exists $u\in L^{\infty}_{\loc}(\setR;L^1(X,\meas)\cap L^{\infty}(X,\meas))$ such that $\left(\XX_t\right)_{\sharp}(u\meas)=u_t\meas$ and it solves the \textit{continuity equation}, i.e. for any $\phi\in \Test(X,\dist,\meas)$ the map $t\mapsto \int_X \phi u_t \di \meas$ is locally absolutely continuous with distributional derivative
		\begin{equation*}
		\frac{\di}{\di t}\int_X \phi u_t\di \meas =\int_X (b\cdot\nabla\phi) u_t \di \meas;
		\end{equation*}
		\item[(iv)] if $\div b=0$ and $\nabla_{\sym}b=0$ then $\XX_t$ admits a representative which is a measure-preserving isometry, i.e.
		\begin{equation*}
		\dist(\XX_t(x),\XX_t(y))=\dist(x,y)
		\qquad
		\forall x,y\in X
		\quad\text{and}\quad
		\left(\XX_t\right)_{\sharp} \meas =\meas,
		\end{equation*}
		for any $t\in \setR$. Furthermore in this case the semigroup property \eqref{eq:semigroup property} is satisfied pointwise.
		\end{itemize}
\end{theorem}
\begin{proof} (i), (ii), (iii) immediately follow from the results in \cite{AmbrosioTrevisan14} (see Theorem~8.3 together with Theorem~4.3 and Theorem~4.4). Let us prove (iv). From \eqref{w} we conclude that $\left(X_t\right)_{\sharp}\meas=\meas$ for any $t\in \setR$. Let us now take $\bar u\in L^{\infty}(X,\meas)\cap W^{1,2}(X,\dist, \meas)$ and $u$ as in (ii). Thanks to \cite[Lemma~5.8]{AmbrosioTrevisan14} we get that $P_{\alpha}u_t\in \Test(X,\dist,\meas)$ is still a solution of the continuity equation for any $\alpha\in (0,1)$. Then we can compute
\begin{align*}
	\frac{\di}{\di t} \frac{1}{2}\int_X |\nabla P_{\alpha}u_t|^2 \di \meas &=-\frac{\di}{\di t} \frac{1}{2}\int_X P_{\alpha}u_t \Delta P_{\alpha} u_t \di \meas\\
	&=-\int_X b\cdot \nabla \Delta P_{\alpha} u_t\ P_{\alpha} u_t \di \meas.
\end{align*}
Since $\div b=0$ and $\nabla_{\sym} b=0$, we deduce 
\begin{equation*}
	-\int_X b\cdot \nabla \Delta P_{\alpha} u_t\ P_{\alpha} u_t \di \meas
	=\int_X b\cdot \nabla P_{\alpha} u_t\ P_{\alpha} \Delta u_t \di \meas=0,
\end{equation*}
therefore
\begin{equation}\label{w1}
	\int_X |\nabla P_{\alpha}u_t|^2 \di \meas=\int_X |\nabla P_{\alpha}\bar u|^2 \di \meas\qquad
	\forall t\in \setR,\quad \forall \alpha\in (0,1).
\end{equation}
Taking the limit in \eqref{w1} as $\alpha\to 0$ it easily follows that $u_t\in H^{1,2}(X,\dist,\meas)$ for any $t\in \setR$ and that $\int_X |\nabla u_t|^2 \di \meas$ does not depend on $t\in \setR$. Using the identity $u_t(x)=\bar u(\XX(-t,x))$ (which can be checked using the semigroup property \eqref{eq:semigroup property} and $\bigl(\XX_t\bigr)_{\sharp}\meas=\meas$) we deduce that, for any $t\in\setR$, 
\begin{equation*}
	\Ch(\bar u\circ \XX_t)=\Ch(\bar u)
	\qquad
	\forall \bar u\in L^{\infty}(X,\meas)\cap H^{1,2}(X,\dist,\meas),
\end{equation*}
and (iv) follows from arguments that have been used several times in the literature, as in \cite[Proposition 4.20]{Gigli13}.
\end{proof}

As we anticipated above, $\RCD(K,N)$ spaces were introduced in \cite{Gigli15} as a finite dimensional counterpart of $\RCD(K,\infty)$. Here we just recall that they can be characterized asking for the quadraticity of $\Ch$, the volume growth condition $\meas(B_r(x))\le c_1\exp(c_2r^2)$ for some (and thus for all) $x\in X$, the validity of the \textit{Sobolev to Lipschitz property} and of a weak form of Bochner's inequality
\begin{equation*}
\frac{1}{2}\Delta\abs{\nabla f}^2-\nabla f\cdot\nabla\Delta f\ge \frac{\left(\Delta f\right)^2}{N}+K\abs{\nabla f}^2 
\end{equation*}   
for any $f\in \Test(X,\dist,\meas)$. We refer to \cite{AmbrosioMondinoSavare15,ErbarKuwadaSturm15} for a more detailed discussion and equivalent characterizations of the $\RCD(K,N)$ condition.

Let us pass to a brief presentation of the main properties of $\RCD(K,N)$ spaces that will play a role in the sequel.

We recall that, as a consequence of the Bishop-Gromov inequality (see \cite[Theorem 30.11]{Villani09}), $\RCD(K,N)$ spaces are locally doubling, that is to say, for any $R>0$ there exists a constant $C_D$ depending only on $R,K$ and $N$ such that
\begin{equation*}
\meas(B_{2r}(x))\le C_D\meas(B_r(x)),
\end{equation*} 
for any $x\in X$ and for any $0<r\leq R$. Another consequence of the Bishop-Gromov inequality is that $\meas(\partial B_r(x))=0$ for any $x\in X$ and for any $r>0$. 

In \cite{Gigli13} Gigli proved that in $\RCD(0,N)$ spaces the \textit{splitting theorem} still holds, extending to this abstract framework the results obtained by Cheeger-Gromoll and Cheeger-Colding for smooth Riemannian manifolds and Ricci limit spaces, respectively.

\begin{theorem}\label{thm:splitting}
Let $(X,\dist,\meas)$ be an $\RCD(0,N)$ m.m.s. containing a line, that is to say a curve $\gamma:\setR\to X$ such that
\begin{equation*}
\dist(\gamma(s),\gamma(t))=\abs{t-s},\quad\forall s,\,t\in\setR.
\end{equation*}
Then there exists a m.m.s. $(X',\dist',\meas')$ such that $(X,\dist,\meas)$ is isomorphic as a m.m.s. to
\begin{equation*}
(X',\dist',\meas')\times (\setR,\dist_{eucl},\Leb^1).
\end{equation*}
Furthermore:
\begin{itemize}
	\item[(i)] If $n\ge $ then $(X',\dist',\meas')$ is an $\RCD(0,N-1)$;
	\item[(ii)] if $N\in [1,2)$ then $X'$ is a point.
\end{itemize}
Moreover, $\gamma(t)=(x',t)$ for any $t\in\setR$, for some $x'\in X'$.
\end{theorem}

Since $\RCD(K,N)$ spaces are locally doubling and they satisfy a local Poincaré inequality (see \cite{VonRenesse08}), the general theory of Dirichlet forms as developed in \cite{Sturm96} grants that we can find a locally H\"older continuous representative of the heat kernel $p$ on $X\times X\times(0,+\infty)$. 
	
Moreover in \cite{JangLiZhang} the following finer properties of the heat kernel over $\RCD(K,N)$ spaces have been proved: there exist constants $C=C_{K,N}>1$ and $c=c_{K,N}\ge0$ such that
	\begin{equation}\label{eq:kernelestimate}
	\frac{1}{C\meas(B_{\sqrt{t}}(x))}\exp\left\lbrace -\frac{\dist^2(x,y)}{3t}-ct\right\rbrace\le p_t(x,y)\le \frac{C}{\meas(B_{\sqrt{t}}(x))}\exp\left\lbrace-\frac{\dist^2(x,y)}{5t}+ct \right\rbrace  
	\end{equation}
	for any $x,y\in X$ and for any $t>0$. Moreover it holds
	\begin{equation}\label{eq:gradientestimatekernel}
	\abs{\nabla p_t(x,\cdot)}(y)\le \frac{C}{\sqrt{t}\meas(B_{\sqrt{t}}(x))}\exp\left\lbrace -\frac{\dist^2(x,y)}{5t}+ct\right\rbrace \quad\text{for $\meas$-a.e. $y\in X$},
	\end{equation}
	for any $t>0$ and for any $x\in X$. We remark that, in the case $K=0$, it is possible to take $c=0$.

Next we recall the notions of tangent space at a given point of a m.m.s. and of regular $k$-dimensional set. Given a m.m.s. $(X,\dist,\meas)$, $x\in\supp\meas$ and $r\in (0,1)$ we shall consider the rescaled p.m.m.s. $(X,r^{-1}\dist,\meas_r^x,x)$, where
\begin{equation}\label{eq:defC(x,r)}
\meas^x_r:=C(x,r)^{-1}\meas,\qquad\quad C(x,r):=\int_{B_r(x)}\biggl(1-\frac{1}{r}\dist(\cdot,x)\biggr)\di\meas.
\end{equation}

\begin{definition}\label{def:tangent}
Let $(X,\dist,\meas)$ be a m.m.s. and let $x\in \supp\meas$. We say that a p.m.m.s. $(Y,\varrho,\mu,y)$ is tangent to $(X,\dist,\meas)$ if there exist $r_i\downarrow 0$ such that $(X,r_i^{-1}\dist,\meas^x_{r_i},x)$ converge to $(Y,\varrho,\mu,y)$ in the pointed measured Gromov-Hausdorff topology (we refer to the forthcoming \autoref{def: mpGH convergence} for the notion of pointed measured Gromov-Hausdorff convergence).\\
We shall denote by $\Tan_x(X,\dist,\meas)$ the collection of all the tangent spaces of $(X,\dist,\meas)$ at $x$. 
\end{definition}

\begin{definition}\label{def:regularset}
Given an $\RCD(K,N)$ m.m.s. $(X,\dist,\meas)$ we will say that $x\in X$ is a $k$-regular point for some integer $1\le k\le N$ if $\Tan_x(X,\dist,\meas)=\set{(\setR^k,\dist_{eucl},c_k\Leb^k,0^k)}$, where 
\begin{equation*}
c_k:=\left(\int_{B_1(0)}\left(1-\abs{x}\right)\di x\right)^{-1}.
\end{equation*}
 
We shall denote by $\mathcal{R}_k\subset X$ the set of $k$-regular points of $(X,\dist,\meas)$.
\end{definition}

The following theorem sharpens one of the conclusions of \cite{MondinoNaber14} and has been proved in \cite{BrueSemola18}.

\begin{theorem}\label{thm:structuretheory}
Let $(X,\dist,\meas)$ be an $\RCD(K,N)$ m.m.s.. Then there exists a unique integer $1\le k\le N$ such that 
\begin{equation*}
\meas(X\setminus\mathcal{R}_k)=0.
\end{equation*}
\end{theorem} 

Let us conclude this subsection recalling the notion of \textit{non collapsed} $\RCD(K,N)$ m.m.s., as introduced in \cite{DePhilippisGigli18}, and some useful property of this class.

\begin{definition}\label{def:ncRCD}
Let $K\in\setR$ and $N\ge 1$. We say that $(X,\dist,\meas)$ is a non collapsed $\RCD(K,N)$ space, $\ncRCD(K,N)$ space for short, if it is $\RCD(K,N)$ and $\meas=\mathcal{H}^N$.
\end{definition}

It is easy to check that, if $(X,\dist,\meas)$ is $\ncRCD(K,N)$, then $N$ has to be an integer.\\
Below we state a useful regularity property of $\ncRCD$ spaces. Its validity follows from the volume cone-metric cone property (see \cite{DePhilippisGigli16}) and the volume rigidity theorem (see \cite[Theorem 1.5]{DePhilippisGigli18}) with arguments analogous to the ones adopted in theory of non collapsed Ricci-limit spaces. 

\begin{theorem}\label{thm:improvedregularity}
Let $(X,\dist,\meas)$ be a $\ncRCD(K,N)$ m.m.s.. Assume that for some $x\in X$ it holds
\begin{equation*}
\left(\setR^N,\dist_{eucl},c_N\Leb^N,0^N\right)\in \Tan_x(X,\dist,\meas).
\end{equation*}
Then $x$ is a regular point, that is to say the tangent at $x$ is unique (and $N$-dimensional Euclidean).
\end{theorem}
 
\begin{proof}
The conclusion follows from \cite[Proposition 2.10]{DePhilippisGigli18}, we just provide here a sketch of the proof for sake of completeness. Let $r_i\downarrow 0$ be a sequence of scales such that the rescalings of the p.m.m.s. $(X,\dist,\meas,x)$ converge in the pmGH topology to $(\setR^N,\dist_{eucl}, c_N\Leb^N, 0^N)$ as $i\to\infty$. Since $\Leb^N(\partial B_1(0))=0$, we get
\begin{equation*}
\lim_{i\to\infty}\meas(B_{r_i}(x))/\bigl(c_N\omega_Nr_i^N\bigr)=1.
\end{equation*}
The Bishop-Gromov inequality allows then to improve this conclusion to 
\begin{equation*}
\lim_{r\to 0}\meas(B_r(x))/\bigl(c_N\omega_Nr^N\bigr)=1,
\end{equation*}
thus, for any $(Y,\dist_Y,\meas_Y,y)\in\Tan_x(X,\dist,\meas)$ we have that $\meas_Y=c_N\mathcal{H}^N$ by \cite[Theorem 1.2]{DePhilippisGigli18} and $\mathcal{H}^N(B^Y_1(y))=\omega_N$. Hence the volume rigidity theorem \cite[Theorem 1.5]{DePhilippisGigli18} applies, yielding that $B^Y_{{1}/{2}}(y)$ is isometric to $B^{\setR^N}_{{1}/{2}}(0)$. Eventually, thanks to the fact that $(Y,\dist_Y,\meas_Y,y)$ is a metric cone with tip $y$, we conclude that it is isomorphic to $\setR^N$. 
\end{proof}

\begin{remark}\label{rm:menguy}
Let us remark that there is no analogue of \autoref{thm:improvedregularity} without the non collapsing assumption. Indeed Menguy built in \cite{Menguy} an example of Ricci limit space with a strictly weakly regular point, that is to say a point with an Euclidean space in the tangent cone whose tangent cone is not unique.
\end{remark}

\subsection{Convergence and stability results for sequences of $\RCD(K,N)$ spaces}

We dedicate this subsection to an overview of the subject of convergence and stability for Sobolev functions defined on converging sequences of metric measure spaces. The main references for this part are \cite{GigliMondinoSavare15} and \cite{AmbrosioHonda}.

\begin{definition}\label{def: mpGH convergence}
	A sequence $\set{(X_i, \dist_i, \meas_i, x_i)}_{i\in \setN}$ of pointed m.m.s. is said to converge in the pmGH topology to $(Y,\varrho,\mu, y)$ if there exist a complete separable metric space $(Z, \dist_Z)$ and isometric embeddings 
	\begin{align*}
	 &\Psi_i : (\supp \meas_i, \dist_i)\to (Z,\dist_Z)
	 \qquad
	 \forall i\in \setN,\\
	&\Psi: (\supp\mu, \varrho)\to (Z,\dist_Z),
	\end{align*}
	such that for every $\eps>0$ and $R>0$ there exists $i_0$ such that for every $i>i_0$
	\begin{equation*}
		\Psi(B^Y_R(y))\subset [\Psi_i(B^{X_i}_R(x_i))]_{\eps},
		\qquad
		\Psi_i(B^{X_i}_R(x_i))\subset [\Psi(B^Y_R(y))]_{\eps},
	\end{equation*}
	where $[A]_{\eps}:=\set{z\in Z\ : \dist_Z(z,A)<\eps}$ for every $A\subset Z$. Moreover $(\Psi_i)_{\#} \meas_i\weakto \Psi_{\#} \mu$, where the convergence is understood in duality with $\Cbs(Z)$.
\end{definition}

In the case of a sequence of uniformly locally doubling metric measure spaces $(X_i,\dist_i,\meas_i,x_i)$ (as in the case of $\RCD(K,N)$ spaces), the pointed measured Gromov-Hausdorff convergence to $(Y,\varrho,\mu,y)$ can be equivalently characterized asking for the existence of a proper metric space $(Z,\dist_Z)$ such that all the metric spaces $(X_i,\dist_i)$ are isometrically embedded into $(Z,\dist_Z)$, $x_i\to y$ and $\meas_i\weakto\mu$ in duality with $\Cbs(Z)$. This is the so called extrinsic approach, that we shall adopt in the rest of the note.

\begin{definition}\label{def:convpuntunif}
Let $(X_i, \dist_i, \meas_i, x_i)$ be pointed m.m.s. converging in the pmGH topology to $(Y,\varrho,\mu, y)$ and let $f_i:X_i\to\setR$, $f:Y\to\setR$. Assume the convergence to be realized into a common metric space $(Z,\dist_Z)$ as above. Then we say that $f_i\to f$ pointwise if
$f_i(x_i)\to f(x)$ for every sequence of points $x_i\in X_i$ such that $x_i\to x$ in $Z$.
If moreover for every $\varepsilon>0$ there exists $\delta>0$ such that
$\abs{f_i(x_i)-f(x)}\le\varepsilon$
for every $i\ge\delta^{-1}$ and every $x_i\in X_i$, $x\in Y$ with $\dist_Z(x_i,x)\le\delta$,
then we say that $f_i\to f$ uniformly.
\end{definition}

\begin{remark}\label{remark:pointwisevsuniform}
	Let us point out that, if all the spaces coincide and are compact, then $f_i\to f$ pointwise according to \autoref{def:convpuntunif} if and only if $f$ is continuous and $f_i\to f$ uniformly. Therefore the terminology ``pointwise convergence'' might be a bit misleading. Nevertheless we prefer to keep using it since this terminology is used in several other works \cite{AmbrosioHonda,AmbrosioHondaTewodrose17,MondinoNaber14}.
\end{remark}

We recall below the notions of convergence in $L^p$ and Sobolev spaces for functions defined over converging sequences of metric measure spaces. We will be concerned only with the cases $p=2$ and $p=1$ in the rest of the note. We refer again to \cite{AmbrosioHonda,GigliMondinoSavare15} for a more general treatment and the proofs of the results we state below.

\begin{definition}\label{def:L2convergence}
We say that $f_i\in L^2(X_i,\meas_i)$ converge in $L^2$-weak to $f\in L^2(Y,\mu)$ if $f_i\meas_i\weakto f\mu$ in duality with $\Cbs(Z)$ and $\sup_i\norm{f_i}_{L^2(X_i,\meas_i)}<+\infty$.\\
We say that $f_i\in L^2(X_i,\meas_i)$ converge in $L^2$-strong to $f\in L^2(Y,\mu)$ if $f_i\meas_i\weakto f\mu$ in duality with $\Cbs(Z)$ and 
$\lim_i\norm{f_i}_{L^2(X_i,\meas_i)}=\norm{f}_{L^2(Y,\mu)}$.
\end{definition}

\begin{definition}\label{def:H12convergence}
We say that $f_i\in H^{1,2}(X_i,\dist_i,\meas_i)$ are weakly convergent to $f\in H^{1,2}(Y,\varrho,\mu)$ if they converge in $L^2$-weak and $\sup_i\Ch^i(f_i)<+\infty$. Strong $H^{1,2}$-converge is defined asking that $f_i$ converge to $f$ in $L^2$-strong and $\lim_i\Ch^i(f_i)=\Ch(f)$.  
\end{definition}

From now until the end of this subsection we always assume that $(X_i,\dist_i,\meas_i)$ are $\RCD(K,N)$ metric measure spaces for any $i\in\setN$.

The following localized lower semicontinuity result will play a role in the sequel of the note. It is taken from \cite[Lemma 5.8]{AmbrosioHonda}. 

\begin{proposition}\label{prop:locallsc}
Let $f_i\in H^{1,2}(X_i,\dist_i,\meas_i)$ be weakly converging in $H^{1,2}$ to $f\in H^{1,2}(Y,\varrho,\mu)$. Then
\begin{equation*}
\liminf_{i\to\infty}\int_Z g\abs{\nabla f_i}\di\meas_i\ge\int_Z g\abs{\nabla f}\di\mu,\qquad\text{for any nonnegative $g\in \Lipbs(Z)$.}
\end{equation*} 
\end{proposition}


\begin{definition}\label{def:BV convergence}
We say that a sequence $(f_i)\subset L^1(X_i,\meas_i)$ converges $L^1$-strongly to $f\in L^1(Y,\mu)$ if 
\begin{equation*}
   \sigma\circ f_i\meas_i\weakto \sigma\circ f\mu
   \qquad
   \text{and} 
   \qquad
   \int_{X_i}|f_i|\di \meas_i\to \int_Y |f|\di \mu,
\end{equation*}
where $\sigma(z):=\sign(z)\sqrt{|z|}$ and the weak convergence is understood in duality with $\Cbs(Z)$. Equivalently, if $\sigma\circ f_i$ $L^2$-strongly converge to $\sigma\circ f$.

We say that $f_i\in \BV(X_i,\meas_i)$ converge in energy in $\BV$ to $f\in \BV(Y,\mu)$ if $f_i$ converge $L^1$-strongly to $f$ and
\begin{equation*}
	\lim_{i\to\infty} |Df_i|(X_i)= |Df|(Y).
\end{equation*}
\end{definition}

\begin{remark}
	The presence of the function $\sigma$ in the definition of $L^1$-strong convergence is necessary due to the lack of reflexivity of $L^1$.
	Indeed the counterpart of \autoref{def:L2convergence} in the case $p=1$ is easily seen to be not equivalent to convergence in $L^1$ norm when all the spaces coincide.
\end{remark}



The following useful stability result is part of \cite[Proposition 3.3]{AmbrosioHonda}.

\begin{proposition}\label{prop:stability}
Let $p\in \{1,2\}$. If $f_i\in L^p(X_i,\meas_i)$ converge in $L^p$-strong to $f\in L^p(Y,\mu)$ then $\phi\circ f_i$ converge to $\phi\circ f$ in $L^p$-strong for any $\phi\in\Lip(\setR)$ such that $\phi(0)=0$. In particular, if $g_i$ are uniformly bounded in $L^\infty$ and $L^1$-strongly convergent to $g$ then
\begin{equation*}
\lim_{i\to\infty}\norm{g_i}_{L^p(X_i,\meas_i)}=\norm{g}_{L^p(Y,\mu)}.
\end{equation*}\\
If $f_i,\,g_i\in L^p(X_i,\meas_i)$ converge in $L^p$-strong to $f$ and $g$ respectively, then $f_i+g_i$ converge to $f+g$ in $L^p$-strong.\\
\end{proposition}

Below we quote a useful compactness criterion borrowed from \cite[Theorem 6.3]{GigliMondinoSavare15} 
(see also \cite[Theorem 7.4]{AmbrosioHonda}).

\begin{theorem}\label{thm:compactnessL2}
Let $f_i\in H^{1,2}(X_i,\dist_i,\meas_i)$ be such that
\begin{equation*}
\sup_{i}\left\lbrace\int_Z\abs{f_i}^2\di\meas_i+\Ch^i(f_i) \right\rbrace <+\infty
\end{equation*}
and
\begin{equation*}
\lim_{R\to\infty}\limsup_{i\to\infty}\int_{Z\setminus B_R(\bar{z})}\abs{f_i}^2\di\meas_i=0,
\end{equation*}
for some (and thus for all) $\bar{z}\in Z$. Then $(f_i)$ has a $L^2$-strongly convergent subsequence to $f\in H^{1,2}(Y,\varrho,\mu)$.
\end{theorem}

Next we pass to a stability/compactness criterion in $H^{1,2}$. Its statement is taken from \cite[Corollary 5.5]{AmbrosioHonda}.

\begin{proposition}\label{prop:stabilityandcompH12}
\begin{itemize}
\item[(a)] If $f_i\in H^{1,2}(X_i,\dist_i,\meas_i)$, $f_i\in\mathcal{D}(\Delta_i)$ converge in $L^2$-strong to $f$ and $\Delta_if_i$ are uniformly bounded in $L^2$, then $f\in\mathcal{D}(\Delta)$, $\Delta_if_i$ converge in $L^2$-weak to $\Delta f$ and $f_i$ converge in $H^{1,2}$-strong to $f$;
\item[(b)] for all $t>0$, $P_t^if_i$ converge in $H^{1,2}$-strong to $P_tf$ whenever $f_i$ converge in $L^2$-strong to $f$.
\end{itemize}
\end{proposition}

We conclude this subsection with a localized convergence result taken from \cite[Theorem 5.7]{AmbrosioHonda}.

\begin{theorem}\label{thm:localizedconv}
Let $v_i,w_i\in H^{1,2}(X_i,\dist_i,\meas_i)$ be strongly convergent in $H^{1,2}$ to $v,w\in H^{1,2}(Y,\varrho,\mu)$, respectively. Then $\nabla v_i\cdot\nabla w_i$ converge $L^1$-strongly to $\nabla v\cdot\nabla w$.
\end{theorem}

\section{Rigidity of the $1$-Bakry-\'Emery inequality and splitting theorem}\label{sec:rigidityBE}

Our aim in this section is to prove a rigidity result for $\RCD(0,N)$ spaces admitting a non constant function satisfying the equality in the Bakry-\'Emery inequality for exponent $p=1$ \eqref{eq:BE1}. Our investigation of the consequences of this rigidity property was motivated by the study of blow-ups of sets of finite perimeter (see \autoref{thm:rigidtangent} below).

\begin{theorem}\label{maintheorem}
Let $(X,\dist,\meas)$ be an $\RCD(0,N)$ m.m.s.. Assume that there exist a non constant function $f\in \Lipb(X)$ and $s>0$ satisfying
\begin{equation}\label{eq: rigidity}
	|\nabla P_s f|=P_s|\nabla f|
	\qquad
	\meas\text{-a.e. in $X$}.
\end{equation}
Then there exists a m.m.s. $(X',\dist',\meas')$ such that $X$ is isomorphic, as a metric measure space, 
to $X'\times\setR$.
Furthermore:
\begin{itemize}
	\item[(i)] if $N\ge 2$ then $(X',\dist',\meas')$ is an $\RCD(0,N-1)$ m.m.s.;
	\item[(ii)] if $N\in [1,2)$ then $X'$ is a point.
\end{itemize}
Moreover, the function $f$ written in coordinates $(x',t)\in X'\times \setR$ depends only on the variable $t$ and it is monotone.
\end{theorem}

\begin{remark}\label{remark: heat flow in L infty}
	Let us point out that the action of the heat semigroup in $L^{\infty}(X,\meas)$ can be defined by mean of
	\begin{equation}
		P_tf(x):=\int_X f(y)p_t(x,y)\di \meas(y),
	\end{equation}
	where $p_t$ is the heat kernel (see \eqref{eq:heat kernel}). Using an approximation argument is it possible to see that, for any $f\in L^{\infty}(X,\meas)$ and every $\phi\in L^1(X,\meas)$ the map $t\to \int_X P_t f \phi \di \meas$ is absolutely continuous with derivative
	\begin{equation*}
		\frac{\di}{\di t} \int_X P_t f\phi \di \meas =\int_X \Delta P_t f\phi  \di\meas,
	\end{equation*}
	in other words $P_tf$ is still a solution of the heat equation.
\end{remark}

\begin{remark}
	The assumption $f\in \Lipb(X)$ in \autoref{thm:splitting} can be replaced with the more general $f\in \Lip(X)$, provided we extend the action of the heat semigroup to the class of Borel functions with at most linear growth at infinity, i.e.
	\begin{equation*}
	  |f(x)|\leq C(1+\dist(x,x_0))
	  \qquad
	  \text{for}\ \text{any }x\in X
	\end{equation*}
for some $x_0\in X$ and $C\geq 0$. Even though under the $\RCD(0,N)$ condition
the Gaussian estimates for the heat kernel provide this extension, we shall consider only the case 
$f\in \Lipb(X)$ that is enough for our purposes.
\end{remark}

In order to better motivate the statement of \autoref{maintheorem} let us spend a few words about the rigidity case in the Bakry \'Emery inequality for $p=2$. Assume that $(M^n, \dist_g, e^{-V}\text{Vol}_g)$ is a smooth weighted Riemannian manifold with nonnegative generalized $N$-Ricci tensor $\Ric_N$, where
\begin{equation*}
	\Ric_N:=\Ric + \Hess V-\frac{\nabla V \otimes \nabla V}{N-n},
\end{equation*}
and the last term is defined to be $0$ when $V$ is constant and $N=n$.
Let $f:M\to\setR$ be such that $\abs{\nabla P_tf}^2=P_t\abs{\nabla f}^2$ for some $t>0$. Then we can compute
\begin{align*}
0=&P_t\abs{\nabla f}^2-\abs{\nabla P_tf}^2 =\int_0^{t}\frac{\di}{\di s}P_s\abs{\nabla P_{t-s}f}^2\di s\\
=&2\int_0^tP_s\left(\abs{\Hess P_{t-s}f}^2+\Ric_N(\nabla P_{t-s}f,\nabla P_{t-s}f)+\frac{(\nabla V\cdot \nabla P_{t-s}f)^2}{N-n}\right)\di s,
\end{align*}
where the second equality follows from the generalized Bochner identity and $\Delta$ is the weighted Laplacian. Therefore $\Hess f\equiv 0$, $(\nabla V\cdot \nabla f)^2\equiv 0$. Thus $\Delta f\equiv 0$ since
\begin{equation*}
	\frac{(\Delta  f)^2}{N}\le |\Hess f|^2+\frac{(\nabla V\cdot \nabla f)^2}{N-n}= 0.
\end{equation*}
Using a standard argument we obtain that $M^n$ splits isometrically as $L\times \setR$ for some Riemannian manifold $L$. Taking into account the fact that $\Delta f=0$ we can prove that also the measure splits.

Furthermore, denoting by $z,t$ the coordinates on $L$ and $\setR$ respectively, it holds that $P_sf(z,t)=f(z,t)=\alpha t$ for any $s\ge 0$ and for any $t\in\setR$, for some constant $\alpha\neq 0$.

Passing to the study of the case $p=1$,  any function $f:\setR^n\to\setR$ such that $\abs{\nabla P_tf}\equiv P_t\abs{\nabla f}$ is of the form $f(z)=\varphi(z\cdot v)$ for some monotone function $\varphi:\setR\to\setR$ and some $v\in\setR^n$. This is due
to the commutation between gradient operator and heat flow on the Euclidean space and to the characterization of the equality case in Jensen's inequality. More in general, thanks to the tensorization property of the heat flow, it is possible to check that on any product m.m.s. $X=X'\times \setR$, any function $f$ depending only on the variable $t\in\setR$ in a monotone way satisfies $\abs{\nabla P_tf}=P_t\abs{\nabla f}$ almost everywhere. 
Basically \autoref{maintheorem} is telling us that, in the setting of $\RCD(0,N)$ spaces, this is the only possible case.

 About the strategy of the proof let us observe that, as the examples above illustrated show, in the rigidity case for $p=1$ it is not necessarily true that the rigid function has vanishing Hessian. Therefore we cannot directly use $P_sf$ as a splitting function. Still our strategy relies on the properties of the normalized gradient $\nabla P_sf/\abs{\nabla P_sf}$. First we will prove that it has vanishing symmetric covariant derivative and then that its flow lines are metric lines. The conclusion will be eventually achieved building upon the splitting theorem.

Let us start proving that if the rigidity condition \eqref{eq: rigidity} holds for some $s>0$ then it must hold for any $s\ge 0$.
\begin{lemma}
	Let $(X,\dist,\meas)$ be an $\RCD(0, N )$ metric measure space and $f\in\Lipb(X)$. If there exists $s>0$ such that
	\begin{equation}\label{eq: rigidity2}
	|\nabla P_s f|=P_s|\nabla f|
	\qquad
	\meas\text{-a.e. in $X$},
	\end{equation}
	Then $|\nabla P_r f|=P_r|\nabla f|$ for any $r\ge 0$.
\end{lemma}
\begin{proof}
	It is simple to check that $|\nabla P_r f|=P_r|\nabla f|$ for any $0\le r\le s$. Indeed, using \eqref{eq: rigidity2} and the Bakry-\'Emery inequality \eqref{eq:BE1}, we have
	\begin{align*}
	0 &\le P_{s-r}\left( P_r|\nabla f|-|\nabla P_r f| \right)
	   =P_s|\nabla f|-P_{s-r}|\nabla P_r f|
	  =|\nabla P_s f|-P_{s-r}|\nabla P_r f|
	   \le 0.	
	\end{align*}
	Let us now fix $\phi\in \Test_c(X,\dist,\meas)$ and set
	\begin{equation}\label{z3}
		F(r):=\int_X ((P_r|\nabla f|)^2-|\nabla P_r f|^2)\phi \di \meas. 
	\end{equation}
	We claim that $F(r)$ is a real analytic function in $(0,+\infty)$. Observe that the claim, together with the information $F\equiv 0$ in $(0,s)$, implies $F(r)=0$ for any $r\ge 0$ and thus our conclusion, due to the arbitrariness of the test function.\\
	Integrating by parts the right hand side in \eqref{z3} and using \eqref{heat equation}, we can write
	\begin{equation*}
		F(r)=\int_X (P_r|\nabla f|)^2 \phi \di \meas+\frac{1}{2}\frac{\di}{\di r} \int_X (P_r f)^2 \phi \di \meas-\frac{1}{2}\int_X (P_rf)^2 \Delta \phi\di \meas,
	\end{equation*}
	so the claim is a consequence of \autoref{lemma:analytic} below.	
\end{proof}

\begin{lemma}\label{lemma:analytic}
	Let $(X,\dist,\meas)$ be an $\RCD(K,N)$ m.m.s.. For any $g\in L^\infty(X,\meas)$ 
	and any $\phi\in L^1(X,\meas)$ the map $t\mapsto \int_X (P_tg)^2 \phi \di \meas$ is real analytic in $(0,+\infty)$.	
\end{lemma}
\begin{proof}
	Exploiting a well-known analyticity criterion for real functions, it is enough to show, for any 
	$[a,b]\subset (0,+\infty)$, the existence of a constant $C=C(K,N, a,b)$ such that
	\begin{equation}\label{z4}
	  \abs{ \frac{\di ^n}{\di t^n} \int_X (P_tg)^2 \phi \di \meas}
	  \le C^n \norm{g}_{L^{\infty}}^2\norm{\phi}_{L^1}
	  \qquad
	  \forall t\in (a,b), \quad 
	  \forall n\in\setN.
	\end{equation}
	Observe that \eqref{z4} can be checked commuting the operators $P_t$ and $\Delta$   
	and using iteratively the estimate
	\begin{equation}\label{z5}
		\norm{\Delta P_t g}_{L^{\infty}}\le C' \norm{g}_{L^{\infty}}
		\qquad
		\forall t\in (a,b),
	\end{equation}
	where $C'>0$ depends only on $N$, $K$, $a$ and $b$.
	
	Let us prove \eqref{z5} arguing by duality. For any $\psi\in L^1\cap L^2(X,\meas)$, we have
	\begin{align*}
	  \abs{\int_X \Delta P_t g\ \psi \di \meas} = &\abs{\int_X \nabla P_{t/2} g\cdot \nabla P_{t/2} \psi  \di \meas} \\
	  \le & \norm{\nabla P_{t/2}g}_{L^{\infty}}\norm{\nabla P_{t/2}\psi}_{L^1}\\
	  \le &C''\norm{g}_{L^{\infty}}\ C'' \norm{\psi}_{L^1},
	\end{align*}
	where the last inequality is a consequence of the following general fact: there exists $C''(N, K, a,b)>0$ such that
	\begin{equation}\label{z33}
		\norm{\nabla P_t h}_{L^p}\le C'' \norm{h}_{L^p}
		\qquad
		\forall t\in (a,b),\quad \forall h\in L^p(X,\meas)\ \text{with}\ 1\le p\le \infty.
	\end{equation}
	In order to check \eqref{z33} we use the Gaussian estimates for the heat kernel and its gradient \eqref{eq:kernelestimate}, \eqref{eq:gradientestimatekernel} obtaining
	that there exists a constant $\alpha >1$ such that
	\begin{equation*}
	|\nabla P_t h|(x)\le C'' P_{\alpha t}\abs{h}(x),\\
	\qquad
	\text{for}\ \meas \text{-a.e.}\ x\in X,
	\quad \forall t\in (a,b),
	\end{equation*}
	and we take the $L^p$ norm both sides.    
\end{proof}

Let us introduce the most important object of our investigation. For any $s>0$ we consider the vector field
\begin{equation}\label{eq: b}
b_s:=\frac{\nabla P_sf}{P_s|\nabla f|},
\end{equation}
that, since $P_s\abs{\nabla f}>0$ $\meas$-a.e., is well defined and satisfies 
\begin{equation}\label{eq: constant lenght}
	|b_s|=1\quad \meas\text{-a.e. in $X$},
	\quad
	\forall
	s>0,
\end{equation}
thanks to \eqref{eq: rigidity}.

The first important ingredient of the proof of \autoref{maintheorem} is the following proposition. Its proof is inspired by an analogous result in \cite{Gigli13}. 

\begin{proposition}[Variation formula, version 1]\label{prop:variational identity}
	For any $s>0$, $t\ge 0$ and any $g\in \Test(X,\dist,\meas)$ it holds
	\begin{equation}\label{eq:variational identity}
		b_{t+s}\cdot \nabla P_t g=P_t(b_s\cdot \nabla g),
		\qquad
		\meas\text{-a.e. in $X$}.
	\end{equation}
\end{proposition}
Before proving \autoref{prop:variational identity} we need to state a simple lemma.	
\begin{lemma}\label{lemma: Ps rigidity}
	For any $s\geq 0$ the function $P_sf$ satisfies
	\begin{equation}\label{eq rigidity s}
	|\nabla P_{t+s}f|=	P_{t}|\nabla P_sf|,
	\qquad
	\text{$\meas$-a.e. in $X$,\quad $\forall t\geq 0$.}
	\end{equation}	
\end{lemma}	

\begin{proof}
	Using first the Bakry-\'Emery inequality \eqref{eq:BE1} and then twice \eqref{eq: rigidity} we get
	\begin{equation*}
		|\nabla P_{t+s} f|\leq P_t|\nabla P_sf|=P_{t+s}|\nabla f|=|\nabla P_{t+s} f|,
	\end{equation*}
	that proves our claim.
\end{proof}

\begin{proof}[Proof of \autoref{prop:variational identity}]
   Let $s>0$, $t\geq 0$ be fixed. The idea of the proof is to obtain \eqref{eq:variational identity} as the Euler equation associated to the functional
   \begin{equation*}
   	\Psi(h):=\int_X (P_t|\nabla h|-|\nabla P_t h|)\phi \di \meas
   	\qquad
   	h\in \Lip(X),
   \end{equation*}
   where $\phi\in \Lip_{\rm bs}$ is a fixed nonnegative cut-off function.
   Indeed, thanks to \autoref{lemma: Ps rigidity} and the Bakry-\'Emery contraction estimate \eqref{eq:BE1}, we know that $P_sf$ is a minimum of $\Psi$. Thus
   \begin{equation*}
   	\frac{\di}{\di \eps}|_{\eps=0} \Psi(P_sf+\eps g)=0
   	\qquad
   	\forall g\in \Test(X,\dist,\meas).
   \end{equation*}
   Notice that the differentiability of $\varepsilon\mapsto\Psi(P_sf+\eps g)$ at $\eps=0$ can be easily checked using $|\nabla P_s f|=P_s|\nabla f|>0$.
   Then we compute
   \begin{align*}
   	0= & \frac{\di}{\di \eps}|_{\eps=0} \Psi(P_sf+\eps g)\\
   	 = & \frac{\di}{\di \eps }|_{\eps=0} \int_X(P_t|\nabla P_s f+\eps\nabla g|-|\nabla (P_{t+s}f+\eps P_t g)|)\phi \di\meas\\
   	 = & \int_X \left(P_t\left( \frac{\nabla P_s f}{|\nabla P_s f|}\cdot \nabla g\right)-\frac{\nabla P_{t+s}f}{|\nabla P_{t+s} f|}\cdot \nabla P_tg\right)\phi\di\meas\\
   	 =& \int_X (P_t(b_s\cdot \nabla g)-b_{t+s}\cdot \nabla P_t g)\phi\di\meas.
   \end{align*}
   The conclusion follows from the arbitrariness of $\phi$.
\end{proof}

As a first consequence of \autoref{prop:variational identity} we get the following.
\begin{proposition}\label{prop:flow by isomorphisms}
	For any $s>0$ it holds $\div b_s=0$ and $\nabla_{\sym}b_s=0$ according to \autoref{def:symmetric covariant derivative}.
	
	In particular, there exists a regular Lagrangian flow $\XX^s:\setR\times X\rightarrow X$ of $b_s$ with 
	\begin{equation*}
		\left(\XX^s_t\right)_{\#}\meas=\meas, \qquad
		\dist(\XX^s_t(x),\XX^s_t(y))=\dist(x,y)
		\qquad \forall t\in \setR,\ \ \forall x,\,y\in X.
	\end{equation*}
\end{proposition}
\begin{proof}
	Let $g\in\Test_c(X,\dist,\meas)$ be fixed. Using \eqref{eq:variational identity} we obtain
	\begin{align*}
		\abs{\int_X b_s\cdot \nabla g(x)\di \meas(x)}= & 
		\abs{\int_X P_t(b_s\cdot \nabla g)(x)\di \meas(x)}\\
		=&\abs{\int_X b_{t+s}\cdot \nabla P_tg(x) \di \meas(x)}\\
		\leq & \int_X |\nabla P_tg|(x)\di \meas(x).
	\end{align*}
	To get $\div b_s=0$ it suffices to show that
	\begin{equation}\label{eq:limitgrad}
		\lim_{t\to \infty}\int_X |\nabla P_tg|(x)\di \meas(x)=0,
	\end{equation}
	for any nonnegative $g\in \Test_c(X,\dist,\meas)$. 
	To this aim we use the Gaussian estimates for the heat kernel and 
	its gradient \eqref{eq:kernelestimate}, \eqref{eq:gradientestimatekernel} concluding 
	that there exist a constant $C=C(N)>0$ and $\alpha >1$ such that
		\begin{equation}\label{z12}
			|\nabla P_t g|(x)\le \frac{C}{\sqrt{t}} P_{\alpha t}g(x),\\
			\qquad
			\text{for}\ \meas \text{-a.e.}\ x\in X.
		\end{equation}
			
	Let us prove that $\nabla_{\sym}b_s=0$ for any $s>0$. First observe that, since $b_s$ is divergence-free we have
	\begin{equation}\label{z1}
		\int_X b_{t+s}\cdot \nabla P_tg\  P_tg\di \meas =
		\frac{1}{2}\int_X b_{t+s}\cdot \nabla (P_tg)^2 \di \meas
		=0,
	\end{equation}
	for any $g\in \Test(X,\dist,\meas)$, for any $s>0$ and $t\ge 0$. Using again \eqref{eq:variational identity} and \eqref{z1} we deduce
	\begin{align*}
	0=&	\frac{\di}{\di t}{\big|_{t=0}}\int_X b_{t+s}\cdot \nabla P_tg\ P_tg\di \meas
	=\frac{\di}{\di t}{\big|_{t=0}}\int_X P_t\left( b_s\cdot \nabla g   \right)P_tg\di \meas\\
	=&\int_X \Delta( b_s\cdot \nabla g)\ g\di \meas+\int_X b_s\cdot \nabla g \Delta
	g \di \meas\\
	=&2\int_X b_s\cdot \nabla g\  \Delta
	g \di \meas,
	\end{align*}
	that, by polarization, implies our claim. 
	
	The second part of the statement follows from (iv) in \autoref{Th: Lagrangianflows}.
\end{proof}
	
	We are now in position to show that $b_s$ does not depend on $s>0$. 
	
	\begin{lemma}[Variation formula, version 2]\label{prop:variational identity2}
		The vector field $b:=b_s$ does not depend on $s>0$.  In particular, it holds
		\begin{equation}\label{eq: variational2}
		b\cdot \nabla P_tg=P_t(b\cdot \nabla g)
		\qquad
		\meas\text{-a.e.},
		\end{equation}
		for every $g\in \Test(X,\dist,\meas)$ and every $t\geq 0$.
	\end{lemma}
	
	The most important ingredient in the proof of \autoref{prop:variational identity2} is the following lemma.
	
	\begin{lemma}\label{lemma: measure preserving isometry}
		Let $(X,\dist,\meas)$ be an $\RCD(K,\infty)$ m.m.s. and let
		$T:X\to X$ be a measure preserving isometry.  Then, for any $f\in L^2(X,\meas)$, it holds
		\begin{equation}\label{z2}
			P_t (f\circ T)(x)=(P_t f)\circ T(x),
		\end{equation}
		for any $t>0$ and for $\meas$-a.e. $x\in X$.
	\end{lemma}
	
	\begin{proof}
	We just provide a sketch of the proof since the result is quite standard in the field.\\
	First we observe that, since $T$ is a measure preserving isometry, it holds that $f\in H^{1,2}(X,\dist,\meas)$ if and only if $f\circ T\in H^{1,2}(X,\dist,\meas)$ and in that case $\Ch(f\circ T)=\Ch(f)$. From this observation we deduce \eqref{z2}, since the heat flow is the gradient flow of the Cheeger energy in $L^2(X,\meas)$.
	   
	\end{proof}

	\begin{proof}[Proof of \autoref{prop:variational identity2}] Let $s>0$ and let $\XX^s$, the regular Lagrangian flow associated to $b_s$, be fixed. 
	
	We know from \autoref{prop:flow by isomorphisms} that for any $t\in\setR$ the flow map $\XX^s_t$ is a measure preserving isometry of $X$. Therefore, for any $r\ge 0$ and any $g\in \Test(X,\dist,\meas)$, using \eqref{z2} with $T=\XX^s_t$ and \eqref{eq:variational identity}, we get
		\begin{align*}
		(b_s\cdot \nabla P_r g) \circ \XX^s_t=&
			\frac{\di}{\di t} P_r(g)\circ \XX^s_t = \frac{\di}{\di t} P_r(g\circ \XX^s_t)\\
			=&P_r((b_s\cdot \nabla g)\circ \XX^s_t)
			= P_r(b_s\cdot \nabla g)\circ \XX^s_t\\
			=&(b_{r+s}\cdot \nabla P_rg) \circ \XX^s_t.
		\end{align*}
	Since $g$ is arbitrary, the first conclusion in the statement follows. The second one is a direct consequence of \autoref{prop:variational identity}.	
	\end{proof}

Let us denote by $\XX$ the regular Lagrangian flow of $b$ from now on, choosing in particular the ``good representative''
of \autoref{Th: Lagrangianflows} (iv). Our next aim is to prove that for any $x\in X$ the curve $t\mapsto \XX_t(x)$ is a line. This will yield the sought conclusion about the product structure of $(X,\dist,\meas)$ by means of the splitting theorem \autoref{thm:splitting}.
\begin{proposition}\label{prop: HJ equation}
	For all $s>0$ the identity
   	\begin{equation}\label{eq: strange}
   	P_s f(\XX_{-t}(x))=\min_{\overline{B}_t(x)} P_s f
   	\end{equation}
   	holds true for any $t\ge 0$ and any $x\in X$.
\end{proposition}

Before then passing to the proof we wish to explain the heuristic standing behind it with a formal computation:
\begin{equation*}
\frac{\di}{\di t}P_sf(\XX_{-t}(x))=-\nabla P_sf\cdot\frac{\nabla P_sf}{\abs{\nabla P_sf}}(\XX_{-t}(x))=-\abs{\nabla P_sf}(\XX_{-t}(x))=-\abs{\nabla \left(P_sf\circ\XX_{t}\right)}(x).
\end{equation*}
Therefore, setting $u(t,x):=P_sf(\XX_{-t}(x))$, it holds that 
\begin{equation}\label{v}
\partial_tu(t,x)+\abs{\nabla_xu(t,x)}=0
\end{equation}
and it is well known that the Hopf-Lax semigroup
\begin{equation}
	 \mathcal{Q}_tu_0(x):=\min_{\overline{B}_t(x)}u_0
\end{equation}
provides a solution of \eqref{v}, and the unique viscosity solution (see \cite{Nakayasu14}). 
\autoref{prop: HJ equation} is just telling us that $u(t,x)=P_sf(X_{-t}(x))$ is precisely the Hopf-Lax semigroup solution.

\begin{proof}[Proof of \autoref{prop: HJ equation}] Let us denote by $u(t,x)$ the left hand side in \eqref{eq: strange}. 
Since $\dist(\XX_{-t}(x),x)\leq t$, the inequality $\geq$ in \eqref{eq: strange} is obvious.

Now, we claim that for all $\gamma\in \Lip_1([0,\infty); X)$ the function $t\mapsto u(t,\gamma(t))$ is nonincreasing.
In order to prove the claim, first we observe that $t\mapsto u(t,x)=P_sf(\XX_{-t}(x))$ is of class $C^1$, since its derivative 
is $-P_s|\nabla f|(\XX_{-t}(x))$ that is a continuous function. Indeed, the validity of this condition for $\meas$-a.e. $x\in X$ follows from the defining conditions of RLF and we can extend it to all $x\in X$ by continuity of the maps $(t,x)\mapsto u(t,x)$ and $(t,x)\mapsto-P_s\abs{\nabla f}(\XX_{-t}(x))$. Then by the Leibniz rule in 
\cite[Lemma 4.3.4]{AmbrosioGigliSavare05}, it suffices to show that
$$
\limsup_{h\to 0^+}\frac{|u(t,\gamma(t+h))-u(t,\gamma(t))|}{h}\leq P_s|\nabla f|(\XX_{-t}(\gamma(t))).
$$
This inequality follows easily from \autoref{lemma: technical} and the inequality $|\nabla P_sf|\leq P_s|\nabla f|$, since
    \begin{align*}
    	\frac{|u(t,\gamma(t+h))-u(t,\gamma(t))|}{h}
    	\leq \dashint_t^{t+h} P_s|\nabla f|(\XX_{-t}(\gamma(r))) \di r,
    \end{align*}
    (here we also used that $r\mapsto \XX_{-t}(\gamma(r))$ is 1-Lipschitz), by taking the limit as $h\downarrow 0$.
     
From the claim, the converse inequality in \eqref{eq: strange} follows easily, because for all $x\in X$ and all minimizers
$\bar x$ of $P_sf$ in $\overline{B}_t(x)$ the geodesic property of $(X,\dist)$ grants the existence of  $\gamma\in \Lip_1([0,\infty); X)$
with $\gamma(t)=x$ and $\gamma(0)=\bar x$. It follows that
$$
u(t,x)=u(t,\gamma(t))\leq u(0,\gamma(0))=u(0,\bar x)=P_sf(\bar x)=\min_{\overline{B}_t(x)}P_sf.
$$ 
\end{proof}

\begin{lemma}\label{lemma: technical}
	Let $(X,\dist,\meas)$ be an $\RCD(K,\infty)$ m.m.s. and $u\in \Lip(X)$. Assume that $|\nabla u|$ has a continuous representative in $L^{\infty}(X,\meas)$. Then
	\begin{equation}\label{eq:genupper}
	  |u(\gamma(t))-u(\gamma(s))|\leq \int_s^t |\nabla u|(\gamma(r))\abs{\gamma'}(r) \di r,	
	\end{equation}
	for any $s<t$ and for any Lipschitz curve $\gamma:\setR\to X$ (where we denoted by $|\nabla u|$ the continuous representative of the minimal relaxed slope of $u$).
\end{lemma}

\begin{proof}
To get the sought conclusion we argue by regularization via heat flow as in the proof of \cite[Theorem 6.2]{AmbrosioGigliSavare14}.\\
Let $\left(\mu^{\lambda}_r\right)_{r\in\setR}$ be defined by $\mu^{\lambda}_r:=\left(P_{\lambda}\right)^*\delta_{\gamma(r)}$. Contractivity yields now that
\begin{align}\label{eq:contrest}
\abs{P_{\lambda}u\left(\gamma(t)\right)-P_{\lambda}u(\gamma(s))}\le& \int_s^t\left(\int\abs{\nabla u}^2\di\mu^{\lambda}_r\right)^{\frac{1}{2}}\abs{\dot{\mu}^{\lambda}_r}\di r\nonumber \\
\le&e^{-K\lambda}\int_s^t\left(\int\abs{\nabla u}^2\di\mu^{\lambda}_r\right)^{\frac{1}{2}}\abs{\dot{\gamma}_r}\di r\\
=&e^{-K\lambda}\int_s^t  \left( P_{\lambda}\abs{\nabla u}^2(\gamma(r))\right)^{\frac{1}{2}}\abs{\dot{\gamma}_r}\di r,\nonumber
\end{align}
for any $\lambda>0$ and for any $s,t\in\setR$.
Passing to the limit as $\lambda\downarrow 0$ both the first and the last expression in \eqref{eq:contrest} and taking into account the continuity of $u$ and $\abs{\nabla u}$, we obtain \eqref{eq:genupper}.
\end{proof}

By means of \autoref{prop: HJ equation} we can easily prove the following.

\begin{corollary}\label{cor:flowlinesarelines}
	For any $x\in X$ the curve $t\mapsto \XX_t(x)$ is a line, that is to say
	\begin{equation*}
	\dist(\XX_t(x),\XX_s(x))=\abs{t-s}\qquad \forall s,t\in\setR.
	\end{equation*}
\end{corollary}
\begin{proof}
   Let us start observing that any $x_t\in \overline{B}_t(x)$ such that
   \begin{equation*}
   	\min_{y\in \overline{B}_t(x)} P_sf(y)=P_sf(x_t)
   \end{equation*}
   has to satisfy $\dist(x,x_t)=t$. Otherwise we might replace $x_t$ with $\XX_{-\eps}(x_t)$ (that belongs to $B_t(x)$ for $\eps$ sufficiently small) and, since $P_sf$ is strictly increasing along the flow lines of $\XX$, we would get a contradiction.\\
   Furthermore $\XX_t(x)\in\overline{B}_t(x)$ since $|b|=1$. Thus it follows from \eqref{eq: strange} that $\dist(\XX_{-t}(x),x)=t$ for any $t\ge 0$. Using the semigroup property and the fact that $\XX_t$ is an isometry for any $t\in \setR$ (see \autoref{prop:flow by isomorphisms}) we get the sought conclusion. 
\end{proof}

\begin{proof}[Proof of \autoref{maintheorem}.]
As we anticipated the conclusion that $X$ is isomorphic to $X'\times\setR$ for some $\RCD(0,N-1)$ m.m.s. $(X',\dist',\meas')$ follows from \autoref{cor:flowlinesarelines} applying \autoref{thm:splitting}.

Let us deal with the second part of the statement.\\ 
First of all we claim that all the flow lines of $\XX$ are vertical lines in $X$, that is to say,
denoting by $(z,s)\in X'\times\setR$ the coordinates on $X$, $\XX_t(z,s)=(z,t+s)$ for any $z\in X'$ and for any $s,t\in \setR$. 
Indeed, since we proved that all integral curves of $b$ are lines in $(X,\dist)$, the construction provided by the splitting theorem
shows that this is certainly true for a fixed $\bar{z}\in X'$. Let us consider any other $z\in X'$ and call $\XX_t((z,0))=\left(\XX^1_t((z,0)),\XX_t^2(z,0)\right)$. Taking into account the semigroup property \eqref{eq:semigroup property}
and the fact that $\XX_t$ is an isometry for any $t\in\setR$, for any $\tau\in\setR$ we can compute
\begin{align*}
\tau^2+\dist_Z^2(\bar{z},z)=&\dist^2\left(\XX_{\tau}((\bar{z},0)),(z,0)\right)=\dist^2\left(\XX_{t+\tau}((\bar{z},0)),\XX_t((z,0))\right)\\=&\dist^2\left((\bar{z},t+\tau),(\XX^1_t((z,0)),\XX_t^2((z,0)))\right)\\
=&\bigl|(\XX_t^2((z,0))-t)-\tau\bigr|^2+\dist_Z^2\left(\bar{z},\XX_t^1((z,0))\right).
\end{align*}
Since $\tau$ is arbitrary,
it easily follows that $\XX_t^2((z,0))=t$ for any $t\in\setR$ and therefore $\XX_t^1((z,0))=z$ for any $t\in\setR$, as we claimed.

From what we just proved it follows that $\nabla P_s f$ is trivial in the $z$ variable and we can conclude that $P_sf$ depends only on the $t$-variable for any $s>0$ thanks to the tensorization of the Cheeger energy (see \cite[Theorem~6.19]{AmbrosioGigliSavare14}).
Passing to the limit as $s\downarrow 0$ we obtain that the same holds true also for $f$.

Knowing that $f$ depends only on the $t$-variable, the monotonicity in this variable can be immediately checked.   
\end{proof}

\section{Convergence and stability results for sets of finite perimeter}\label{sec:compstab}

In this section we establish some useful compactness and stability results for sequences of sets of finite perimeter defined on a pmGH converging sequence of $\RCD(K,N)$ m.m. spaces. Most of the results we present here adapt and extend to the case of our interest those of \cite{AmbrosioHonda}.

Until the end of this section we fix a sequence $\set{(X_i, \dist_i, \meas_i, x_i)}_{i\in \setN}$ of pointed $\RCD(K,N)$ m.m. spaces  converging in the pmGH topology to $(Y,\varrho,\mu, y)$ and a proper metric space $(Z, \dist_Z)$ that realizes this convergence according to \autoref{def: mpGH convergence}. 

Since in the rest of the note we will be mainly interested on the case of indicator functions, let us observe that, in that case, we can rephrase the notion of $L^1$-strong convergence introduced in \autoref{def:BV convergence} in the following way. 

\begin{definition}\label{def:L1 convergence of sets}
	We say that a sequence of Borel sets $E_i\subset X_i$ such that $\meas_i(E_i)<\infty$ for any $i\in\setN$ converges in $L^1$-strong to a Borel set $F\subset Y$ with $\mu(F)<\infty$ if $\chi_{E_i}\meas_i\weakto \chi_F\mu$ in duality with $\Cbs(Z)$ and $\meas_i(E_i)\to \mu(F)$.
	
	We also say that a sequence of Borel sets $E_i\subset X_i$ converges in $L^1_{\loc}$ to a Borel set $F\subset Y$ if $E_i\cap B_R(x_i)\to F\cap B_R(y)$ in $L^1$-strong for any $R>0$.
\end{definition}

\begin{remark}
Let us remark that $L^1$-strong convergence implies $L^1_{\loc}$-strong convergence as a consequence of \autoref{lemma:convergence intersection} and the following observation:
\begin{equation*}
\chi_{B_R(x_i)}\to\chi_{B_R(y)}\quad\text{in $L^1$-strong, for any $R>0$.}
\end{equation*}
This convergence property follows from the already remarked fact that spheres have vanishing measure on $\RCD(K,N)$ spaces.
\end{remark}

Let us begin with a compactness result which adapts \cite[Proposition~7.5]{AmbrosioHonda} to the case of our interest (basically, we add the uniform $L^{\infty}$ bound and this allows to remove the assumption on the existence of a common isoperimetric profile).

\begin{proposition}\label{prop:BV compactness}
Let $(X_i,\dist_i,\meas_i,x_i)$, $(Y,\varrho,\mu,y)$, and $(Z,\dist_Z)$ be as above and fix $r>0$.
For any sequence of functions $f_i\in \BV(X_i,\meas_i)$ such that $\supp f_i\subset \overline{B}_r(x_i)$ for any $i\in\setN$ and
\begin{equation}
\sup_{i\in \setN}\left\lbrace  |Df_i|(X_i)+\norm{f_i}_{L^{\infty}(X_i,\meas_i)}\right\rbrace <\infty,
\end{equation}
there exist a subsequence $i(k)$ and $f\in L^\infty(Y,\mu)\cap \BV(Y,\varrho,\mu)$ with
$\supp f\subset \overline{B}_r(y)$ such that $f_{i(k)}\to f$ in $L^1$-strong.
\end{proposition}

As a corollary, a truncation and a diagonal argument provide a compactness result for sequences of
sets with locally uniformly bounded perimeters.

\begin{corollary}\label{cor:perimeters compactness}
For any sequence of Borel sets $E_i\subset X_i$ such that
\begin{equation}\label{eq: bounded perimeters}
	\sup_{i\in \setN} \Per(E_i, B_R(x_i))<\infty
	\qquad
	\forall R>0
\end{equation}
there exist a subsequence $i(k)$ and a Borel set $F\subset Y$ such that $E_{i(k)}\to F$ in $L^1_{\loc}$.
\end{corollary}
We postpone the proof of \autoref{prop:BV compactness} and \autoref{cor:perimeters compactness} after a technical lemma that will play a role also in the sequel.

\begin{lemma}\label{lemma:convergence intersection}
	Let $(X_i,\dist_i,\meas_i,x_i)$, $(Y,\varrho,\mu,y)$, and $(Z,\dist_Z)$ be as above and $E_i,\,\tilde{E}_i\subset X_i$ satisfy $\meas_i(E_i)+\meas_i(\tilde{E}_i)<\infty$. If $E_i\to F$ and $\tilde{E}_i\to\tilde{F}$ in $L^1$-strong, for some Borel sets $F,\,\tilde{F}\subset Y$, then 
	$E_i\cap\tilde{E}_i\to F\cap\tilde{F}$ in $L^1$-strong.
\end{lemma}

\begin{proof}
	Observing that
	\begin{equation*}
		\chi_{E_i\cap\tilde{E}_i}=\chi_{E_i}\cdot\chi_{\tilde{E}_i}
		=\frac{1}{4}\left[(\chi_{E_i}+\chi_{\tilde{E}_i})^2- (\chi_{E_i}-\chi_{\tilde{E}_i})^2 \right],
	\end{equation*}
	the conclusion follows from \autoref{prop:stability}.
\end{proof}

\begin{proof}[Proof of \autoref{cor:perimeters compactness}]
	We claim that, possibly extracting a subsequence that we do not relabel, there exist radii $R_\ell\uparrow \infty$ as $\ell\to\infty$ with the following property
	\begin{equation}\label{z17}
		\sup_{i\in \setN} \Per(B_{R_\ell}(x_i),X_i)<\infty
		\qquad
		\forall l\in\setN.
	\end{equation}
	Indeed, applying the coarea formula in the localized version of \autoref{cor:loccoarea} to the functions $\dist(x_i,\cdot)$ and recalling that $\abs{\nabla \dist(x_i,\cdot)}_i=1$ $\meas_i$-a.e. for any $i$, we obtain
	\begin{equation*}
	\int_0^R\Per(B_r(x_i),X_i)\di r=\meas_i(B_R(x_i))\quad\text{for any $R>0$ and $i\in\setN$}.
	\end{equation*}
	Observing that for any $R>0$ it holds $\meas_i(B_R(x_i))\to\mu(B_R(y))$, an application of Fatou's lemma yields now
	\begin{equation}\label{eq:limtech}
	\int_0^R\liminf_{i\to\infty}\Per(B_r(x_i),X_i)\di r\le\liminf_{i\to\infty}\meas_i(B_R(x_i))=\mu(B_R(y))\quad\text{for any $R>0$.}
	\end{equation}
	The claimed conclusion \eqref{z17} can be obtained from \eqref{eq:limtech} via a diagonal argument.\\
	For any $\ell\in\setN$ we can now estimate
	\begin{equation*}
	\sup_{i\in \setN}	\Per(E_i\cap B_{R_\ell}(x_i),X)\le \sup_{i\in \setN}\Per(E_i,B_{R_\ell+1}(x_i))+\sup_{i\in \setN}\Per(B_{R_\ell}(x_i),X)<\infty,
	\end{equation*}
	thanks to the locality and subadditivity of perimeters (see \cite[pg. 8]{Am02}) for the first inequality and to \eqref{eq: bounded perimeters}, \eqref{z17} for the second one. Thus for any $\ell\in\setN$ we can apply \autoref{prop:BV compactness} to the functions $f_i:=\chi_{E_i\cap B_{R_\ell}(x_i)}$. Observing that $L^1$-strong limits of characteristic functions are characteristic functions (as a consequence of \autoref{prop:stability}), we can use a diagonal argument together with \autoref{lemma:convergence intersection} to recover the global limit set.
\end{proof}

\begin{proof}[Proof of \autoref{prop:BV compactness}]
Let us fix $t>0$. For any $i\in \setN$ we write $f_i=P^i_tf_i+(f_i-P^i_tf_i)$ where, for any $i\in\setN$, $P^i_t$ denotes the heat semigroup on $(X_i,\dist_i,\meas_i)$. Observe that, as a consequence of the regularizing estimates \eqref{eq:regularizingheat}, it holds that
\begin{equation}\label{z6}
	\sup_{i\in \setN} \left\lbrace \int_{Z} |P^i_t f_i|^2 \di \meas_i + \Ch^i(P^i_t f_i)\right\rbrace <+\infty,
\end{equation}
where $\Ch^i$ is the Cheeger energy on $(X_i,\dist_i, \meas_i)$. Moreover, we claim that
\begin{equation}\label{z7}
	\limsup_{R\to \infty}\sup_{i\in\setN} \int_{Z\setminus B_R(x_i)}|P^i_t f_i|^2\di \meas_i=0\qquad\forall t>0.
\end{equation}	
Indeed, using both the Gaussian estimates for the heat kernel in \eqref{eq:kernelestimate}, we get
\begin{align*}
	\int_{Z\setminus B_R(x_i)}&|P^i_t f_i|^2\di \meas_i \\ \le &
	\norm{f_i}_{L^{\infty}(X_i,\meas_i)} \int_{Z\setminus B_R(x_i)}P^i_t|f_i| \di \meas_i\\
	\le & C\norm{f_i}_{L^{\infty}(X_i,\meas_i)}\int_{Z\setminus B_R(x_i)} \int_{B_r(x_i)} \frac{e^{-\frac{\dist^2(x,y)}{5t}+ct}}{\meas_i{(B_{\sqrt{t}}(x))}}|f_i(y)|\di \meas_i(y)\di \meas_i(x)\\
	\le & Ce^{-\frac{(R-r)^2}{10t}}\norm{f_i}_{L^{\infty}(X_i,\meas_i)}\int_{Z\setminus B_R(x_i)} \int_{B_r(x_i)} \frac{e^{-\frac{\dist^2(x,y)}{10t}+ct}}{\meas_i{(B_{\sqrt{t}}(x))}}|f_i(y)|\di \meas_i(y)\di \meas_i(x)\\
	\le &C_{t}e^{-\frac{(R-r)^2}{10t}}\norm{f_i}_{L^{\infty}(X_i,\meas_i)}\int_{Z}P^i_{\alpha t}|f_i| \di \meas_i\\
	\le &C_{t}e^{-\frac{(R-r)^2}{10t}}\norm{f_i}_{L^{\infty}(X_i,\meas_i)}\norm{f_i}_{L^1(X_i,\meas_i)},
\end{align*}
where $\alpha>0$ is a constant depending only on $K$ and $N$.

Taking into account \eqref{z6} and \eqref{z7}, we can apply \autoref{thm:compactnessL2} to get that $P^i_tf_i$ admits a subsequence converging in $L^1$-strong. In order to conclude the proof it suffices to observe that
\begin{equation*}
	\lim_{t\to 0^+}\sup_{i\in \setN} \int_{X_i}|P^i_tf_i-f_i| \di\meas_i=0,
\end{equation*}
as it follows from the inequality
\begin{equation*}
	\int_{X_i}|P^i_tf_i-f_i| \di\meas_i\le C(K,t)|Df_i|(X_i),
\end{equation*}
with $C(K,t)\sim \sqrt{t}$ as $t\to 0$ (see for instance \cite[Proposition 6.3]{AmbrosioHonda}).
\end{proof}

Let us pass to a lower semicontinuity result for the total variations.

\begin{proposition}\label{prop:semicontinuity BV}
Let $(X_i,\dist_i,\meas_i,x_i)$ be $\RCD(K,N)$ m.m. spaces converging in the pmGH topology to $(Y,\varrho,\mu,y)$ and $(Z,\dist_Z)$ realizing the convergence as above. Let $f_i\in \BV(X_i,\meas_i)$ converge in $L^1$-strong to $f\in L^1(Y,\mu)$. If $\sup_{i}\abs{Df_i}(X_i)<\infty$ 
then $f\in\BV(Y,\varrho,\mu)$ and 
    \begin{equation}\label{eq:lscvariation}
    \liminf_{i\to \infty}|Df_i|(X_i)\ge |Df|(Y).
    \end{equation}
    Furthermore, if 
    	\begin{equation}\label{eq:unifinftybound}
    		\sup_{i\in \setN}\norm{f_i}_{L^{\infty}(X_i,\meas_i)}<\infty,	
    	\end{equation}
    then
	\begin{equation}\label{z8}
		\liminf_{i\to \infty} \int_{X_i} g \di |Df_i|\ge \int_{Y} g \di |Df|,
		\qquad
		\text{ for all $g\in \Lipbs(Z)$ nonnegative.}
	\end{equation}
\end{proposition}

Before than proving \autoref{prop:semicontinuity BV} we state and prove a simple corollary of it.

\begin{corollary}\label{cor:weak convergence total variations}
	Let $(X_i,\dist_i,\meas_i,x_i)$ be $\RCD(K,N)$ m.m. spaces converging in the pmGH topology to $(Y,\varrho,\mu,y)$ and $(Z,\dist_Z)$ realizing the convergence as above.
	For any $f_i\in \BV(X_i,\dist_i,\meas_i)$ convergent in energy in $\BV$ to $f\in \BV(Y,\varrho,\mu)$ such that
	$\sup_i\norm{f_i}_{L^{\infty}(X_i,\meas_i)}<+\infty$,	
	it holds that $|Df_i|\weakto |Df|$ in duality with $\Cbs(Z)$.
\end{corollary}

\begin{proof}[Proof of \autoref{cor:weak convergence total variations}]
	From \eqref{z8} we can deduce with a standard measure theoretic argument that
	\begin{equation}\label{z9}
		\liminf_{i\to \infty} |Df_i|(A)\ge |Df|(A)
		\qquad
		\forall A\subset Z\ \text{open and bounded.}
	\end{equation}
	Let $\nu$ be any weak limit point of $|Df_i|$, in the weak topology induced by $\Cbs(Z)$, along some subsequence $i(k)$ (the sequence $|Df_i|(X_i)$ is bounded and therefore the family $\left\lbrace \abs{Df_i}\right\rbrace_i$ is weakly compact). For any open and bounded set $A\subset Z$ such that $\nu(\partial A)=0$, it holds $\lim_k |Df_{i(k)}|(A)=\nu(A)$. Hence, taking into account also \eqref{z9}, we get $|Df|(A)\le \nu(A)$. Thus $|Df|\le \nu$, as measures in $Z$.
	On the other hand, since the evaluation on open sets is lower semicontinuous w.r.t. the weak convergence induced by $\Cbs(Z)$,
	by definition of convergence in energy in $\BV$, we have
	$\nu(Z)\le \liminf_k|Df_{i(k)}|(Z)=|Df|(Z)$ and therefore $\nu=\abs{Df}$.	
\end{proof}

\begin{proof}[Proof of \autoref{prop:semicontinuity BV}]
	The first part of the statement has been proved in \cite[Theorem 6.4]{AmbrosioHonda}.
	
	Let us deal with the second one. Fix any $t>0$ and observe that $P^i_t f_i\to P_tf$ in $H^{1,2}$ according to \autoref{def:H12convergence}. Indeed, the $L^1$-strong convergence of $f_i$ to $f$, combined with \eqref{eq:unifinftybound}, yields that $f_i$ converge in $L^2$-strong to $f$ by \autoref{prop:stability}. Therefore we can apply \autoref{prop:stabilityandcompH12} to obtain the claimed conclusion. Hence \autoref{prop:locallsc} applies, yielding that
	\begin{equation}\label{eq:regularized}
			\liminf_{i\to \infty} \int_{Z} g |\nabla P^i_tf_i|\di\meas_i\ge \int_{Z} g |\nabla P_tf|\di \mu,
		\qquad
		\text{for all $g\in \Lipbs(Z)$ nonnegative.}
	\end{equation}
	In order to prove \eqref{z8} starting from its regularized version \eqref{eq:regularized}, we argue as in the proof of \cite[Lemma 5.8]{AmbrosioHonda}.  Taking into account the Bakry-\'Emery contraction estimate $\abs{\nabla P_th}\le e^{-Kt}P_t^{*}\abs{Dh}$ (see \eqref{eq:BE1})
	and the estimate
	\begin{equation*}
		\norm{P_tg-g}_{L^{\infty}}\le C(K,N,t){\rm Lip}(g),\qquad
		\text{with}\ C(K,N,t)\sim \sqrt{t}\ \text{as}\ t\to 0
	\end{equation*}
	which is available over any $\RCD(K,N)$ m.m.s. (and can be proved using the Gaussian estimates for the heat kernel \eqref{eq:kernelestimate}), we obtain
 \begin{align}\label{eq:fromregtononreg}
   \liminf_{i\to\infty}\int_Zg\di\abs{D f_i}\ge&\liminf_{i\to\infty}\int_Z P^i_tg\di\abs{Df_i}-\limsup_{i\to\infty}\int_Z|P^i_tg-g|\di\abs{Df_i}\\
  \ge&e^{Kt}\liminf_{i\to\infty}\int_Zg\abs{\nabla P^i_tf_i}\di\meas_i-C(K,N,t){\rm Lip}(g)\limsup_{i\to\infty}\abs{Df_i}(X_i)\nonumber\\
  \ge& e^{Kt}\int_Zg\abs{\nabla P_tf}\di\mu-C(K,N,t){\rm Lip}(g)\limsup_{i\to\infty}\abs{Df_i}(X_i).\nonumber
\end{align}
The sought conclusion \eqref{z8} can be obtained passing to the $\liminf$ as $t\to 0$ in \eqref{eq:fromregtononreg}, recalling that $\abs{\nabla P_tf}\mu\weakto\abs{Df}$ in duality with $\Cbs(Z)$ as $t\downarrow 0$.
\end{proof}

The next result deals with the possibility of approximating in $\BV$ energy a set of finite perimeter in the limit space with a sequence of sets of finite perimeter defined on the approximating spaces.

\begin{proposition}\label{prop:Mosco convergence global}
		Let $(X_i,\dist_i,\meas_i,x_i)$ be $\RCD(K,N)$ m.m. spaces converging in the pmGH topology to $(Y,\varrho,\mu,y)$ and let $(Z,\dist_Z)$ be realizing the convergence as above. Let $F\subset Y$ be a bounded set of finite perimeter. Then there exists a subsequence $(i_k)$ and (uniformly bounded) sets of finite perimeter $E_{i_k}\subset X_{i_k}$ such that $\chi_{E_{i_k}}\to \chi_F$ in energy in $\BV$ as $k\to\infty$.	
\end{proposition}
\begin{proof}
	 Let us begin observing that the first part of \cite[Theorem 8.1]{AmbrosioHonda} provides existence of a sequence 
	 $(g_i)\subset\BV(X_i,\meas_i)$ strongly converging in $\BV$ to $\chi_F$. 
	 Since by assumption $F\Subset B_R(y)$ for some $R>0$, we can find a Lipschitz function $\eta:Z\to[0,1]$ with support
	 contained in $B_{2R}(y)$ such that $\eta|_{B_R(y)}\equiv 1$ and it is easy to check that the sequence $f_i:=\eta g_i$ still converges in $L^1$-weak to $\chi_F$ and satisfies $\abs{Df_i}\to\Per(F)$ as $i\to\infty$.
	Furthermore, possibly composing with $\phi(z):=(z\wedge 1)\vee 0$, using \autoref{prop:stability} and observing that $\abs{D\phi\circ f_i}(X_i)\le\abs{Df_i}(X_i)$ for any $i\in\setN$ while $\abs{D\phi\circ\chi_F}(Y)=\abs{D\chi_F}(Y)$, we can assume that $0\le f_i\le 1$ for any $i\in \setN$. In particular $\sup_{i\in\setN}\norm{f_i}_{L^{\infty}(X_i,\meas_i)}<+\infty$. Therefore, \autoref{prop:BV compactness} applies and we obtain that, possibly extracting a subsequence that we do not relabel, $f_i$ converge in $\BV$ energy to $\chi_F$.
	
	Let us now assume, possibly extracting one more subsequence, that the measures $(f_i)_\#(\chi_{B_{2R}(y)}\meas_i)$
	weakly converge to some measure $\sigma$ in $[0,1]$. Under this assumption, 
	we claim that $\chi_{\{f_i>\lambda\}}$ still converge to $\chi_F$ in $L^1$-strong for $\Leb^1$-a.e. $\lambda\in (0,1)$.\\
	In order to prove this claim, we fix $\lambda\in (0,1)$ that is not an atom of $\sigma$, so that
	\begin{equation}\label{z10}
	  \lim_{\eps\to 0} \lim_{i\to\infty} \meas_i(\set{\lambda-\eps<f_i\le \lambda})=0.
	  \qquad
	\end{equation}
	From \eqref{z10}, using \autoref{prop:stability}, it is immediate to get the $L^1$-strong convergence of 
	$\chi_{\{f_i>\lambda\}}$ to $\chi_F$: indeed, it suffices to observe that for all $\eps\in (0,\lambda)$ the
	functions $\psi_\eps\circ f_i$ still $L^1$-strongly converge to $\psi_\eps\circ\chi_F=\chi_F$ for 
	any $\psi$ continuous, identically equal to 0 on $[0,\lambda-\eps]$
	and identically equal to 1 on $[\lambda,1]$.
    From the $L^1$-strong convergence we get, in particular, 
	\begin{equation}\label{z11}
		\liminf_{i\to\infty} \Per(\set{f_i>\lambda}, X_i)\geq \Per(F,Y)
		\qquad
		\qquad\text{for $\Leb^1$-a.e. $\lambda\in (0,1)$.}
	\end{equation}
	On the other hand, the coarea formula \autoref{thm:coarea} and the strong convergence of $f_i$ yield
	\begin{equation}\label{z111}
	    \limsup_{i\to \infty}\int_0^1\Per(\set{f_i>\lambda},X_i)\di \lambda
		=\limsup_{i\to \infty}|Df_i|(X_i)=\Per(F,Y).
	\end{equation}
Thanks to Scheffè's lemma,
the combination of \eqref{z11} and \eqref{z111} gives that $\Per(\set{f_i>\lambda}, X_i)$ converge in $L^1(0,1)$ to
the constant $\Per(F,Y)$. Extracting a subsequence $(i(k))$ pointwise convergent on $(0,1)\setminus I$ with $\Leb^1(I)=0$
and setting $E_k=\{f_{i(k)}>\lambda\}\subset B_{2R}(y)$ with $\lambda\in (0,1)\setminus I$ and $\sigma(\{\lambda\})=0$, the conclusion is achieved.	
\end{proof}

Let us conclude this section with a convergence result for quasi-minimal sets of finite perimeter. It will play a key role in the study of blow-ups of sets of finite perimeter we are going to perform in \autoref{sec:tangentsRCD}. The strategy of the proof is classical, see for instance \cite[Theorem 4.8]{Ambrosio97}. 

\begin{proposition}\label{prop: sequence of lambda minimizers}
		Let $(X_i,\dist_i,\meas_i,x_i)$ be $\RCD(K,N)$ m.m. spaces converging in the pmGH topology to $(Y,\varrho,\mu,y)$ and let $(Z,\dist_Z)$ be realizing the convergence as above. For any $i\in\setN$, let $\lambda_i\ge 1$ and let $E_i\subset X_i$ be a set of finite perimeter satisfying the following 
		$\lambda_i$-minimality condition: there exists $R_i>0$ such that
		\begin{equation*}
			\Per(E_i, X_i)\le \lambda_i \Per(E', X_i)
			\qquad
			\forall E'\subset X_i\ \text{such that}\ E_i\Delta E'\Subset B_{R_i}(x_i).
		\end{equation*}
		Assume that, as $i\to\infty$, $E_i\to F$ in $L^1_{\loc}$ for some set $F\subset Y$ of locally finite perimeter, $\lambda_i\to 1$ and $R_i\to \infty$. 
		Then
		\begin{itemize}
		   \item[(i)] $F$ is an entire minimizer of the perimeter (relative to $(Y,\varrho,\mu)$), namely
		   \begin{equation*}
		   	  \Per(F, B_r(y))\leq \Per(F',B_r(y))\quad\text{whenever}\ F\Delta F'\Subset B_r(y)\Subset Y \text{and $r>0$};
		   \end{equation*}
		   \item[(ii)] $|D\chi_{E_i}|\weakto|D\chi_F|$ in duality with $\Cbs(Z)$.
		\end{itemize}
\end{proposition}

\begin{proof}
 Let us fix $\bar y\in Y$ and let $F'\subset Y$ be a set of locally finite perimeter satisfying $F\Delta F'\Subset B_r(\bar{y})$.  
 Let $\bar{x}_i\in X_i$ converging to $\bar y$ in $Z$ and $R>0$ be such that the following properties hold true:
  \begin{equation}\label{z20}
  	\sup_{i\in \setN} \Per(B_R(x_i),X_i)<+\infty
  	\qquad
  	\text{and}
  	\qquad
  	B_r(\bar{x}_i)\Subset B_R(x_i)\qquad \forall i\in \setN.
  \end{equation}
Using \autoref{prop:Mosco convergence global} we can find a sequence of sets of finite 
perimeter $E'_i\subset X_i$ converging to $F\cap B_R(y)$ in $\BV$ energy
(note that $F\cap B_R(y)$ is a set of finite perimeter thanks to \eqref{z20}).

Let $\nu$ be any weak limit of the sequence of measures with uniformly bounded mass $|D\chi_{E_i}|$. We claim that
\begin{equation}\label{z18}
	\nu(B_s(\bar y))\le \Per(F', B_s(\bar y))
	\ \ 
	\text{for $\Leb^1$-a.e. $s\in (0,r)$.}
\end{equation}
Before proving \eqref{z18} let us illustrate how to use it to conclude the proof. First of all, notice that \eqref{z9} gives $\nu\geq |D\chi_F|$; if we apply \eqref{z18} with $F'=F$ we conclude that $\nu=|D\chi_F|$ locally and then globally, achieving the conclusion (ii) in the statement. The validity of the local minimality condition (i) follows combining the identification $\nu=|D\chi_F|$ with \eqref{z18}, letting $s\uparrow r$.

Let us pass to the proof of \eqref{z18}.
We fix a parameter $s\in (0,r)$ with $\nu(\partial B_s(\bar y))=0$, $\Per(F',\partial B_s(\bar y))=0$ and set
\begin{equation}
	\tilde{E}_i^s:=\left(E_i'\cap B_s(\overline{x}_i)\right)\cup \left(E_i\setminus B_s(\overline{x}_i)\right).
\end{equation}
Using the locality of the perimeter (see \cite{Am01, Am02}) and the $\lambda_i$-minimality of $E_i$ (notice that $R_i\ge r$ for $i$ big enough), we get
\begin{align*}
	\Per(E_i, \overline{B}_s(\bar{x}_i)) = & \Per(E_i, B_r(\bar{x_i}))-\Per(E_i,B_r(\bar{x_i})\setminus\overline{B}_s(\bar{x}_i))\\
	\le &  \lambda_i \Per(\tilde{E}^s_i, B_r(\bar{x}_i))-\Per(E_i,B_r(\bar{x}_i)\setminus\overline{B}_s(\bar{x}_i))\\
	=&\lambda_i \Per(\tilde{E}^s_i, B_s(\bar{x}_i))+\lambda_i\Per(\tilde{E}^s_i, \partial B_s(\bar{x}_i))\\
	&+\lambda_i \Per(\tilde{E}^s_i,B_r(\bar{x}_i)\setminus\overline{B}_s(\bar{x}_i))-\Per(E_i,B_r(\bar{x}_i)\setminus\overline{B}_s(\bar{x}_i))\\
	=&\lambda_i \Per(E_i', B_s(\bar{x}_i))+\lambda_i\Per(\tilde{E}^s_i, \partial B_s(\bar{x}_i))
	+(\lambda_i-1)\Per(E_i,B_r(\bar{x}_i)\setminus\overline{B}_s(\bar{x}_i)).
\end{align*}
Observe that, taking the limit as $i\to \infty$, thanks to our choice of $s$, it holds that: $(\lambda_i-1)\Per(E_i,B_r(\bar{x}_i)\setminus\overline{B}_s(\bar{x}_i))\to 0$, 
$\Per(E_i, \overline{B}_s(\bar{x}_i))\to \nu(B_s(\bar y))$ and eventually $\lambda_i \Per(E_i', B_s(\bar{x}_i))\to \Per(F',B_s(\bar y))$, since $\chi_{E_i'}\to \chi_{F'\cap B_R(y)}$ in $\BV$ energy
and therefore \autoref{cor:weak convergence total variations} applies.
It remains only to prove that
\begin{equation}\label{z19}
	\liminf_{i\to \infty}\Per(\tilde{E}_i^s, \partial B_s(\bar{x}_i))=0,
	\qquad
	\text{for}\ \Leb^1\text{-a.e.}\ s\in (0,r).
\end{equation}
Applying \eqref{z21} of
\autoref{lem:lei} below with $f=\chi_{E_i'}-\chi_{E_i}$ we get
\begin{equation*}
	\Per(\tilde{E}^s_i,X\setminus \overline{B}_s(\bar x_i))\le \int_{X_i}|\chi_{E_i'}-\chi_{E_i}| \di |D\chi_{B_s(\bar x_i)}|+\Per(E_i, X\setminus \overline{B}_s(\bar x_i))\quad\text{for $\Leb^1$-a.e. $s\in(0,r)$},
\end{equation*}
that, together with the strong locality the perimeter, yields
\begin{equation}\label{eq:estimperl1}
    	P(\tilde{E}^s_i, \partial B_s(\bar x_i))\le \int_{X_i}|\chi_{E_i'}-\chi_{E_i}| \di |D\chi_{B_s(\bar x_i)}|,\quad\text{for $\Leb^1$-a.e. $s\in(0,r)$}.
\end{equation}
Using Fatou's lemma, \eqref{eq:estimperl1}, the local version of the coarea formula of \autoref{cor:loccoarea} and eventually \autoref{lemma:convergence intersection} to prove that $\chi_{E_i'}-\chi_{E_i}\to \chi_F-\chi_{F'}$ in $L^1$-strong, we conclude that
\begin{align*}
   \int_0^r \liminf_{i\to \infty}	\Per(\tilde{E}^s_i, \partial B_s(\bar x_i))\di s
   &\le \liminf_{i\to \infty}\int_0^r	\Per(\tilde{E}_i^s, \partial B_s(\bar x_i)) \di s\\
   &\le \liminf_{i\to \infty}\int_0^r\int_{X_i}|\chi_{E_i'}-\chi_{E_i}| \di |D\chi_{B_s(\bar x_i)}|\\
   &=\liminf_{i\to \infty} \int_{B_r(\bar x_i)} |\chi_{E_i'}-\chi_{E_i}| \di \meas_i=0,
\end{align*}
therefore yielding \eqref{z19}.
\end{proof}

\begin{lemma}[Leibniz rule in $\BV$]\label{lem:lei}
	Let $(X,\dist,\meas)$ be an $\RCD(K,\infty)$ m.m.s. and let $x\in X$. 
	For any $f\in \BV(X,\dist,\meas)\cap L^{\infty}(X,\meas)$ and $\Leb^1$-a.e. $r\in (0,+\infty)$ it holds
	\begin{equation}\label{eq:chain rule}
	\bigl|D \bigl(f\chi_{B_r(x)}\bigr)\bigr|(X)\le \int_X |f| \di |D\chi_{B_r(x)}|+|Df|( B_r(x))
	\end{equation}	
	and therefore locality gives
	\begin{equation}\label{z21}
		\bigl|D\bigl(f\chi_{B_r(x)}\bigr)\bigr|(X\setminus B_r(x))\le \int_X |f| \di |D\chi_{B_r(x)}|,\quad\text{for $\Leb^1$-a.e. $r\in(0,+\infty)$}.
	\end{equation}
\end{lemma}

\begin{proof}
Let us begin observing that the stated conclusion makes sense since, in view of the coarea formula \autoref{thm:coarea}, $\int\abs{f}\di |D\chi_{B_r(x)}|$ is well defined for $\Leb^1$-a.e. $r\in(0,+\infty)$.\\
We divide the proof into two intermediate steps. In the first one we are going to prove that \eqref{eq:chain rule} holds true under the
assumption $f\in\Lipb(X,\dist)$. In the second one we prove the sought inequality passing to the limit the inequalities for regularized functions that we obtained previously.

\textbf{Step 1.}
More generally in this step we are going to prove, arguing by regularization on $g$, that, for any $f\in \Lipb(X,\dist)$ and for any nonnegative function $g\in\BV(X,\dist,\meas)\cap L^{\infty}(X,\meas)$, it holds 
\begin{equation}\label{eq:leibdoubreg}
\abs{D\left(fg\right)}(X)\le \int_X\abs{f}\di\abs{Dg}+\int_X\abs{g}\abs{\nabla f}\di\meas.
\end{equation}
Observe that, if $g\in\Lipb(X,\dist)$ then \eqref{eq:leibdoubreg} follows from the Leibniz rule. Hence, by the $L^{\infty}-\Lip$ regularization of the heat semigroup it follows that, for any $t>0$,
\begin{equation}\label{eq:leibheat}
\bigl|D\bigl(fP_tg\bigr)\bigr|(X)\le\int_X \abs{f}\abs{\nabla P_tg}\di\meas+\int_XP_tg\abs{\nabla f}\di\meas.
\end{equation}
The convergence of $P_tg$ to $g$ in $L^1(X,\meas)$ as $t\to 0$, the lower semicontinuity of the total variation and the Bakry-\'Emery contraction estimate allow us to pass to the $\liminf$ at the left hand-side and to the limit at the right hand-side in \eqref{eq:leibheat} to get \eqref{eq:leibdoubreg} (see also the proof of the second step for further details on the limiting procedure).

\textbf{Step 2.}
It follows from what we just proved and from the $L^{\infty}-\Lip$ regularization property of the heat flow on $\RCD(K,\infty)$ m.m. spaces that, for any $t>0$,
	\begin{equation}\label{eq:regleibniz}
	\bigl|D\bigl(P_tf\chi_{B_r(x)}\bigr)\bigr|(X)\le \int_X |P_tf| \di |D\chi_{B_r(x)}|+|DP_tf|(\overline{B}_r(x))\quad\text{for $\Leb^1$-a.e. $r\in(0,+\infty)$}.
	\end{equation}	
Next we observe that $P_tf\chi_{B_r(x)}\to f\chi_{B_r(x)}$ in $L^1(X,\meas)$ as $t\to 0^+$ and therefore, by the lower semicontinuity of the total variation w.r.t. $L^1$ convergence it holds
\begin{equation}\label{eq:lscpt}
\bigl|D\bigl(f\chi_{B_r(x)}\bigr)\bigr|(X)\le\liminf_{t\to 0^+}\bigl|D(P_tf\chi_{B_r(x)})\bigr|(X).
\end{equation}
Furthermore, the $L^1(X,\meas)$ convergence of $P_tf$ to $f$ and the coarea formula \autoref{thm:coarea} grant that we can find a sequence $t_i\downarrow 0$ in such a way that $P_{t_i}f$ converges in $L^1(X,|D\chi_{B_r(x)}|)$ to $f$ for $\Leb^1$-a.e. $r\in(0,+\infty)$.
Eventually, let us observe that, due to the Bakry-\'Emery contraction estimate \eqref{eq:BE1},
\begin{equation*}
\limsup_{t\to 0^+}\abs{DP_tf}(\overline{B}_r(x))\le 
\limsup_{t\to 0^+}e^{-Kt}P_t^{*}\abs{Df}(\overline{B}_r(x))\leq \abs{Df}(\overline{B}_r(x)),\quad\text{$\forall r\in (0,+\infty)$.}
\end{equation*}
Passing to the $\liminf$ as $t_i\downarrow 0$ at the left hand-side of \eqref{eq:regleibniz} taking into account \eqref{eq:lscpt} and to the limit at the right hand-side taking into account what we observed above, we get the sought estimate \eqref{eq:chain rule}.
\end{proof}

\section{Tangents to sets of finite perimeter in $\RCD(K,N)$ spaces}\label{sec:tangentsRCD}

In this section we study the structure of blow-ups of sets of finite perimeter over $\RCD(K,N)$ metric measure spaces. Inspired by the Euclidean theory developed  by De Giorgi in the pioneering papers \cite{DeGiorgi54,DeGiorgi55}, this can be seen as a first step in a program aimed at understanding 
the fine structure of sets of finite perimeter. 

Before then stating the main results we introduce a definition of tangent for sets of finite perimeter in this abstract setting.

\begin{definition}[Tangents to a set of finite perimeter]\label{def:tan}
	Let $(X,\dist,\meas)$ be an $\RCD(K,N)$ m.m.s., $x\in X$ and let $E\subset X$ be a set of locally finite perimeter. 
	We denote by $\Tan_x(X,\dist,\meas,E)$ the collection of quintuples $(Y,\varrho,\mu,y, F)$ satisfying the following two properties:
	\begin{itemize}
		\item[(a)] $(Y,\varrho,\mu,y)\in\Tan_x(X,\dist,\meas)$ and $r_i\downarrow 0$ are such that the rescaled spaces $(X,r_i^{-1}\dist,\meas_x^{r_i},x)$ converge to $(Y,\varrho,\mu,y)$ in the pointed measured Gromov-Hausdorff topology;
		\item[(b)] $F$ is a set of locally finite perimeter in $Y$ with $\mu(F)>0$ and, if $r_i$ are as in (a), then the sequence $f_i=\chi_E$ converges in $L^1_\loc$ to $\chi_F$ according to \autoref{def:L1 convergence of sets}.
	\end{itemize}
\end{definition}

It is clear that the following locality property of tangents holds:
\begin{equation}\label{eq:localitytangents}
\meas\bigl(A\cap (E\Delta F)\bigr)=0\qquad\Longrightarrow\qquad\Tan_x(X,\dist,\meas,E)=\Tan_x(X,\dist,\meas,F)\qquad\forall x\in A. 
\end{equation}
whenever $E,\,F$ are sets of locally finite perimeter and $A\subset X$ is open.

We are ready to state the main results of this section.

\begin{theorem}\label{thm:rigidtangent}
	Let $(X,\dist,\meas)$ be an $\RCD(K,N)$ m.m.s. and let $E\subset X$ be a set 
	of locally finite perimeter. For $|D\chi_E|$-a.e. $x\in X$ the set $\Tan_x(X,\dist,\meas,E)$ is not empty and for all $(Y,\varrho,\mu,y,F)\in \Tan_x(X,\dist,\meas,E)$, one has
	\begin{equation}\label{eq:rigidblowup}
	|\nabla P_s \chi_F| \mu= P_s^*|D\chi_F|\qquad\forall s>0,
	\end{equation}
	where $P_s=P_s^Y$ is the heat semigroup relative to $(Y,\varrho,\mu)$. In particular, for all $t\geq 0$, 
	all functions $f=P_t\chi_F$ satisfy
	\begin{equation*}
     |\nabla P_s  f| = P_s|\nabla f|\qquad\text{$\mu$-a.e. in $Y$, for all $s>0$.}
	\end{equation*}
	Moreover, for each $x\in X$ as above there exists a pointed m.m.s. $(Z,\dist_Z,\meas_Z,\bar z)$ such that 
	\begin{equation}
	(Y,\varrho,\mu,y,F)=
	\left((Z\times\setR),\dist_Z\times\dist_{eucl},\meas_Z\times\Leb^1,(\bar{z},0),\set{t>0}\right),
	\end{equation} 
	where we denoted by $t$ the coordinate of the Euclidean factor in $Z\times\setR$.
	Furthermore:
	\begin{itemize}
		\item[(i)] if $N\ge 2$ then $Z$ is an $\RCD(0,N-1)$ m.m.s.;
		\item[(ii)] if $N\in [1,2)$ then $Z$ is a point.
	\end{itemize}
\end{theorem}

A suitable version of the iterated tangent theorem by Preiss (see \autoref{th: tangent of tangent}) implies also the following.

\begin{theorem}\label{thm:tangenthalfspace}
	Let $(X,\dist,\meas)$ be an $\RCD(K,N)$ m.m.s. and let $E\subset X$ be a set of locally finite perimeter. 
	Then $E$ admits a Euclidean half-space as tangent at $x$ for $\abs{D\chi_E}$-a.e. $x\in X$, that is to say
	\begin{equation*}
	\left(\setR^k,\dist_{eucl},c_k\Leb^k,0^k,\set{x_k>0}\right)\in \Tan_x(X,\dist,\meas,E),\qquad\text{for some $k\in [1,N]$.}
	\end{equation*} 
\end{theorem} 

\begin{proof}[Proof of \autoref{thm:tangenthalfspace}]
	We claim that the stated conclusion holds true at all points $x\in X$ such that both the iterated tangent property of \autoref{th: tangent of tangent} and the rigidity property stated in \autoref{thm:rigidtangent} are satisfied (observe that $\abs{D\chi_E}$-a.e. point satisfies these two properties).\\ 
	Indeed, if $(Y,\varrho,\mu,y,F)\in \Tan_x(X,\dist,\meas,E)$, combining \autoref{thm:rigidtangent} with \autoref{maintheorem}, we can say that $(Y,\varrho,\mu)$ is isomorphic to $Z\times\setR$ for some $\RCD(0,N-1)$ m.m.s. $(Z,\dist_Z,\meas_Z)$. Furthermore, another consequence of \autoref{maintheorem} is that $F=\set{t>t_0}$ for some $t_0\in\setR$, where we denoted by $t$ the coordinate on the Euclidean factor of $Y$. Up to
	a translation we can also assume that $y=(\overline{z},0)$ for some $\overline{z}\in Z$.\\
	We go on observing that, if $i:Z\to Y$ denotes the canonical inclusion $i(z):=(z,0)$, it holds $\abs{D\chi_F}=i_{\sharp}\meas_Z$ and, for this reason, we shall identify in the sequel $\abs{D\chi_F}$ and $\meas_Z$. Moreover, it is easy to check that, if 
	$(W,\dist_W,\meas_W,\bar w)\in \Tan_z(Z,\dist_Z,\meas_Z)$, then 
	$$(W\times\setR,\dist_W\times\dist_{eucl},\meas_W\times\Leb^1,(\bar w,0),\set{t>0})\in \Tan_{(z,0)}(Y,\varrho,\mu,F).$$ 
	The sought conclusion can now be obtained choosing $z$ to be a regular point of $(Z,\dist_Z,\meas_Z)$ (recall that $\meas_Z$-a.e. point of $Z$ 
	is regular), so that $W$ is a Euclidean space of dimension $k\in [0, N-1]$ and applying \autoref{th: tangent of tangent} to conclude that 
	$$(W\times\setR,\dist_W\times\dist_{eucl},\meas_W\times\Leb^1,(\bar w,0),\set{t>0})\in \Tan_x(X,\dist,\meas,E).$$     
\end{proof}

In the case when the ambient space is non collapsed (see \autoref{def:ncRCD}) one can improve the conclusion of \autoref{thm:tangenthalfspace} above.

\begin{corollary}\label{cor:imprnoncollapsed}
	Let $(X,\dist,\meas)$ be a $\ncRCD(K,N)$ m.m.s. and let $E\subset X$ a set of locally finite perimeter. Then $|D\chi_E|$ is concentrated 
	on the $N$-dimensional regular set of $X$ and, in addition,
	\begin{equation*}
	\Tan_x(X,\dist,\meas,E)=\left\lbrace (\setR^N, \dist_{eucl}, c_N\Leb^N,0^N, \set{x_N>0}) \right\rbrace
	\qquad
	\text{for}\ \abs{D\chi_E}\text{-a.e.}\ x\in X.
	\end{equation*}
\end{corollary}

\begin{proof}[Proof \autoref{cor:imprnoncollapsed}]
	Observe that \autoref{thm:tangenthalfspace} gives that $|D\chi_E|$ is concentrated on the set of points where at least one metric measured tangent is Euclidean. In the case when the ambient space is a non collapsed $\RCD(K,N)$ m.m.s. this set coincides with the regular set $\mathcal{R}_N$ (see \autoref{thm:improvedregularity} and \cite{DePhilippisGigli18}).
	To conclude, it remains to invoke \autoref{thm:rigidtangent} and to observe that, in the case of a Euclidean ambient space \eqref{eq:rigidblowup} easily yields that $F$ is a half-space (even without appealing to \autoref{maintheorem}).
\end{proof}

\begin{remark}
Let us remark that, without the non collapsing assumption, it is not directly possible to improve \autoref{thm:tangenthalfspace} to \autoref{cor:imprnoncollapsed}, as the example mentioned in \autoref{rm:menguy} shows. We leave the problem of the concentration of the perimeter measure on the regular set in the general case to a forthcoming work.
\end{remark}

Given the statement of \autoref{cor:imprnoncollapsed}, it sounds natural to introduce the following.

\begin{definition}
	Let $(X,\dist,\meas)$ be $\ncRCD(K,N)$ metric measure space and let $E\subset X$ be of locally finite perimeter. Then we define the reduced boundary $\mathcal{F}E$ of $E$ by
	\begin{equation}\label{eq:bellafrontiera}
	\mathcal{F}E:=\left\lbrace x\in\supp\abs{D\chi_E}:\quad	\Tan_x(X,\dist,\meas,E)=\left\lbrace (\setR^N, \dist_{eucl}, c_N\Leb^N,0^N, \set{x_N>0}) \right\rbrace \right\rbrace. 
	\end{equation}
\end{definition}

It is easily seen that ${\mathcal F}E$ is contained in the essential boundary $\partial^*E$, namely the complement
of the sets of density and rarefaction. In the more general context of PI spaces it is known after \cite{Am02} that $|D\chi_E|$ is
representable as $\theta\mathcal{S}\res\partial^* E$ for some density $\theta$, where $\mathcal{S}$ denotes the
measure induced by the gauge function $\zeta(B_r(x))=\meas(B_r(x))/r$ with Carath\'eodory's construction. The following
result refines this representation formula in the setting of non collapsed $\RCD(K,N)$ spaces.

\begin{corollary}\label{cor:identificationdensity}
	Let $(X,\dist,\meas)$ be a $\ncRCD(K,N)$ m.m.s. and let $E\subset X$ be a set of locally finite perimeter. Then 
	\begin{equation*}
	\Per(E,\cdot)=\mathcal{S}^{N-1}\res\mathcal{F}E,
	\end{equation*}
	where we denote by $\mathcal{S}^{\alpha}$ the $\alpha$-dimensional spherical Hausdorff measure.
\end{corollary}

\begin{proof}
	We claim that
	\begin{equation}\label{eq:constdensity}
	\lim_{r\downarrow 0}\sup_{y\in B_s(x),\ s\le r}\frac{\Per(E,B_s(y))}{\frac{\omega_{N-1}}{2^{N-1
		}}\diam^{N-1}(B_s(y))}=1,
	\end{equation}
	for any $x\in \mathcal{F}E$. 
	Let us write
	\begin{align*}
		\frac{\Per(E,B_r(x))}{\frac{\omega_{N-1}}{2^{N-1
			}}\diam^{N-1}(B_r(x))}= &
		\frac{r\Per(E,B_r(x))}{\omega_{N-1}C(x,r)}\ \frac{C(x,r)}{\meas(B_r(x))}\ \frac{\meas(B_r(x))}{r^N}\ \left(\frac{2r}{\diam(B_r(x))}\right)^{N-1} \\
		=& \frac{\Per^{X_r}(E, B_1(x))}{\omega_{N-1}} \frac{1}{\meas_x^r(B_1(x))}  \frac{\meas(B_r(x))}{r^N}\ \left(\frac{2r}{\diam(B_r(x))}\right)^{N-1},
	\end{align*}
	where $C(x,r)$ is the constant in \eqref{eq:defC(x,r)}.
	Using the very definition of the reduced boundary $\mathcal{F}E$, the continuity of the diameter w.r.t. GH-convergence, the weak convergence of the perimeters proved in \autoref{cor:tangents} and the fact ${\meas(B_r(x))}/{r^N}\to\omega_N$ as $r\to 0^+$ when $x\in X$ is regular (see \cite[Corollary 1.7]{DePhilippisGigli18}) we get
	\begin{equation}\label{z0}
		\liminf_{r\downarrow 0}\sup_{y\in B_s(x),\ s\le r}\frac{\Per(E,B_s(y))}{\frac{\omega_{N-1}}{2^{N-1
			}}\diam^{N-1}(B_s(x))}\ge \liminf_{r\downarrow 0} \frac{\Per(E,B_r(x))}{\frac{\omega_{N-1}}{2^{N-1
		}}\diam^{N-1}(B_r(x))}=1.
	\end{equation}
	Let us now prove the converse inequality. In order to do so it suffices to show that, for any sequence of radii $r_i\to 0$ and points $x_i\in B_{r_i}(x)$, it holds
	\begin{equation*}
		\limsup_{i\to \infty}\frac{\Per(E,B_{r_i}(x_i))}{\frac{\omega_{N-1}}{2^{N-1
			}}\diam^{N-1}(B_{r_i}(x_i))}\le 1.
	\end{equation*}
	We consider the sequence $(X,r_i^{-1}, C(x,r_s)^{-1}\meas, x)$, that converges to $(\setR^N, \dist_{eucl}, c_N\Leb^N,0^N)$, and we assume $x_i\to z\in \overline B_1(0^N)$. Arguing as in the proof of \eqref{z0}, we write
	\begin{align*}
	\frac{\Per(E,B_{r_i}(x_i))}{\frac{\omega_{N-1}}{2^{N-1
		}}\diam^{N-1}(B_{r_i}(x_i))}= &
	\frac{r_i\Per(E,B_{r_i}(x_i))}{\omega_{N-1}C(x,r_i)}\ \frac{C(x,r_i)}{\meas(B_{r_i}(x))}\ \frac{\meas(B_{r_i}(x))}{r_i^N}\ \left(\frac{2r_i}{\diam(B_{r_i}(x))}\right)^{N-1} \\
	=& \frac{\Per^{X_{r_i}}(E, B_1(x_i))}{\omega_{N-1}} \frac{1}{\meas_x^r(B_1(x))}  \frac{\meas(B_{r_i}(x))}{r_i^N}\ \left(\frac{2r_i}{\diam(B_{r_i}(x))}\right)^{N-1},
	\end{align*}
	and deduce
	\begin{equation*}
		\lim_{i\to \infty}\frac{\Per(E,B_{r_i}(x_i))}{\frac{\omega_{N-1}}{2^{N-1
			}}\diam^{N-1}(B_{r_i}(x_i))}
		=\frac{1}{\omega_{N-1}} \haus^{N-1}(B_1(z)\cap\set{x_N=0})\le 1.
	\end{equation*}

	 The stated conclusion follows now from \eqref{eq:constdensity} applying \cite[Theorem 5]{Magnani15} and checking that $(X,\dist)$ is \textit{diametrically regular}, i.e.
	for every $x\in X$ and $R>0$ there exists $\delta_{x,R}>0$ such that $(0,\delta_{x,R})\ni t\mapsto \diam(B_t(y))$ is continuous for every $y\in \overline B_R(x)$ (see the discussion above Theorem 5 in \cite{Magnani15}). To this aim we prove a stronger claim: for any $x\in X$ it holds
	\begin{equation*}
		|\diam(B_t(x))-\diam(B_s(x))|\le 2|t-s|
		\qquad
		\forall t,s\ge 0.
	\end{equation*}
	This property is a consequence of the fact that $X$ is a geodesic metric measure space. Indeed, assuming $t>s$ without loss of generality, for any $y\in B_t(x)$ we can find $y'\in B_s(x)$ realizing $\dist(y,y')\le |t-s|$.
	In particular, for any $\eps>0$, we can consider $y_1, y_2\in B_t(x)$ such that $\dist(y_1,y_2)\ge \diam(B_t(x))-\eps$ and, taking $y_1', y_2'\in B_s(x)$ as above, we conclude
	\begin{equation*}
		\diam(B_t(x))\le \dist(y_1,y_2)+\eps \le 2|t-s|+\dist(y_1',y_2')+\eps\le 2|t-s|+\eps+\diam(B_s(x)),
	\end{equation*}
    that implies the sought conclusion letting $\eps\downarrow 0$.
\end{proof}

The rest of this section is devoted to the proof \autoref{thm:rigidtangent}. 
First, we are going to prove that tangents are non empty almost everywhere with respect to the perimeter measure, as a consequence of the compactness results developed in \autoref{sec:compstab} and \autoref{prop:foundamental perimeter estimates}.
Then, we will prove that they are rigid, in a suitable sense. This rigidity property will be achieved building mainly on two ingredients: lower semicontinuity and locality of the perimeter and the Bakry-\'Emery inequality, together with the characterization of its equality cases we obtained in \autoref{sec:rigidityBE}.

We start stating an asymptotic minimality result that stems from the lower semicontinuity of the perimeter.
It has been proved,  in a slightly weaker form (namely with a smaller class of competitors $E'$), first in \cite{Am01} under Ahlfors regularity assumption and then,
in \cite{Am02}, for the general case. We refer to the very recent \cite[Theorem 6.1]{Shanmugalingam18} for the present form. 
The basic idea originates, to the authors' knowledge, in the work of
Fleming \cite{Fleming} (see also \cite{White,Cheeger99} for variants of this idea in different contexts).

\begin{proposition}[Asymptotic minimality and doubling]\label{prop:Asymptotic minimality and doubling}
	Let $(X,\dist,\meas)$ be an $\RCD(K,N)$ m.m.s. and $E\subset X$ be a set of locally finite perimeter. For $|D\chi_E|$-a.e. $x\in X$ there exist $r_x>0$ and
	$\omega_x(r):(0,r_x)\to [0,+\infty)$ such that $\omega_x(r)\to 0$ as $r\to 0^+$ and
	\begin{equation}\label{eq: quasiminimo}
	\Per(E,B_r(x))\leq (1+\omega_x(r))\Per(E',B_r(x))
	\end{equation}
	whenever $E\Delta E'\Subset B_r(x)$. In addition,
	\begin{equation}\label{eq:asydou}
	\limsup_{r\to 0^+}\frac{|D\chi_E|(B_{2r}(x))}{|D\chi_E|(B_r(x))}<+\infty.
	\end{equation}
\end{proposition} 

Also the following density estimates are important to prove that tangents are almost everywhere non empty. We refer again to \cite{Am01,Am02} for its proof.

\begin{proposition}\label{prop:foundamental perimeter estimates}
	Let $(X,\dist,\meas)$ be an $\RCD(K,N)$ m.m.s. and $E\subset X$ be a set of locally finite perimeter. For $|D\chi_E|$-a.e. $x\in X$ it holds
	\begin{equation}\label{v1}
	  0<\liminf_{r\to 0^+}\frac{r|D\chi_E|(B_r(x))}{\meas(B_r(x))}
	  \le \limsup_{r\to 0^+}\frac{r|D\chi_E|(B_r(x))}{\meas(B_r(x))}
	  <+\infty,
	\end{equation}
	and
	\begin{equation}\label{v2}
		\liminf_{r\to 0^+}\min \left\lbrace \frac{\meas(E\cap B_r(x))}{\meas(B(x,r))};\frac{\meas(E^c\cap B_r(x))}{\meas(B(x,r))}
			\right\rbrace>0.
	\end{equation}
\end{proposition}

\begin{corollary}\label{cor:tangents} Let $(X,\dist,\meas)$ be an $\RCD(K,N)$ m.m.s. and let $E\subset X$ be a set of locally finite perimeter. Then, 
for $|D\chi_E|$-a.e. $x\in X$ one has $\Tan_x(X,\dist,\meas,E)\neq\emptyset$ and, if $(Y,\varrho,\mu, y, F)$ is as in \autoref{def:tan}, the following properties
hold true: 
	\begin{itemize}
		\item[(a)] $F$ is an entire minimizer of the perimeter (relative to $(Y,\varrho,\tau)$), i.e.
		$$|D\chi_F|(B_r(y))\leq |D\chi_{F'}|(B_r(y))\qquad\text{ whenever $F\Delta F'\Subset B_r(y)\Subset Y$;}$$
		\item[(b)] realizing the convergence in a proper metric space $(Z,\dist_Z)$, the perimeters $|D^i\chi_E|$ relative 
		to the rescaled spaces in condition (a) of \autoref{def:tan}
		weakly converge, in duality with $\Cbs(Z)$, to $|D\chi_F|$.
	\end{itemize}
\end{corollary}
\begin{proof} 
	Let us consider $x\in X$ such that the statements of \autoref{prop:Asymptotic minimality and doubling} and \autoref{prop:foundamental perimeter estimates} hold true and a sequence of radii $r_i\to0$ such that $(X, r^{-1}\dist,\mu_x^{r},x )\to (Y,\varrho,\mu,y)$ in the pmGH topology. Thanks to \eqref{v1} and \autoref{cor:perimeters compactness} with 
$\chi_{E_i}=\chi_{E}$, possibly extracting a subsequence we can assume that there exists a set $F\subset Y$ with locally finite perimeter such that
$\chi_E\to \chi_F$ in $L^1_{\loc}$. Note that $\mu(F)>0$ thanks to \eqref{v2}. This implies that $(Y,\varrho,\mu,y,F)\in \Tan(E,x)$. To achieve (a) and (b) 
it is enough to apply \autoref{prop: sequence of lambda minimizers}, recalling \eqref{eq: quasiminimo}.
\end{proof}

The next key result to prove \autoref{thm:rigidtangent} is \autoref{prop:techincalheat}. Before stating and proving it we need a technical lemma.

\begin{lemma}\label{lemma:Gammaliminfforheat}
Let $(X_n,\dist_n,\meas_n)$ be $\RCD(K,N)$ m.m. spaces $mGH$ converging to $(Y,\varrho,\mu)$ and assume that the convergence is realized into a proper metric space $(Z,\dist_Z)$. Let $\eta_n$, $\eta$ be nonnegative Borel measures giving finite mass to bounded sets, such that $\supp\eta_n\subset\supp\meas_n$, $\supp\eta\subset\supp\mu$ and $\eta_n$ weakly converge to $\eta$ in duality with $\Cbs(Z)$. Then
\begin{equation}\label{eq:lscheat}
P^Y_t\eta(x)\le\liminf_{n\to\infty}P^n_t\eta_n(x_n),\quad
\text{for any $t>0$ and for any $\supp\meas_n\ni x_n\to x\in \supp\mu$.}
\end{equation}
\end{lemma}

\begin{proof}
In \cite[Theorem 3.3]{AmbrosioHondaTewodrose17}, building on \cite{GigliMondinoSavare15}, 
it is proved that, denoting by $p^n$ and $p^Y$ the heat kernels of $(X_n,\dist_n,\meas_n)$ and $(Y,\varrho,\mu)$ respectively, it holds
\begin{equation}\label{eq:convheat}
\lim_{n\to\infty}p^n_t(x_n,y_n,t)=p^Y_t(x,y),\quad\text{for any $t>0$},
\end{equation} 
whenever $\supp\meas_n\times\supp\meas_n\ni(x_n,y_n)\to (x,y)\in\supp\mu\times\mu$.
Since 
\begin{equation*}
P^Y_t\eta(x)=\int p_t^Y(x,y)\di\eta(y)\quad\text{and }\quad P^n_t\eta_n(x_n)=\int p_t^n(x_n,y)\di\eta(y),
\end{equation*}
the validity of \eqref{eq:lscheat} follows from \autoref{lemma:lowersc} and Fatou's lemma
with the obvious choice for the weakly convergent sequence of measures and $f_n(\cdot):=p_t^n(x_n,\cdot)$, $f:=p_t(x,\cdot)$, which satisfy the lower semicontinuity condition \eqref{eq:Gammaliminf} in view of \eqref{eq:convheat}.
\end{proof}

\begin{proposition}\label{prop:techincalheat}
	Let $E\subset X$ be a set of finite perimeter and let $(Y,\varrho,\mu,y,F)\in \Tan_x(X,\dist,\meas,E)$ for some $x\in X$. Let $r_i\downarrow 0$ be a sequence of radii realizing the convergence in \autoref{def:tan}. Then
	\begin{equation*}
	|\nabla^iP_t^i\chi_E|\meas_i\weakto |\nabla^YP_t^Y\chi_F|\mu
	\qquad\text{in duality with $\Cbs(Z)$, for any $t>0$.}
	\end{equation*}  
\end{proposition}

\begin{proof}
	We wish to implement a strategy very similar to the one adopted in the proof of \autoref{prop:stabilityandcompH12} (see \cite[Theorem 5.4, Corollary 5.5]{AmbrosioHonda} and \cite{GigliMondinoSavare15}).
	
	Let us begin proving that, for any $\supp\meas_i\ni x_i\to x\in\supp\mu$ and for any $t>0$, it holds
	\begin{equation}\label{eq:proplim}
	\lim_{i\to\infty}P^i_t\chi_E(x_i)=P^Y_t\chi_F(x).
	\end{equation} 
	To this aim we first observe that, by the very definition of tangent, it holds that $\chi_E\meas_n\weakto\chi_F\mu$  in duality with $\Cbs(Z)$ and therefore \autoref{lemma:Gammaliminfforheat} yields
	\begin{equation}\label{eq:propliminf}
	P^Y_t\chi_F(x)\le\liminf_{i\to\infty}P^i_t\chi_E(x_i).
	\end{equation} 
	Moreover, since $(1-\chi_E)\meas_n\weakto(1-\chi_F)\mu$ in duality with $\Cbs(Z)$, applying \autoref{lemma:Gammaliminfforheat} once more and with a simple algebraic manipulation, we obtain
	\begin{equation}\label{eq:proplimsup}
	\limsup_{i\to\infty}P^i_t\chi_E(x_i)\le P^Y_t\chi_F(x).
	\end{equation}
	Combining \eqref{eq:propliminf} with \eqref{eq:proplimsup} we obtain \eqref{eq:proplim}.
	
	Let us proceed observing that, in view of the quantitative form of the $L^{\infty}$-$\Lip$ regularization on $\RCD(K,\infty)$ spaces 
	provided by \eqref{eq:infinitolip}, for any $t>0$ the functions $P^i_t\chi_E$ and $P^Y_t\chi_F$ are uniformly Lipschitz. 
	
	Fix now reference points $y\in Y$ and $X_i\ni x_i\to y$. Building upon \cite[Lemma 3.1]{MondinoNaber14}, for any $R>0$ it is possible to find Lipschitz cut-off functions $\eta_R:Y\to[0,1]$, $\eta^i_R:X_i\to[0,1]$ such that $\supp\eta_R\subset B^Y_{2R}(y)$, $\supp\eta^i_R\subset B^i_{2R}(x_i)$, $\eta_R|_{B^Y_R(y)}\equiv 1$, $\eta_R^i|_{B^i_R(x_i)}\equiv 1$, uniformly Lipschitz, with uniformly bounded laplacians and such that $\eta^i_R$ converge to $\eta_R$ both pointwise and $L^2$-strongly. We remark indeed that, in view of \cite[Proposition 3.2]{AmbrosioHondaTewodrose17}, pointwise and $L^2$-strong convergence are equivalent for uniformly bounded, uniformly continuous and uniformly boundedly supported functions.
	Let us observe that, if we are able to prove that
	\begin{equation*}
	f_i:=\eta_R^iP^i_t\chi_E\to \eta_RP^Y_t\chi_F=:f\qquad\text{strongly in $H^{1,2}$ for all $R>0$,}
	\end{equation*}
	the conclusion will follow from the locality of the minimal weak upper gradient and \autoref{thm:localizedconv}, which grants the $L^1$-strong convergence of $|\nabla^i\bigl(\eta_R^iP^i_t\chi_E\bigr)|^2$ to $|\nabla^Y\eta_RP^Y_t\chi_F|^2$ (that we can improve to $L^1$-strong convergence of 
	$|\nabla^i\bigl(\eta_R^iP^i_t\chi_E\bigr)|$ to $|\nabla^Y\eta_RP^Y_t\chi_F|$ in view of the uniform Lipschitz bounds and of \autoref{prop:stability}).

	In order to prove the above claimed convergence, we begin observing that $f_i$ converge pointwise to $f$ by \eqref{eq:proplim} and the very construction of the family of cut-off functions $\eta_R^i$. Therefore, taking into account the uniform Lipschitz bounds, the uniform boundedness and the uniform bounds on the supports, $f_i\to f$ strongly in $L^2$ by \cite[Proposition 3.2]{AmbrosioHondaTewodrose17}. To improve the convergence from $L^2$-strong to $H^{1,2}$-strong we wish to apply \autoref{prop:stabilityandcompH12}. In order to do so, it remains to prove that $\Delta f_i$ are uniformly bounded in $L^2$. To this aim we compute
	\begin{equation}\label{eq:lapappr}
	\Delta f_i=\Delta\eta^i_RP_t^i\chi_E+2\nabla\eta^i_R\cdot\nabla P_t^i\chi_E+\eta^i_R\Delta P_t^i\chi_E
	\end{equation} 
	and observe that all the terms at the right hand side in \eqref{eq:lapappr} are uniformly bounded in $L^2$ in view of the uniform $L^{\infty}$ bounds on values, minimal weak upper gradients and laplacians of the cut-off functions, the uniform $L^{\infty}$ and Lipschitz bounds on $P_t^{i}\chi_E$ and the regularizing estimate for the Laplacian under heat flow in \eqref{eq:regularizingheat}.
\end{proof}

We are now ready to prove \autoref{thm:rigidtangent}.

\begin{proof}[Proof of \autoref{thm:rigidtangent}]
Let us consider the case when $E$ has finite perimeter. The generalization to sets of locally finite perimeter can be obtained  building upon \autoref{lem:lei} and \eqref{eq:localitytangents}, arguing in a standard way.

Recall that the $BV$-version \eqref{eq:BE1} of the 1-Bakry-\'Emery contraction estimate  gives
$$
|\nabla P_t \chi_E| \meas\leq e^{-Kt} P_t^*|D\chi_E|\qquad\forall t>0.
$$
Let $h_t:X\to [0,1]$ be the density of $e^{Kt}|\nabla P_t\chi_E|\meas$ with respect to $P_t^*|D\chi_E|$. 
Then, one has 
$$
\int_X (1-P_th_t)\di |D\chi_E|=|D\chi_E|(X)-\int_X h_t dP_t^*|D\chi_E|=
|D\chi_E|(X)-e^{Kt}\int_X|\nabla P_t \chi_E| \di\meas.
$$
By lower semicontinuity, this proves that $g_t:=1-P_th_t$ converges to $0$ strongly in $L^1(X,|D\chi_E|)$.

Now, setting for simplicity of notation $\nu=|D\chi_E|$, we claim that
\begin{equation}\label{eq:goodlimit}
\lim_{t\downarrow 0}{1\over\nu(B_{R\sqrt{t}}(x))}\int_{B_{R\sqrt{t}}(x)} g_t\,d\nu=0\quad\qquad\forall R>0,
\text{ for $\nu$-a.e. $x\in X$.}
\end{equation}
Thanks to the asymptotic doubling property \eqref{eq:asydou}, it is sufficient to prove the
result $\nu$-a.e. on a Borel set $F$ with this property: for some $L>0$, for all $x\in F$ and $0<r<1/L$ one has 
$\nu(B_{5r}(x))\leq L\nu(B_r(x))$. By Vitali's theorem, it follows that the localized maximal function
$$
M|g|(x):=
\begin{cases}
\displaystyle{\sup\limits_{r\in (0,1/L)}\frac{\int_{B_r(x)}|g|\di\nu}{\nu(B_r(x))}}&\text{if $x\in F$;}\\ 0 &\text{if $x\in X\setminus F$;}
\end{cases}
$$
satisfies
$$
\nu(\{M|g|>\tau\})\leq {L \over\tau }\int |g| \di\nu\qquad\quad\forall\tau>0.
$$
Let us apply this estimate to the functions $g_t=1-P_th_t$: given $\eps>0$, for $t<t(\eps)$ one has
$\int g_t\di\nu<\eps^2$, and then $\nu(\{Mg_t>\eps\})\leq L\eps$. We obtain that
$$
\int_{B_r(x)} g_t\,\di\nu\leq\eps\nu(B_r(x))\qquad\quad\text{for $r<\frac 1{L}$, $t<t(\eps)$}
$$
for all $x\in F_\eps\subset F$, with $\mu(F\setminus F_\eps)$ smaller than $L\eps$.
In particular, on $F_\eps$ one has
$$
\limsup_{t\downarrow 0}{1\over\nu(B_{R\sqrt{t}}(x))}\int_{B_{R\sqrt{t}}(x)} g_t\,d\nu\leq\eps\quad\qquad\forall R>0.
$$
Since $\eps$ is arbitrary, we have proved that \eqref{eq:goodlimit} holds $\nu$-a.e. on $F$.

The claimed conclusion \eqref{eq:rigidblowup} will be achieved through two intermediate steps starting from \eqref{eq:goodlimit}.\\ 
First, let us observe that, for any $R,\,s,\,t>0$ and for any $x\in X$, it holds
\begin{equation}\label{eq:rescaling}
\frac{1}{\nu(B_{R\sqrt{t}}(x))}\int_{B_{R\sqrt{t}}(x)}g_{ts}\di\nu
=\frac{1}{|D^t\chi_E|\bigl(B^t_R(x)\bigr)}\int_{B^t_R(x)}P_s^t\left(1-e^{Kt}\frac{\abs{\nabla^ tP_s^t\chi_E}}
{\left(P^t_s\right)^*\abs{D^t\chi_E}}\right)\di|D^t\chi_E|,
\end{equation}
where we denoted by $P^t$, $\nabla^t$, $D^t$ and $B^t$ the heat semigroup, the minimal weak upper gradients, the total variation measure and the balls associated to the rescaled metric measure structure $(X,\sqrt{t}^{-1}\dist, \meas_x^{\sqrt{t}},x)$ and we are identifying measures absolutely continuous w.r.t. the reference one with their densities.\\

\noindent
{\bf Step 1.} We claim that, if $(Y,\varrho,\mu,y,F)\in \Tan_x(X,\dist,\meas,E)$ and $t_i\downarrow 0$ is a sequence realizing the convergence in \autoref{def:tan}, then
\begin{equation}\label{eq:liminf}
\int P_s\left(1-\frac{|\nabla P_s\chi_F|}{P_s^*|D\chi_F|}\right)\di\eta_R
\le\liminf_{i\to\infty}
\int P_s^{t_i}\left(1-e^{Kst_i}\frac{|\nabla^{t_i}P_s^{t_i}\chi_E|}{\bigl(P^{t_i}_s\bigr)^*|D^{t_i}\chi_E|}\right)\di\eta_R^i,
\end{equation} 
for $\Leb^1$-a.e. $R>0$, where 
\begin{equation*}
\begin{cases}
\displaystyle{\eta_R:=\frac{1}{\abs{D\chi_F}(B_R(y))}|D\chi_F|\res B_R(y),}&\\\\
\displaystyle{\eta^i_R:=\frac{1}{\abs{D^{t_i}\chi_E}(B^{t_i}_R(x))}|D^{t_i}\chi_E|\res B^{t_i}_R(x).}&
\end{cases}
\end{equation*}

In order to prove \eqref{eq:liminf}, we begin observing that $\eta_R^i$ weakly converges to $\eta_R$
for $\Leb^1$-a.e. $R>0$. Therefore, the validity of \eqref{eq:liminf} will follow from \autoref{lemma:lowersc} if we prove that  
\begin{equation}\label{eq:liminfsought}
P_s\left(1-\frac{|\nabla P_s\chi_F|}{\bigl(P_s\bigr)^*|D\chi_F|}\right)(w)
\le\liminf_{i\to\infty}
P_s^{t_i}\left(1-e^{Kst_i}\frac{|\nabla^{t_i}P_s^{t_i}\chi_E|}{\bigl(P^{t_i}_s\bigr)^*|D^{t_i}\chi_E|}\right)(w_i),
\end{equation}
whenever $w_i\in X_i\to w\in Y$.
Let us observe that, for any $\varphi\in\Cbs(Z)$, it holds 
\begin{equation}\label{eq:soughtlimsup}
\limsup_{i\to\infty}e^{Kst_i}\int\varphi\frac{\abs{\nabla^{t_i}P_s^{t_i}\chi_E}}{\bigl(P^{t_i}_s\bigr)^*|D^{t_i}\chi_E|}\di\meas_i\le
\int\varphi\frac{|\nabla P_s\chi_F|}{P_s^*|D\chi_F|}\di\mu.
\end{equation} 
Indeed, by \autoref{prop:techincalheat}, 
$|\nabla^{t_i}P_s^{t_i}\chi_E|\meas_i$ weakly converge to 
$|\nabla P_s\chi_F|\mu$ in duality with $\Cbs(Z)$, and the functions 
\begin{equation*}
f_i:=\frac{\varphi}{\bigl(P^{t_i}_s\bigr)^*|D^{t_i}\chi_E|}\qquad\text{and}\qquad f:=\frac{\varphi}{P_s^*|D\chi_F|}
\end{equation*}
are continuous, have uniformly bounded supports and satisfy the upper semicontinuity property
\eqref{eq:Gammalimsup} thanks to \autoref{lemma:Gammaliminfforheat} (recall that $|D^{t_i}\chi_E|$ weakly converge to $|D\chi_F|$ in duality with $\Cbs(Z)$). Hence \eqref{eq:liminfsought} and then \eqref{eq:liminf} follow from \autoref{lemma:uppersc}, taking into account also \autoref{rm:droppingbounds}.

\smallskip\noindent
{\bf Step 2.} We can now prove \eqref{eq:rigidblowup}.
If we choose $x\in X$ such that \eqref{eq:goodlimit} holds true (we proved above that $\abs{D\chi_E}$-a.e. $x\in X$ has this property), combining \eqref{eq:rescaling} with \eqref{eq:liminf}, we obtain
\begin{equation}\label{eq:0limit}
\int_{B_R(y)}P_s\left(1-\frac{\abs{\nabla P_s\chi_F}}{P_s^*|D\chi_F|}\right)\di |D\chi_F|\le 0.
\end{equation}
Observing that, by gradient contractivity on the $\RCD(0,N)$ space $(Y,\varrho,\mu)$, it holds 
\begin{equation}\label{eq:signproperty}
1-\frac{|\nabla P_s\chi_F|}{P_s^*|D\chi_F|}\ge 0\qquad \text{$\mu$-a.e. on $Y$,}
\end{equation} 
we can let $R\to\infty$ in \eqref{eq:0limit} to get
\begin{equation}\label{eq:0limitfull}
\int P_s\left(1-\frac{\abs{\nabla P_s\chi_F}}{P_s^*|D\chi_F|}\right)\di |D\chi_F|=0.
\end{equation} 
Then, using once more the sign property \eqref{eq:signproperty}, we obtain \eqref{eq:rigidblowup}.

Combining the just proved rigidity \eqref{eq:rigidblowup} with \autoref{maintheorem}, we can say that $(Y,\varrho,\mu)$ is isomorphic to $Z\times\setR$ for some $\RCD(0,N-1)$ m.m.s. $(Z,\dist_Z,\meas_Z)$. Furthermore, another consequence of \autoref{maintheorem} is that $F=\set{t>t_0}$ for some $t_0\in\setR$, where we denoted by $t$ the coordinate on the Euclidean factor of $Y$. Up to a translation we can also assume that $y=(\bar{z},0)$ for some $\bar{z}\in Z$.

%
\end{proof}

\appendix
\section{Appendix}\label{sec:Appendix}
In this appendix we prove a version of the iterated tangent theorem by Preiss (see \cite{Preiss87}). The proof is inspired by those of \cite[Theorem 3.2]{GigliMondinoRajala15} and \cite[Theorem 6.4]{AmbrosioKleinerLeDonne08}, dealing with pmGH tangents to $\RCD(K,N)$ spaces and tangents to sets of finite perimeters over Carnot groups, respectively (see also \cite{LeDonne11} for a previous result regarding pGH-tangents of metric spaces equipped with a doubling measure).

\begin{theorem}\label{th: tangent of tangent}
	Let $(X,\dist,\meas)$ be an $\RCD(K,N)$ m.m.s. and let $E\subset X$ be a set of finite perimeter. 
	Then for $|D\chi_E|$-a.e. $x\in X$ the following property holds true:
	for every $(Y,\varrho,\mu,y, F)\in \Tan_x(X,\dist,\meas,E)$ one has
	\begin{equation*}
		\Tan_{y'}(Y,\varrho,\mu,F)\subset \Tan_x(X,\dist,\meas,E)\qquad\text{for every $y'\in \supp |D\chi_F|$.}
	\end{equation*}
\end{theorem}

Thanks to \autoref{cor:tangents} we need only to prove the result at $|D\chi_E|$-a.e. $x\in X$ for all $(Y,\varrho,\mu,y, F)\in \Tan^*_x(X,\dist,\meas,E)$, where 
$\Tan^*_x(X,\dist,\meas,E)$ is defined adding to the conditions in \autoref{def:tan} the condition (b) of \autoref{cor:tangents}, namely that the perimeter
measures of the rescaled spaces weakly converge, in the duality with $\Cbs(X)$, to the perimeter measure of $F$.

Let us briefly recall the notion of outer measure and its main properties.
Given a positive measure $\mu$ over a metric space $(X,\dist)$ we set
\begin{equation}\label{eq: outer measure}
	\mu^*(A):=\inf\set{\mu(B):\ B \ \text{Borel}, \,\,A\subset B },
	\qquad
	\forall A\subset X.
\end{equation}
It is immediate to see that $\mu^*$ is countably sub-additive.
Let us remark that if $\mu$ is asymptotically doubling then
\begin{equation}\label{eq:density outer measure}
	\lim_{r\downarrow 0} \frac{\mu^*(A\cap B_r(x))}{\mu(B_r(x))}=1
	\qquad \text{for}\ \mu^*\text{-a.e.}\ x\in A.
\end{equation}
Indeed, we can find a set $B\in\Borel(X)$ containing $A$ such that $\mu(B)=\mu^*(A)$, so that
$\mu^*(C\cap A)=\mu(C\cap B)$ for every $C\in\Borel(X)$. In particular, taking $C=B_r(x)$, we have
\begin{equation*}
	\lim_{r\downarrow 0} \frac{\mu^*(A\cap B_r(x))}{\mu(B_r(x))}=
	\lim_{r\downarrow 0} \frac{\mu(B\cap B_r(x))}{\mu(B_r(x))}=1,
\end{equation*}
for every $x\in B$ of density $1$ for the measure $\mu$. Since $\mu$ is asymptotically doubling, $\mu$-a.e $x\in B$ has this property and \eqref{eq:density outer measure} follows.

Let us start with a technical lemma.

\begin{lemma}\label{lemma: appendixA}
	Let $(X,\dist,\meas)$ and let $E\subset X$ be as in the assumptions of \autoref{th: tangent of tangent}. 
	Let $A\subset X$ and $x\in A$ be such that
	\begin{equation*}
		\lim_{r\downarrow 0} \frac{|D\chi_E|^*(A\cap B_r(x))}{|D\chi_E|(B_r(x))}=1,
	\end{equation*}
	where $|D\chi_E|^*$ is the outer measure associated to $|D\chi_E|$ according to \eqref{eq: outer measure}.
	Assume that $(Y,\varrho,\mu, F)\in \Tan^*_x(X,\dist,\meas,E)$ and consider
	\begin{align*}
	&\Psi_i : (X, r_i^{-1}\dist)\to (Z,\dist_Z)
	\qquad
	\forall i\in \setN,\\
	&\Psi: (Y, \dist_Y)\to (Z,\dist_Z),
	\end{align*}
	a family of isometries realizing the pmGH convergence as in \autoref{def: mpGH convergence}.
	Then, for any $y'\in \supp |D\chi_F|$, there exists a sequence $(x_i)\subset A$ such that 
	\begin{equation*}
	\lim_{i\to \infty}	\dist_Z(\Psi_i(x_i),\Psi(y'))=0.
	\end{equation*}
\end{lemma}

Roughly speaking, \autoref{lemma: appendixA} tells us that it is possible to approximate every point in the support of any tangent 
by means of points in $A$, whenever $A$ is ``large'' in a measure-theoretic sense.

\begin{proof}[Proof of \autoref{lemma: appendixA}]
	As a first step we show the existence of an auxiliary sequence $(x_i)\subset X$, satisfying $\lim_i\dist_Z(\Psi_i(x_i), \Psi(y'))=0$ and 
	\begin{equation}\label{z13}
		\lim_{i\to \infty} \frac{r_i |D \chi_E|(B_{rr_i}(x_i))}{C(x,r_i)}= |D\chi_F|(B_r(y')),
		\qquad
		\text{for $\Leb^1$-a.e. $r>0$},
	\end{equation}
	where $C(x,r_i)$ was introduced in \eqref{eq:defC(x,r)}.

    Let us set $X_i:=\Psi_i(X)$, $E_i:=\Psi_i(E)$ and, with a slight abuse of notation, identity $F$ to $\Psi(F)$ and $y'$ to $\Psi(y')$.
	Since by assumption it holds that $|D\chi_{E_i}|\weakto |D\chi_F|$, we have
	\begin{equation*}
		\lim_{i\to\infty} |D\chi_{E_i}|(B^Z_r(y'))=|D\chi_F|(B^Z_r(y')),
		\qquad
		\text{for $\Leb^1$-a.e. $r>0$}.
	\end{equation*}
	This implies that the distance of $y'$ from $X_i$ is infinitesimal as $i\to\infty$, 
	hence we can find points $z_i\in X_i$ converging to $y'$ in $Z$ satisfying
	\begin{equation*}
	   \lim_{i\to\infty} |D \chi_{E_i}|(B^Z_r(z_i))= |D\chi_F|(B^Z_r(y')),
	   \qquad
	   \text{for $\Leb^1$-a.e. $r>0$}.
	\end{equation*}
	Let us set $x_i:=\Psi_i^{-1}(z_i)$. Observe that $|D\chi_F|(B^Z_r(y'))=|D\chi_F|(B^Y_r(y'))$ and
	\begin{equation*}
		|D\chi_{E_i}|(B^Z_r(z_i))=\frac{r_i |D \chi_E|(B_{rr_i}(x_i))}{C(x,r_i)},
	\end{equation*}
    so that we get \eqref{z13}.
	
	Let us now argue by contradiction.
	Assuming the conclusion of the lemma to be false we might find $\eps>0$ such that
	the limit in \eqref{z13} holds with $r=\eps$ and
	\begin{equation*}
		B_{\eps r_i}(x_i)\cap A=\emptyset
		\qquad \text{for}\ i\ \text{sufficiently large},
	\end{equation*}
	with $x_i$ and $r_i$ as in \eqref{z13}.
	Let $M>0$ be large enough to grant that 
	\begin{equation}\label{eq:prop1M}
	B_{\eps r_i}(x_i)\subset B_{Mr_i}(x)
	\end{equation} 
	(it is simple to see that such a constant exists, since the convergence in $Z$ of $z_i=\Psi(x_i)$ ensures
	$\dist(x,x_i)=O(r_i)$).
	Arguing as in the first part of the proof it is possible to see that
		\begin{equation}\label{eq:prop2M}
			\lim_{i\to \infty}\frac{r_i |D \chi_E|(B_{Mr_i}(x))}{C(x,r_i)}
			=|D\chi_F|(B_M(y'))
			\qquad
			\text{for}\ \Leb^1\text{-a.e.}\ M>0
		\end{equation}
	and from now on we assume, possibly increasing $M$, that both \eqref{eq:prop1M} and \eqref{eq:prop2M} hold true.
	Then, in view of \eqref{eq:prop1M}, we have
	\begin{equation*}
		\frac{|D\chi_E|^*(A\cap B_{Mr_i}(x))}{|D\chi_E|(B_{Mr_i}(x))} =
		\frac{|D\chi_E|^*(A\cap (B_{Mr_i}(x)\setminus B_{\eps r_i}(x_i)))}{|D\chi_E|(B_{Mr_i}(x))}
		\leq 1-\frac{|D\chi_E|(B_{\eps r_i}(x_i))}{|D\chi_E|(B_{Mr_i}(x))}.
	\end{equation*}
 Observe that the left hand side converges to $1$ as $i\to \infty$, since $x$ is of density $1$ for $A$. 
 Therefore, to get the sought contradiction, it suffices to show that
	\begin{equation*}
		\liminf_{i\to \infty}\frac{|D\chi_E|(B_{\eps r_i}(x_i))}{|D\chi_E|(B_{Mr_i}(x))}>0.
	\end{equation*}
	Using \eqref{z13} and \eqref{eq:prop2M}, we get
	\begin{equation*}
		\liminf_{i\to \infty}\frac{|D\chi_E|(B_{\eps r_i}(x_i))}{|D\chi_E|(B_{Mr_i}(x))}= 
		\frac{\lim_i\frac{r_i |D \chi_E|(B_{\eps r_i}(x_i))}{C(x,r_i)}}{\lim_i\frac{r_i |D \chi_E|(B_{Mr_i}(x))}{C(x,r_i)}}
		\geq \frac{|D\chi_F|(B_{\eps}(y'))}{\abs{D\chi_F}(B_M(y'))}>0,
	\end{equation*}
	where the last inequality holds true since we are assuming that $y'\in \supp |D\chi_F|$.
\end{proof}

Before passing to the proof of \autoref{th: tangent of tangent} we need to introduce a definition and a lemma.

\begin{definition}\label{remark: new notion of convergence}
	We shall denote by $\mathcal{F}(K,N)$ the set of equivalence classes of quintuples $\mathfrak{X}=(X,\dist,\meas,x, \nu)$ where $(X,\dist,\meas,x)$ is a pointed $\RCD(K,N)$ m.m.s and $\nu$ is a nonnegative and locally finite Borel measure with $\supp \nu\subset \supp \meas$,
	modulo the equivalence relation $\sim$ defined as follows. We say that $(X_1,\dist_1, \meas_1, x_1, \nu_1)\sim (X_2,\dist_2,\meas_2,x_2,\nu_2)$ if there exists an isometry $T:(\supp\meas_1,\dist_1)\to (\supp\meas_2,\dist_2)$ such that $T_{\sharp} \meas_1=\meas_2$, $T(x_1)=x_2$ and $T_{\sharp}\nu_1=\nu_2$.
	We shall denote by $\mathcal{F}$ the union of the sets $\mathcal{F}(K,N)$ for $K\in \setR$, $1\le N<+\infty$. Observe that $\mathcal{F}$ can be realized as a countable union of sets $\mathcal{F}(K,N)$.
	
	Let us introduce a distance in $\mathcal{F}$.
	Fix $\mathfrak{X}_1=(X_1,\dist_1,\meas_1,x_1,\nu_1)$, $\mathfrak{X}_2=(X_2,\dist_2,\meas_2,x_2,\nu_2)$ in $\mathcal F$,
	a proper metric measure space $(Z,\dist_Z)$ and isometric embeddings $\Psi_i:(X_i,\dist_i)\to (Z,\dist_Z)$, $i=1,\,2$. For any 
	integer $n\geq 1$ we define 
	\begin{align*}
	\mathcal{D}_{n,\Psi_1,\Psi_2}&(\mathfrak{X}_1,\mathfrak{X}_2):=\\ &
	\dist_{H}(\Psi_1(X_1\cap \overline{B}(x_1,n)),\Psi_2(X_2\cap \overline{B}(x_2,n)))\wedge 1\\
	&+\abs{\log\left(\frac{\meas_1(B(x_1,n))}{\meas_2( B(x_2,n))}\right)}\wedge 1 +W_1^Z\left( (\Psi_1)_{\sharp}\frac{\chi_{B(x_1,n)}}{\meas_1(B(x_1,n))}\meas_1,   (\Psi_2)_{\sharp}\frac{\chi_{B(x_2,n)}}{\meas_2(B(x_2,n))}\meas_2\right)\\
	&+\abs{\log\left(\frac{\nu_1(B(x_1,n))}{\nu_2(B(x_2,n))}\right)}\wedge 1
	+W_1^Z\left( (\Psi_1)_{\sharp} \frac{\chi_{B(x_1,n)}}{\nu_1(B(x_1,n))}\nu_1, (\Psi_2)_{\sharp} \frac{\chi_{B(x_2,n)}}{\nu_2(B(x_2,n))}\nu_2 \right),
	\end{align*}
	where $\dist_{H}$ is the Hausdorff distance between compact subsets of $Z$ and $W_1^Z$ is the $1$-Wasserstein distance in 
	$(Z,\dist_Z\wedge 1)$, namely
	\begin{equation}
	W_1^Z(\mu,\nu):=\inf\left\lbrace \int_Z \dist_Z(x,y)\wedge 1 \di \pi(x,y):\ \pi\in\Gamma(\mu,\nu)
	\right\rbrace,
	\end{equation}
	with $\Gamma(\mu,\nu)\subset\Prob(X\times X)$ the set of probability measures having $\mu$ and $\nu$ as
	marginals.
	We finally define
	\begin{equation}\label{eq:defDDD}
	\mathcal{D}(\mathfrak{X}_1,\mathfrak{X}_2):=\inf_{\Psi_1,\Psi_2}\left\lbrace \dist_Z(\Psi_1(x_1),\Psi_2(x_2))+\sum_{n=1}^{\infty}\frac{1}{2^n}
	\mathcal{D}_{n,\Psi_1,\Psi_2}(\mathfrak{X}_1,\mathfrak{X}_2)\right\rbrace ,
	\end{equation}
	the infimum being taken among all possible proper metric spaces $(Z,\dist_Z)$ and all 
	isometric embeddings $\Psi_i:(X_i,\dist_i)\to (Z,\dist_Z)$ for $i=1,\,2$.
\end{definition}

\begin{lemma}\label{lemma:new convergence of mms}
	$\mathcal{D}$ is a distance over $\mathcal{F}$ and a sequence $(X_i,\dist_i, \meas_i, x_i, \nu_i)\subset \mathcal{F}$ converges to $(Y,\varrho,\mu,y,\nu)$ in the topology induced by $\mathcal{D}$ if and only if $(X_i,\dist_i, \meas_i, x_i)\to (Y,\varrho,\mu,y)$ in the pmGH topology and $\nu_i\weakto \nu$ in duality with $\Cbs(Z)$, where $(Z,\dist_Z)$ is a metric space where the pmGH convergence is realized.
	Moreover the subspace 
	\begin{equation}\label{eq:barF}
	\mathcal{\overline{F}}:=\set{(X,\dist,\meas,x,\nu)\in \mathcal{F}:\ \nu=h\meas,\  \text{with}\ h\in L^{\infty}(X,\meas)}	
	\end{equation}
	is separable.	
\end{lemma}
\begin{proof}
	The verification that $\mathcal{D}$ is a distance is quite standard, see for instance \cite{GigliMondinoSavare15} .
	The equivalence between the two notions of convergence can be proved following the same strategy in the proof of \cite[Theorem 3.15]{GigliMondinoSavare15}, the only difference here being the addition to the quadruple of the measure $\nu$.
	Let us prove that $\mathcal{\overline{F}}$ is separable.	It is enough to prove that, given $K$ and $N$, for any $k>0$ the set
	\begin{equation}
	\mathcal{\overline{F}}_k(K,N):=\set{(X,\dist,\meas,x,\nu)\in \mathcal{F}(K,N):\ \nu=h\meas,\  \text{with}\ \norm{h}_{L^{\infty}(X,\meas)}\le k}	
	\end{equation}
	is compact. Let us fix a sequence $(X_i,\dist_i,\meas_i,x_i,\nu_i)\subset\mathcal{\overline{F}}_k(K,N)$. We can assume, up to extract a subsequence, that $(X_i,\dist_i, \meas_i, x_i)\to (Y,\varrho,\mu,y)$ in the pmGH topology. Let us fix a proper metric space $(Z,\dist_Z)$ realizing this convergence.
	Since $\nu_i\le k\meas_i$ and $\meas_i\weakto \mu$ in duality with $\Cbs(Z)$ we deduce that the measures $\nu_i$ are locally bounded in $Z$, uniformly in $i\in \setN$. Therefore, possibly extracting a subsequence, there exists a positive measure $\nu$ in $Z$ such that $\nu_i\weakto \nu$ in duality with $\Cbs(Z)$. It is immediate to check that $\nu\ll\mu$, with density uniformly bounded by $k$. This concludes the proof.
\end{proof}

We are ready to prove \autoref{th: tangent of tangent}.

\begin{proof}[Proof of \autoref{th: tangent of tangent}]
Since tangents are invariant w.r.t. rescaling and closed w.r.t. $\mathcal{D}$-convergence, 
it is enough to prove that the set of points $x\in X$ such that there exist 
$(Y,\varrho,\mu,y, F)\in \Tan^*_x(X,\dist,\meas,E)$ and $y'\in \supp |D\chi_F|$ such that 
	\begin{equation*}
	(Y,\varrho, \mu_{y'}^1, y', F )\notin \Tan_x(X,\dist,\meas,E)
	\end{equation*}
is $|D\chi_E|^*$-negligible, where $\mu_1^{y'}:=C(y',1)^{-1}\mu$ (see \autoref{def:tan}).
	
Let us fix positive integers $k,\,m$ and a closed subset $\mathcal{U}\subset \mathcal{\overline{F}}$ with diameter, measured 
w.r.t. the distance $\mathcal{D}$ in \eqref{eq:defDDD}, smaller than $(2k)^{-1}$. Since, according to \autoref{lemma:new convergence of mms}, 
$\mathcal{\overline{F}}$ is separable, it is enough to prove that
	\begin{align*}
	A_{k,m}:= \bigl\{ x\in X\ &:\ \exists\ (Y, \varrho, \mu, y, F)\in \Tan^*_x(X,\dist,\meas,E)\cap \mathcal{U}\ \text{and}\ y'\in \supp |D\chi_F|
	\text{  such that}\\
	&\mathcal{D}((Y,\varrho, \mu_{y'}^1, y', F ),(X, r^{-1}\dist,\meas^{r}_x, x, E))\geq 2k^{-1}\ \ \  \forall r\in (0, 1/m)
	\bigr\}
	\end{align*}
	is $|D\chi_E|^*$-negligible, where we identified the set $F$ with the measure $\chi_F \mu$.
	
	If, by contradiction, $|D\chi_E|^*(A_{k,m})>0$, then, since $|D\chi_E|$ is asymptotically doubling by \autoref{prop:Asymptotic minimality and doubling}, we can find $x\in A_{k,m}$ such that
	\begin{equation*}
	\lim_{r\downarrow 0} \frac{|D\chi_E|^*(A_{k,m}\cap B_r(x))}{|D\chi_E|(B_r(x))}=1,
	\end{equation*}
	see \eqref{eq:density outer measure}.
	Since $x\in A_{k,m}$ there exist $(Y, \varrho, \mu, y, F)\in \Tan_x^*(X,\dist,\meas,E)\cap\mathcal{U}$ and $y'\in \supp|D\chi_F|$ such that 
	$\mathcal{D}((Y,\varrho, \mu_{y'}^1, y', F ),(X, r^{-1}\dist,\meas^{r}_x, x, E))\geq 2k^{-1}$ for any $r\in (0, 1/m)$
	and \autoref{lemma: appendixA} grants the existence of a sequence $(x_i)\subset A_{k,m}$ such that 
	\begin{equation*}
	\lim_{i\to \infty} \dist_Z(\Psi_i(x_i),\Psi(y'))=0,
	\end{equation*} 
	where $\Psi_i$, $\Psi$ are the embedding maps of \autoref{def: mpGH convergence}.
	Then, by definition of pmGH convergence, using the space $(Z,\dist_Z)$ we deduce
	\begin{equation*}
	(X, r_i^{-1}\dist,\meas^{r_i}_x, x_i)\to (Y, \varrho, \mu, y').
	\end{equation*}
	Since $\chi_{B^Z(\bar{z},1)} (1-\dist_Z(\cdot,\bar{z}))$ belongs to $\Cb(Z)$ for every $\bar{z}\in Z$, it is immediate to check that
	\begin{equation*}
	(X, r_i^{-1}\dist,\meas^{r_i}_{x_i}, x_i)\to (Y, \varrho, \mu_{y'}^1, y'),
	\qquad
	\text{in the pmGH topology,}
	\end{equation*}
	and	$(\Psi)_{\#} \chi_E\meas_{x_i}^{r_i}\weakto \Psi_{\#} \chi_F \mu_{y'}^1$ in duality with $\Cbs(Z)$,
	that, thanks to \eqref{lemma:new convergence of mms}, is equivalent to
	\begin{equation}\label{z14}
	\mathcal{D}((X, r_i^{-1}\dist,\meas^{r_i}_{x_i}, x_i, E), (Y,\varrho, \mu_{y'}^1,y', F ))\to 0,
	\end{equation}
	see \autoref{remark: new notion of convergence}.
	Since $x_i\in A_{k,m}$ we can find $(Y_i, \varrho_i, \mu_i, y_i, F_i)\in \Tan^*_{x_i}(X,\dist,\meas,E)\cap \mathcal{U}$ and $y_i'\in \supp|D\chi_{F_i}|$ 
	such that $\mathcal{D}((Y_i,\varrho_i, (\mu_i)_{y_i'}^1, y_i', F_i ),(X, r^{-1}\dist,\meas^{r}_{x_i}, x_i, E))\geq 2k^{-1}$ for any $r\in (0,1/m)$.
	
	Using \eqref{z14} and taking into account that by construction $\text{diam}\ \mathcal{U}<(2k)^{-1}$, we find the sought contradiction
	\begin{align*}
	2 k^{-1} & \leq \mathcal{D}((Y_i,\varrho_i, (\mu_i)_{y_i'}^1, y_i', F_i ),(X, r_i^{-1}\dist,\meas^{r_i}_{x_i}, x_i, E))\\
	& \leq \mathcal{D}( (Y,\varrho, \mu_{y'}^1, y', F),(X, r_i^{-1}\dist,\meas^{r_i}_{x_i}, x_i, E))
	+\mathcal{D}((Y_i,\varrho_i, (\mu_i)_{y_i'}^1, y_i', F_i ),(Y,\varrho, \mu_{y'}^1, y', F))\\
	& \leq\mathcal{D}((Y,\varrho, \mu_{y'}^1, y', F),(X, r_i^{-1}\dist,\meas^{r_i}_{x_i}, x_i, E))+(2k)^{-1}\\
	& \leq k^{-1},
	\end{align*}
	for $i$ large enough. 
\end{proof}


\begin{thebibliography}{GMS13}

\bibitem[A97]{Ambrosio97}
\textsc{L. Ambrosio:}
\textit{Corso introduttivo alla teoria geometrica della misura ed alle superfici minime.} 
Appunti dei Corsi Tenuti da Docenti della Scuola. [Notes of
Courses Given by Teachers at the School]. Scuola Normale Superiore, Pisa,  (1997), ii+144.           

\bibitem[A01]{Am01}
\textsc{L. Ambrosio:}
\textit{Some fine properties of sets of finite perimeter in Ahlfors regular metric measure spaces,} 
Adv. Math., {\bf 159} (2001), 51--67.

\bibitem[A02]{Am02}
\textsc{L. Ambrosio:}
\textit{Fine properties of sets of finite perimeter in doubling metric measure spaces,} 
Set-Valued Anal., {\bf 10} (2002), 111--128.

\bibitem[A18]{Ambrosio18}
\textsc{L. Ambrosio:}
\textit{Calculus, heat flow and curvature-dimension bounds in metric measure spaces,}
Proceedings of the ICM 2018, (2018).

\bibitem[AF14]{AmbrosioFeng}
\textsc{L. Ambrosio, J. Feng:}
\textit{On a class of first order Hamilton-Jacobi equations in metric spaces,}
J. Differential Equations, {\bf 256} (2014) 2194--2245.

\bibitem[ADM14]{AmbDiM14}
\textsc{L. Ambrosio, S. Di Marino:}
\textit{Equivalent definitions of BV space and of total variation on metric measure spaces.}    
J. Funct. Anal., {\bf 266} (2014), 4150--4188.

\bibitem[AGMR15]{AmbrosioGigliMondinoRajala15}
\textsc{L. Ambrosio, N. Gigli, A. Mondino, T. Rajala:}
\textit{Riemannian Ricci curvature lower bounds in metric measure spaces with $\sigma$-finite measure.}
Trans. Amer. Math. Soc., {\bf 367} (2015), 4661–4701. 

\bibitem[AGS05]{AmbrosioGigliSavare05}
\textsc{L. Ambrosio, N. Gigli, G. Savar\'e:}
\textit{Gradient flows in metric spaces and in the space of
 Probability measures.} Lectures in Mathematics ETH Zurich,
Birkh\"auser, 2005.
						
\bibitem[AGS14]{AmbrosioGigliSavare14}
\textsc{L. Ambrosio, N. Gigli, G. Savar\'e:}
\textit{ Metric measure spaces with Riemannian Ricci curvature bounded from below}.
Duke Math. J., {\bf 163} (2014), 1405--1490.   

\bibitem[AGS15]{AmbrosioGigliSavare15}
\textsc{L. Ambrosio, N. Gigli, G. Savar\'e:}
\textit{Bakry-\'Emery curvature-dimension condition and Riemannian Ricci curvature bounds.}
                Ann. Probab., {\bf 43} (2015), 339--404. 
				
\bibitem[AH17]{AmbrosioHonda}
\textsc{L. Ambrosio, S. Honda}:
\textit{New stability results for sequences of metric measure spaces with uniform Ricci bounds from below}.
Measure Theory in Non-Smooth Spaces, De Gruyter Open, Warsaw, (2017), 1--51.		

\bibitem[AHT18]{AmbrosioHondaTewodrose17}
\textsc{L. Ambrosio, S. Honda, D. Tewodrose:}
\textit{Short-time behavior of the heat kernel and Weyl's law on $\RCD^*(K,N)$-spaces}.
Ann. Global Anal. Geom., {\bf 53} (2018), 97--119.
		
\bibitem[AKL08]{AmbrosioKleinerLeDonne08}	
\textsc{L. Ambrosio, B. Kleiner, E. Le Donne:}
\textit{Rectifiability of sets of finite perimeter in Carnot groups: existence of a tangent hyperplane}.
J. Geom. Anal., {\bf 19} (2009), 509–540.
				
\bibitem[AMS15]{AmbrosioMondinoSavare15}
\textsc{L. Ambrosio, A. Mondino, G. Savar\'e}:
\textit{Nonlinear diffusion equations and curvature conditions in metric measure spaces}.
ArXiv preprint: 1509.07273 (2015), to appear in Memoirs Amer. Math. Soc. 		
			
\bibitem[AT14]{AmbrosioTrevisan14}
\textsc{L. Ambrosio, D. Trevisan:}
\textit{Well posedness of Lagrangian flows and continuity equations in metric measure spaces.}
Anal. PDE, {\bf 7} (2014), 1179--1234.
								
\bibitem[BS18]{BrueSemola18}
\textsc{E. Bru\'e, D. Semola:}
\textit{Constancy of the dimension for $RCD(K,N)$ spaces via regularity of Lagrangian flows.}
To appear in Comm. Pure Appl. Math., ArXiv:1804.07128. 			
				
\bibitem[C99]{Cheeger99}
\textsc{J. Cheeger:}
\textit{Differentiability of Lipschitz functions on metric measure spaces.}
Geom. Funct. Anal., {\bf 9} (1999), 428–517.			
			
			
		
		
			
\bibitem[CM16]{Cavalletti-Milman}
\textsc{F. Cavalletti, E. Milman:}
\textit{The Globalization Theorem for the Curvature Dimension Condition,}
preprint on arXiv: 1612.07623.			
			
					
\bibitem[DG54]{DeGiorgi54}
\textsc{E. De Giorgi:}
\textit{Su una teoria generale della misura $(r-1)$-dimensionale in uno spazio ad $r$ dimensioni,}
Ann. Mat. Pura Appl. (4) {\bf 36}, (1954). 191–213.

\bibitem[DG55]{DeGiorgi55}
\textsc{E. De Giorgi:}
\textit{Nuovi teoremi relativi alle misure $(r-1)$-dimensionali in uno spazio ad $r$ dimensioni,}
Ricerche Mat. 4 (1955), 95–113. 
			
\bibitem[DPG16]{DePhilippisGigli16}
\textsc{G. De Philippis, N. Gigli:}
\textit{From volume cone to metric cone in the nonsmooth setting.}
Geom. Funct. Anal., {\bf 26} (2016), 1526--1587.			
			
\bibitem[DPG18]{DePhilippisGigli18}
\textsc{G. De Philippis, N. Gigli:}
\textit{ Non-collapsed spaces with Ricci curvature bounded from below}.
J. Éc. polytech. Math., {\bf 5} (2018), 613–650.			

\bibitem[DPMR17]{DePhlippisMarcheseRindler17}
\textsc{G. De Philippis, A. Marchese, F. Rindler:}
\textit{On a conjecture of Cheeger,} 
Measure theory in non-smooth spaces, 145–155,
Partial Differ. Equ. Meas. Theory, De Gruyter Open, Warsaw, 2017. 

			
\bibitem[EKS15]{ErbarKuwadaSturm15}
\textsc{M. Erbar, K. Kuwada, K.-T. Sturm:}
\textit{On the equivalence of the entropic curvature-dimension condition and Bochner’s inequality on metric measure spaces.}
Invent. Math., {\bf 201} (2015), 993--1071.		


\bibitem[EGLS18]{Shanmugalingam18}
\textsc{E.-B. Sylvester, T.-G. James, P. Lahti, N. Shanmugalingam:}
\textit{Asymptotic behavior of $BV$ functions and sets of finite perimeter in metric measure spaces,} 
Preprint arXiv:1810.05310.

\bibitem[FF60]{FedererFleming60}
\textsc{H. Federer, H.-W. Fleming:}
\textit{Normal and integral currents,} 
Ann. of Math. (2) {\bf 72} (1960) 458–520. 

\bibitem[F66]{Fleming}
\textsc{H.-W. Fleming}:
\textit{Flat chains over a finite coefficient group.}
Trans. Amer. Math. Soc., {\bf 121} (1966), 160--186.
			
\bibitem[G13]{Gigli13}
\textsc{N. Gigli:}
\textit{The splitting theorem in non-smooth context}. Preprint arXiv:1302.5555.
		
\bibitem[G14]{Gigli14}
\textsc{N. Gigli:}
\textit{An overview of the proof of the splitting theorem in spaces with non-negative Ricci curvature.}
Anal. Geom. Metr. Spaces, {\bf 2} (2014), 169–213.

\bibitem[G15]{Gigli15}
\textsc{N. Gigli:}
\textit{On the differential structure of metric measure spaces and applications.}
Mem. Amer. Math. Soc., {\bf 236} (2015), vi--91.

\bibitem[G18]{Gigli18}
\textsc{N. Gigli:}
\textit{Nonsmooth differential geometry: an approach tailored for spaces with Ricci curvature bounded from below.}
Mem. Amer. Math. Soc., {\bf 251} (2018), v--161.

\bibitem[GH16]{GigliHan16}
\textsc{N. Gigli, B.-X. Han:}
\textit{Independence on $p$ of weak upper gradients on $RCD$ spaces}.
J. Funct. Anal., {\bf 271} (2016), 1–11.

\bibitem[GMR15]{GigliMondinoRajala15}
\textsc{N. Gigli, A. Mondino, T. Rajala:}
\textit{Euclidean spaces as weak tangents of infinitesimally Hilbertian metric measure spaces with Ricci curvature
bounded below.} J. Reine Angew. Math., {\bf 705} (2015), 233--244.
	   	   
\bibitem[GMS15]{GigliMondinoSavare15}	   
\textsc{N. Gigli, A. Mondino, G. Savar\'e:}
\textit{Convergence of pointed non-compact metric measure spaces and stability of Ricci curvature bounds and heat flows.}	   
Proc. Lond. Math. Soc. (3), {\bf 111} (2015), 1071–1129.
	   
\bibitem[GP16a]{GigliPasqualetto16a}
\textsc{N. Gigli, E. Pasqualetto:}
\textit{Behaviour of the reference measure on $\RCD$ spaces under charts,} 
preprint, arXiv:1607.05188.	   
	   	    
	   	    
\bibitem[GP16b]{GigliPasqualetto16b}
\textsc{N. Gigli, E. Pasqualetto:}
\textit{Equivalence of two different notions of tangent bundle on rectifiable metric measure spaces,}
Preprint arXiv:1611.09645. 	   	    
	   	    
\bibitem[JLZ14]{JangLiZhang}
\textsc{R. Jiang, H. Li, H. Zhang:}
\textit{Heat Kernel Bounds on Metric Measure Spaces and Some Applications.}
Potential Anal., {\bf 44} (2016), 601--627.		    

\bibitem[KM18]{KellMondino18}
\textsc{M. Kell, A. Mondino:}
\textit{On the volume measure of non-smooth spaces with Ricci curvature bounded below,}
Ann. Sc. Norm. Super. Pisa Cl. Sci. (5) {\bf 18} (2018), no. 2, 593–610. 
	    
\bibitem[LD11]{LeDonne11}
\textsc{E. Le Donne:}
\textit{Metric spaces with unique tangents.}
Ann. Acad. Sci. Fenn. Math., {\bf 36} (2011), 683--694.	    

\bibitem[LV09]{LottVillani}
\textsc{J. Lott, C. Villani:}
\textit{Ricci curvature for metric-measure spaces via optimal transport.}
Ann. of Math. (2), {\bf 169} (2009), 903--991.

\bibitem[M15]{Magnani15}
\textsc{V. Magnani:}
\textit{On a measure theoretic area formula.}
Proc. Roy. Soc. Edinburgh Sect. A, {\bf 145} (2015), 885--891.

\bibitem[Me01]{Menguy}
\textsc{X. Menguy:}
\textit{Examples of strictly weakly regular points.}
Geom. Funct. Anal., {\bf 11} (2001), 124--131.

\bibitem[Mi03]{MirandaJr}
\textsc{M. Miranda Jr.:}
\textit {Functions of bounded variation on ``good'' metric spaces.} 
J. Math. Pures Appl., {\bf 82} (2003), 975--1004.

\bibitem[MN14]{MondinoNaber14}
\textsc{A. Mondino, A. Naber:}
\textit{Structure theory of metric measure spaces with lower Ricci curvature bounds}.
ArXiv preprint 1405.2222 (2014), to appear on J. Eur. Math Soc.
    
\bibitem[N14]{Nakayasu14}
\textsc{A. Nakayasu:}
\textit{Metric viscosity solutions for Hamilton-Jacobi equations of evolution type}.
Adv. Math. Sci. Appl., {\bf 24} (2014), 333--351.

\bibitem[P11]{Petrunin}
\textsc{A. Petrunin:}
\textit{Alexandrov meets Lott-Villani-Sturm.} 
M\"unster J. Math., {\bf 4} (2011), 53--64.

\bibitem[P87]{Preiss87}
\textsc{D. Preiss:}
\textit{Geometry of measures in $\setR^n$: distributions, rectifiability and densities.}
Ann. of Math., {\bf 125} (1987), 537--643.

\bibitem[S96]{Sturm96}
\textsc{K.-T. Sturm:}
\textit{Analysis on local Dirichlet spaces. III. The parabolic Harnack inequality.}
J. Math. Pures Appl. (9), {\bf 75} (1996), 273--297.

\bibitem[S06a]{Sturm06a}
\textsc{K.-T. Sturm:}
\textit{On the geometry of metric measure spaces I.}
Acta Math., {\bf 196} (2006), 65--131.
		
\bibitem[S06b]{Sturm06b}
\textsc{K.-T. Sturm:}
\textit{On the geometry of metric measure spaces II.}
Acta Math., {\bf 196} (2006), 133--177.

\bibitem[T15]{Tolsa1}
\textsc{X. Tolsa:}
\textit{Characterization of {$n$}-rectifiability in terms of {J}ones' square function: part I.}
Calc. Var. \& PDE, {\bf 54} (2015), 3643--3665.
 
\bibitem[AT15]{Tolsa2}
\textsc{J. Azzam, X. Tolsa:}
\textit{Characterization of {$n$}-rectifiability in terms of {J}ones' square function: Part II.}
Geometric and Functional Analysis, {\bf 25} (2015), 1371--1412.

\bibitem[V09]{Villani09}
\textsc{C. Villani:}
\textit{Optimal transport. Old and New.}
Grundlehren der Mathematischen Wissenschaften, {\bf 338}. Springer-Verlag Berlin, 2009.

\bibitem[VR08]{VonRenesse08}
\textsc{M.-K. Von Renesse:}
\textit{On local Poincaré via transportation.}
Math. Z., {\bf 259} (2008), 21--31.	

\bibitem[We00]{Weaver}
\textsc{N. Weaver:}
\textit{Lipschitz algebras and derivations. II. Exterior differentiation}.
J. Funct. Anal., {\bf 178} (2000), 64--112.

\bibitem[W89]{White}
\textsc{B. White:}
\textit{A new proof of the compactness theorem for integral currents.}
Comm. Math. Helvetici, {\bf 64} (1989), 207--220.
			
\end{thebibliography}
\end{document}